\documentclass{amsart}
\usepackage{amsmath,amssymb,tikz,xcolor}
\usetikzlibrary{arrows,shapes,positioning,decorations.markings}
\usetikzlibrary{shapes.geometric}

\usepackage[margin=1.5cm]{geometry}
\newtheorem{theorem}{Theorem}[section]
\newtheorem{lemma}[theorem]{Lemma}
\newtheorem{corollary}[theorem]{Corollary}
\newtheorem{conjecture}[theorem]{Conjecture}
\newtheorem{proposition}[theorem]{Proposition}

\newtheorem{remark}[theorem]{Remark}

\newtheorem{notation}[theorem]{Notation}
\newtheorem{definition}[theorem]{Definition}

\renewcommand{\wr}{\mathop{\mathrm{wr}}}

\newcommand{\Alt}{\mathop{\mathrm{Alt}}}
\newcommand{\Sym}{\mathop{\mathrm{Sym}}}
\newcommand{\Cay}{\mathop{\Gamma}}
\newcommand{\Aut}{\mathop{\mathrm{Aut}}}

\allowdisplaybreaks

\numberwithin{equation}{section}
\def\Z#1{{\bf Z}(#1)}
\def\cent#1#2{{\bf C}_{#1}(#2)}
\def\norm#1#2{{\bf N}_{#1}(#2)}
\def\Dic{{\rm Dic}}

\makeatletter
\newcommand{\myitem}[1]{%
\item[#1]\protected@edef\@currentlabel{#1}%
}
\makeatother

\begin{document}

\title{On the asymptotic enumeration of Cayley graphs} 

\author[J. Morris]{Joy Morris}
\address{Department of Mathematics and Computer Science,
University of Lethbridge,\newline Lethbridge, AB. T1K 3M4. Canada.}
\email{joy.morris@uleth.ca}
\author[M. Moscatiello]{Mariapia Moscatiello}
\address{Mariapia Moscatiello, Dipartimento di Matematica \lq\lq Tullio Levi-Civita\rq\rq,\newline
 University of Padova, Via Trieste 53, 35121 Padova, Italy} 
\email{mariapia.moscatiello@math.unipd.it}

\author[P. Spiga]{Pablo Spiga}
\address{Pablo Spiga,
Dipartimento di Matematica e Applicazioni, University of Milano-Bicocca,\newline
Via Cozzi 55, 20125 Milano, Italy}\email{pablo.spiga@unimib.it}

\thanks{Address correspondence to P. Spiga, E-mail: pablo.spiga@unimib.it.}

\begin{abstract}
In this paper we are interested in the asymptotic enumeration of Cayley graphs. It has previously been shown that almost every Cayley digraph has the smallest possible automorphism group: that is, it is a digraphical regular representation (DRR). In this paper, we approach the corresponding question for undirected Cayley graphs. The situation is complicated by the fact that there are two infinite families of groups that do not admit any graphical regular representation (GRR). 

The strategy for digraphs involved analysing separately the cases where the regular group $R$ has a nontrivial proper normal subgroup $N$ with the property that the automorphism group of the digraph fixes each $N$-coset setwise, and the cases where it does not. In this paper, we deal with undirected graphs in the case where the regular group has such a nontrivial proper normal subgroup.

\smallskip

\begin{center}\textit{In memory of Carlo Casolo: a dear good friend}\end{center}
\end{abstract}

\maketitle
\section{Introduction}\label{section:introduction}

We consider only finite groups and graphs in this paper. A graph $\Gamma$ is an ordered pair $(V,E)$ with $V$ a finite non-empty set of vertices, and $E$ a set of unordered pairs from $V$, representing the edges. An automorphism of a graph is a permutation on $V$ that preserves the set $E$.

\begin{definition}
Let $R$ be a group and $S=S^{-1}$ an inverse-closed subset of $R$. The \emph{Cayley graph} $\Cay(R,S)$ is the graph with $V=R$ and $\{r,t\} \in E$ if and only if $tr^{-1} \in S$. 
\end{definition}

The problem of finding graphical regular representations (GRRs) for groups has a long history. Mathematicians have studied graphs with specified automorphism groups at least as far back as the 1930s, and in the 1970s there were many papers devoted to the topic of finding GRRs (see for example \cite{babai11,Het,Im1, Im2,Im3,NW1,NW2,NW3,Wat}), although the ``GRR" terminology was coined somewhat later.

\begin{definition}
A \emph{graphical regular representation} (GRR) for a group $R$ is a graph whose full automorphism group is the group $R$ acting regularly on the vertices of the graph.
\end{definition}

It is an easy observation that when $\Cay(R,S)$ is a Cayley graph, the group $R$ acts regularly on the vertices as a group of graph automorphisms. A GRR for $R$ is therefore a Cayley graph on $R$ that admits no other automorphisms.

The main thrust of much of the work through the 1970s was to determine which groups admit GRRs. This question was ultimately answered by Godsil in \cite{God}.

\begin{theorem}[Godsil, \cite{God}]
A group has a graphical regular representation if and only if it is not one of:
\begin{itemize}
\item a generalised dicyclic group (see Definition~\ref{defeq:2});
\item an abelian group of exponent greater than $2$; or
\item one of $13$ small groups (of order at most 32).
\end{itemize}
\end{theorem}

A corresponding result for DRRs by Babai was much simpler, requiring no excluded families and finding only $5$ exceptional small groups.

Babai and Godsil made the following conjecture.

\begin{conjecture}[\cite{BaGo}; Conjecture 3.13, \cite{Go2}]
If $R$ is not generalised dicyclic or abelian of exponent greater than $2$, then for almost all inverse-closed subsets $S$ of $R$, $\Cay(R,S)$ is a GRR.
\end{conjecture}

The details of this conjecture are somewhat imprecise; we are interested in the following more specific formulation:
$$\lim_{r \to \infty} \min\left\{ \frac{|\{S \subseteq R: \Aut(\Cay(R,S))=R\}|}{2^r}: R\text{ admits a GRR and }|R|=r\right\} =1.$$
From Godsil's theorem, as $r\to \infty$, the condition ``$R$ admits a GRR" is equivalent to ``$R$ is neither a generalised dicyclic group, nor abelian of exponent greater than $2$."

The corresponding result for Cayley digraphs (which does not  require any families of groups to be excluded) was proved by the first and third authors in~\cite{MSMS}. 

The strategy used in \cite{MSMS} (which was based on previous work in \cite{BaGo} by Babai and Godsil) to prove that almost every Cayley digraph is a DRR, involved three major pieces. One piece was to show that there are not many Cayley digraphs admitting digraph automorphisms that are also group automorphisms. A second piece of the proof involved considering the possibility that the group $R$ has a proper nontrivial normal subgroup $N$, and there is a digraph automorphism that fixes every orbit of $N$ setwise. This piece itself naturally divides into two parts. If $|N|$ is relatively small in comparison with $|R|$, then showing that roughly $2^{|R|/|N|}$ digraphs do not admit a particular type of automorphism is significant, while if $|N|$ is relatively large (for example if $|N|=|R|/c$ for some constant $c$) this sort of bound is not useful for our purposes. Conversely, if $|N|$ is relatively large then showing that roughly $2^{|N|}$ digraphs do not admit a particular type of automorphism is significant, but such a bound is not useful if $|N|$ is relatively small. So we need to combine bounds of each type to come up with an overall bound. The third and final piece of the proof involved considering the possible existence of digraph automorphisms that do not fix all orbits of any normal subgroup $N$ of $R$.

While the second piece may not seem entirely natural, it is important to consider because it covers a possibility that does not readily succumb to induction. If a graph only admits automorphisms that fix every orbit of $N$ setwise, then the quotient graph on the orbits of $N$ may be in fact a GRR. The induced subgraph on a single orbit may very well also be a GRR, so an inductive argument will reduce a non-GRR to two smaller GRRs, making induction virtually impossible to use effectively.  

Similarly to the results about existence of GRRs and DRRs, the requirement that a connection set for a graph must be inverse-closed creates complications that make the proof of the Babai-Godsil conjecture more difficult for graphs than for digraphs. Rather than trying to accomplish the full result in a single paper, it makes sense to divide the work into the main pieces that were used to prove the DRR result, and attempt to show each of these pieces for GRRs. 

The first piece, showing that there are not many Cayley graphs admitting graph automorphisms that are also group automorphisms (unless the group is generalised dicyclic or abelian of exponent greater than $2$) was accomplished by the third author in~\cite{spiga11}. Some of the main results from that work are also used in this paper, and we have included them as Theorem~\ref{l:aut} and Proposition~\ref{propo:aut}.

The goal of this paper is to complete the second piece of the proof: that is, to show that the number of Cayley graphs on $R$ that admit nontrivial graph automorphisms that fix the vertex $1$ and normalise some proper nontrivial normal subgroup $N$ of $R$, is vanishingly small as a proportion of all Cayley graphs on $R$. 

As in the work on DRRs, this problem naturally divides into the cases where the normal subgroup $N$ is ``large" or ``small" relative to $|R|$. Our main results are Theorem~\ref{main1} and Theorem~\ref{main2}, which we prove in Sections~\ref{sec:Nlarge} and~\ref{sec:Nsmall}, respectively. In the case of graphs, it emerges that we also need to consider separately graph automorphisms that fix or invert every element of the group. We deal with these in Section~\ref{sec:inversion}, and this piece of our work applies whether or not $R$ admits any proper nontrivial normal subgroup.

 Given a finite group $R$, we let $2^{\mathbf{c}(R)}$ denote the number of inverse-closed subsets of $R$. (The value $\mathbf c(R)$ is defined explicitly in Definition~\ref{defeq:1}.)
\begin{theorem}\label{main1}
Let $R$ be a finite group and let $N$ be a non-identity proper normal subgroup of  $R$.  Then, the set 
$$\{S\subseteq R\mid S=S^{-1},\,R=\norm{\Aut(R,S)}{R},\,  \exists f\in \norm{\Aut(\Cay(R,S))}{N}\textrm{ with }f\ne 1 \textrm{ and }1^f=1\},$$
has cardinality at most $2^{\mathbf{c}(R)-\frac{|N|}{96}+2\log_2|R|+(\log_2|R|)^2+3}$. Moreover, if $R$ is neither abelian of exponent greater than $2$ nor generalised dicyclic, we may drop the condition ``$R=\norm  {\Aut(\Cay(R,S))}R$'' in the definition of the set.
\end{theorem}

\begin{theorem}\label{main2}
Let $R$ be a finite group and let $N$ be a non-identity proper normal subgroup of  $R$.  Then, the set 
\begin{align*}
\{S\subseteq R\mid& S=S^{-1},\,R=\norm{\Aut(R,S)}{R},\,  \exists f\in \norm{\Aut(\Cay(R,S))}{N}\textrm{ with }f\ne 1 \textrm{ and }1^f=1, f \textrm{ fixes each }N\textrm{-orbit setwise}\}
\end{align*}
has cardinality at most $2^{\mathbf{c}(R)-\frac{|R|}{192|N|}+(\log_2|R|)^2+3}$. Moreover, if $R$ is neither abelian of exponent greater than $2$ nor generalised dicyclic, we may drop the condition ``$R=\norm  {\Aut(\Cay(R,S))}R$'' in the definition of the set.
\end{theorem}
By distinguishing the cases that $|N|\ge \sqrt{|R|}$ and $|R:N|\ge \sqrt{|R|}$, we obtain the following corollary.
\begin{corollary}\label{cor}
Let $R$ be a finite group and let $N$ be a non-identity proper normal subgroup of  $R$.   Then, the set 
\begin{align*}
\{S\subseteq R\mid& S=S^{-1},\,R=\norm{\Aut(R,S)}{R},\,  \exists f\in \norm{\Aut(\Cay(R,S))}{N}\textrm{ with }f\ne 1 \textrm{ and }1^f=1, f \textrm{ fixes each }N\textrm{-orbit setwise}\}
\end{align*}
has cardinality at most $2^{\mathbf{c}(R)-\frac{\sqrt{|R|}}{192}+2\log_2|R|+(\log_2|R|)^2+3}$. Moreover, if $R$ is neither abelian of exponent greater than $2$ nor generalised dicyclic, we may drop the condition ``$R=\norm  {\Aut(\Cay(R,S))}R$'' in the definition of the set.
\end{corollary}

Prior to launching into the pieces of the proof mentioned above, we provide some additional background and introductory material.

\subsection{General notation}\label{section:notation}
\begin{definition}\label{defeq:1}
{\rm 
Given a finite group $R$ and $x\in R$, we let $o(x)$ denote the order of the element $x$ and we let $${\bf I}(R):=\{x\in R\mid o(x)\le 2\}$$ be the set of elements of $R$ having order at most $2$. Given a subset $X$ of $R$, we write ${\bf I}(X):=X\cap {\bf I}(R)$. Given an inverse-closed subset $X$ of $R$, we let $$\mathbf{c}(X):=\frac{|X|+|{\bf I}(X)|}{2}.$$ 
}
\end{definition}

\begin{definition}\label{defeq:2}{\rm
Let $A$ be an abelian group of even order and of exponent greater than $2$, and let $y$ be an involution of $A$. The generalised dicyclic group $\Dic(A, y, x)$ is the group $\langle A, x\mid x^2=y, a^x=a^{-1},\forall a\in A\rangle$. A group is called generalised dicyclic if it is isomorphic to some $\Dic(A, y, x)$. When $A$ is cyclic, $\Dic(A, y, x)$ is called a dicyclic or generalised quaternion group.

We let $\bar{\iota}_A:\Dic(A,y,x)\to \Dic(A,y,x)$ be the mapping defined by $(ax)^{\bar{\iota}_A}=ax^{-1}$ and $a^{\bar{\iota}_A}=a$, for every $a\in A$. In particular, $\bar{\iota}_A$ is an automorphism of $\Dic(A,y,x)$.
The role of the label ``$A$'' in $\bar{\iota}_A$ seems unnecessary, however we use this label to stress one important fact. An abstract group $R$ might be isomorphic to $\Dic(A,y,x)$, for various choices of $A$. Therefore, since the automorphism $\bar{\iota}_A$ depends on $A$ and since we might have more than one choice of $A$, we prefer a notation that emphasizes this fact.

It follows from~\cite[Section~$2.1$ and~4]{MSV} that, if $D=\Dic(A,x,y)$ is generalized dicyclic over $A$, then either $A$ is characteristic in $D$, or $D\cong Q_8\times C_2^\ell$ for some $\ell\in \mathbb{N}$. In particular, when $D$ is not isomorphic to $Q_8\times C_2^\ell$, the automorphism $\bar{\iota}_A$ is uniquely determined by $D$.

When $D= Q_8\times C_2^\ell$, the group $D$ is generalized dicyclic over three distinct abelian subgroups; namely, if $Q_8=\langle i,j\rangle$, then $D$ is generalized dicyclic over $\langle i\rangle\times C_2^\ell$, $\langle j\rangle\times C_2^\ell$ and $\langle ij\rangle\times C_2^\ell$. In particular, we have three distinct options for the automorphism $\bar{\iota}_A$: one for each of these abelian subgroups. For simplicity, we denote by $\bar{\iota}_i,\bar{\iota}_j$ and $\bar{\iota}_k$ the corresponding automorphisms. It is not hard to check that $\bar{\iota}_k=\bar{\iota}_i\bar{\iota}_j$ and hence $\langle \bar{\iota}_i,\bar{\iota}_j\rangle$ is elementary abelian of order $4$.}
\end{definition}

\begin{definition}\label{defeq:2_2_2}
{\rm Let $A$ be an abelian group. We let $\iota_A:A\to A$ denote the automorphism of $A$ defined by $x^{\iota_A}=x^{-1}$ $\forall x\in A$. Very often, we drop the label $A$ from $\iota_A$ because this should cause no confusion.}
\end{definition}

In what follows we use the following facts repeatedly.
\begin{remark}\label{rem : 1}{\rm 
Let $X$ be a finite group. Since a chain of subgroups of $X$ has length at most $\log_2(|X|)$, $X$ has a generating set of cardinality at most $\lfloor \log_2(|X|)\rfloor\le \log_2(|X|)$.

Any automorphism of $X$ is uniquely determined by its action on the elements of a generating set for $X$. Therefore $|\Aut(X)|\le |X|^{\lfloor\log_2(|X|)\rfloor}\le 2^{(\log_2(|X|))^2}$. }
\end{remark}

\begin{lemma}\label{lemma111}Let $R$ be a finite group and let $X$ be an inverse-closed subset of $X$. The number of inverse-closed subsets $S$ of $X$ is $2^{\mathbf{c}(X)}.$ In particular, $R$ has $2^{\mathbf{c}(R)}$ inverse-closed subsets. 
\end{lemma}
\begin{proof}
 Given an arbitrary inverse-closed subset $S$ of $X$, $S\cap {\bf I}(X)$ is an arbitrary subset of ${\bf I}(X)$ whereas in $S\cap (X\setminus {\bf I}(X))$ the elements come in pairs, where each element is paired up to its inverse. Thus the number of inverse-closed subsets of $X$ is $$2^{|{\bf I}(X)|}\cdot 2^{\frac{|X\setminus {\bf I}(X)|}{2}}=2^{{\bf c}(X)}.$$
The last statement follows using $X=R$.
\end{proof}

The following important results by the third author deal with the case where there is a  graph automorphism that is also a group automorphism of $R$.

\begin{theorem}[\cite{spiga11}, Lemma $2.7$]\label{l:aut}Let $R$ be a finite group and let $\varphi$ be a non-identity automorphism of $R$. Then, one of the following holds
\begin{enumerate}
\item\label{eq:aut00} the number of $\varphi$-invariant inverse-closed subsets of $R$ is at most $2^{\mathbf{c}(R)-\frac{|R|}{96}}$,
\item\label{eq:aut01}$\cent R\varphi$ is abelian of exponent greater than $2$ and has index $2$ in $R$, $R$ is a generalized dicyclic group over $\cent R\varphi$ and $\varphi=\bar{\iota}_{\cent R\varphi}$,
\item\label{eq:aut02}$R$ is abelian of exponent greater than $2$ and $\varphi$ is the automorphism of $R$ mapping each element to its inverse.
\end{enumerate} 
\end{theorem}

\begin{proposition}[\cite{spiga11}, Proposition $2.8$]\label{propo:aut}Let $R$ be a finite group and suppose that $R$ is not an abelian group of exponent greater than $2$ and that $R$ is not a generalized dicyclic group. Then the set
$$\{S\subseteq R\mid S=S^{-1}, R<\norm{\Aut(\Cay(R,S))}{R}\}$$
has cardinality at most $2^{\mathbf{c}(R)-|R|/96+(\log_2|R|)^2}$.
\end{proposition}

\begin{notation}\label{Subdivision}
With $R$ a finite group that is neither abelian of exponent greater than $2$ nor generalised dicyclic, we define
$$\mathcal S_N=\{S\subseteq R\mid S=S^{-1},\,  \exists f\in \norm{\Aut(\Cay(R,S))}{N}\textrm{ with }f\ne 1 \textrm{ and }1^f=1\},$$ so that $|\mathcal S_N|$ is a value we aim to bound to prove Theorem~$\ref{main1}$.
We divide $\mathcal S_N$ into three subsets:
\begin{align*}
\mathcal{S}_N^1&:=\{S\in \mathcal{S}_N\mid R<\norm{\Aut(\Cay(R,S))}{R}\},\\
\mathcal{T}_N&:=\{S\in \mathcal{S}_N\setminus\mathcal{S}_N^1\mid  \exists x\in R\textrm{ and }\exists f\in \norm{\Aut(\Cay(R,S))}{N}\textrm{ with }1^f=1\textrm{ and }x^f\notin \{x,x^{-1}\}\},\\
\mathcal{U}_N&:=\mathcal{S}_N\setminus\mathcal{S}_N^1\setminus\mathcal T_N.
\end{align*}
so $$\mathcal S_N=\mathcal S_N^1 \cup \mathcal T_N \cup \mathcal U_N.$$
\end{notation}

Observe that $$\mathcal U_N=\{S\in \mathcal{S}_N\setminus \mathcal S_N^1 \mid  \forall f\in \norm{\Aut(\Cay(R,S))}{N}\textrm{ with }1^f=1\textrm{ we have }x^f\in \{x,x^{-1}\} \forall x \in R\}.$$

Proposition~\ref{propo:aut} already provides us with a bound for $|\mathcal S_N^1|$. In the next section, we will show that $|\mathcal U_N|=0$.

\section{Graph automorphisms that fix or invert every group element}\label{sec:inversion}

The bulk of this section consists of a long lemma in which we show that if a nontrivial permutation  that fixes or inverts every element of a group exists, then the normaliser of $R$ in the appropriate group is in fact larger than $R$. This means that any connection sets that could arise in $\mathcal U_N$ have actually already arisen in $\mathcal S_N^1$, and therefore do not appear in $\mathcal U_N$.

\begin{lemma}\label{yetanother} Let $G$ be a subgroup of $\Sym(R)$ with $R<G$ and with the property that $r^g\in \{r,r^{-1}\}$, for every $ r\in R$ and for every $g\in G_1$. Then $\norm G R>R$.
\end{lemma}
	\begin{proof}
		We argue by contradiction and, among all groups satisfying the hypothesis of this lemma, we
		choose $G$ with $|R||G|$ as small as possible and with $$R=\norm G R.$$
		In this proof, we denote by $r^g$ the image of the point $r\in R$ via the permutation $g$ and we denote by $r^{\iota_g}:=g^{-1}rg$ the conjugation of $r$ via $g$.

		Let $M$ be a subgroup of $G$ with $R<M$. For every $r\in R$ and for every $x\in M_1=M\cap G_1,$ $r^x\in \{r,r^{-1}\}$, and, from the modular law, $$R=M\cap R=M\cap \norm GR=\norm MR.$$ Therefore, by the minimality of our counterexample, we get $M=G$. As $M$ was an arbitrary subgroup of $G$ with $R<M$, we deduce 
		\begin{equation}\label{selfnormalizing1}R\textrm{ is a maximal subgroup of }G.
		\end{equation}

		Let $K$ be the core of $R$ in $G$, that is, $K:=\bigcap_{g\in G}R^g$. 
		
		We claim that	\begin{equation}\label{selfnormalizing2}
		\textrm{the core of }R \textrm{ in }G \textrm{ is }1.
		\end{equation}

		To prove this claim we argue by contradiction and we suppose that $K\ne 1$. Let $\bar{G}$ be the permutation group induced by $G$ on the action on $K$-orbits. Moreover, we let $\bar{\,}:G\to \bar{G}$ denote the natural projection.
		
		Let $H$ be the kernel of $\,\bar{\,}$. Thus $H$ is the largest subgroup of $G$ fixing each $K$-orbit setwise and $H\le G_1K$. Since $R$ is a maximal subgroup of $G$ and $R\le RH\le G$, we have that either $R=RH$ or $G=RH$.
		
		 In the first case, $H\le R$ and, since $H\le G_1K$, from the modular law we obtain $H\le R\cap G_1K=(R\cap G_1)K=K$, that is, $H=K$. Moreover, as $H=K\le R$, we have $\bar{R}=\norm{\bar{G}}{\bar{R}}$.
		Now, $\bar{R}$ is a regular subgroup of $\bar{G}\le \Sym(\bar{R})$ and, for every $\bar{r}\in \bar{R}$ and for every $\bar{g}\in \bar{G}_1$, we have $\bar{r}^{\bar{g}}\in \{\bar{r},\bar{r}^{-1}\}$. 
	Using our assumption that $K \neq 1$,  we get that $|\bar{R}|<|R|$, and by
		 the minimality of our couterexample we have that $\bar{G}=G/K=R/K=\bar{R}$. That is, $G=R$ contradicting the fact that $R$ is a proper subgroup of $G$.
		
		So the second case holds, and $G=RH$, so $G_1$ acts trivially on $K$-orbits. In other words, $G_1$ fixes each $K$-orbit setwise. Thus $H=KG_1$, and consequently
%
		\begin{equation}\label{eq:day1}
		KG_1\unlhd G.
		\end{equation}

		Suppose there exist $x\in G_1$ and $r\in R$ such that $r^x=r^{-1}$ and $o(rK)\ge 3$. Then $r^x=r^{-1}\in r^{-1}K=(rK)^{-1}\ne rK$, contradicting the fact that $G_1$ fixes each $K$-orbit. This shows that, 
		\begin{equation}\label{allinvolutions}
		\textrm{for every }x\in G_1\textrm{ and for every }r\in R\textrm{ either }r^x=r\textrm{ or }o(rK)\le 2.
		\end{equation}
		
		Let $L$ be  the subgroup of $R$ fixed pointwise by $G_1$, that is, $L:=\{r\in R\mid G_r=G_1\}$. (The set $L$ is indeed a subgroup of $R$, because  it is a block of imprimitivity for the action of $G$ on $R$ containing the point $1$.) Clearly, $L<R$, because $G_1\ne 1$.
		Now, from~\eqref{allinvolutions}, we deduce that, for every $r\in R\setminus L$, $o(rK)\le 2$. Hence, 	\begin{align}\label{inv} 		\textrm{every element in }	\;\frac{R}{K}\setminus \frac{KL}{K} \;	\;	\textrm{is an involution}.	\end{align}

		Now, by (\ref{inv}), we must have $\langle xK\in R/K \mid x^2\notin K\rangle\le L/K$. 
		Since either $|R/K:\langle xK\in R/K \mid x^2\notin K\rangle|=2$ or $R/K$ is a 2-group, we deduce that one of the following holds
		\begin{enumerate}
			\item\label{case1pablo} $R/K$ is an elementary abelian $2$-group,
			\item\label{case2pablo} $R=KL$,
			\item\label{case3pablo} $|R:KL|=2$ and every element in $R/K\setminus KL/K$ is an involution.
		\end{enumerate}
		In what follows, we analyze these three alternatives.
		
		\smallskip
		
		\noindent\textsc{Case~\eqref{case1pablo}} 
		
		\smallskip
		
		\noindent Since $R/K$ and $G_1$ are elementary abelian $2$-groups, we deduce that $G/K$ is a $2$-group. From $R/K<G/K$, it follows that $\norm {G/K}{R/K}>R/K$. So $\norm G R>R$, but this contradicts our choice of $G$ and $R$.

		\smallskip
		
		\noindent\textsc{Case~\eqref{case2pablo}} 
		
		\smallskip
		
		\noindent Let $f\in G_1$ with $f\ne 1$. Now, as $G_1$ normalizes $K$, the action of $f$ on the points in $K$ coincides with the action of $f$ by conjugation on $K$. Thus, $k^{\iota_f}=k^f\in \{k,k^{-1}\}$, for every $k\in K$. In particular, $\iota_f$ is a non-trivial automorphism of $K$ with the property that it maps each element to itself or to its inverse (so every inverse-closed subset of $K$ is invariant under $\iota_f$). 
		Therefore using Theorem~\ref{l:aut} only one of the following holds true: 
		\begin{itemize}
			\item $K$ is abelian of exponent greater than $2$ and $\iota_f=\iota$ is the automorphism inverting each element of $K$,
			\item $K$ is generalised dicyclic over an abelian subgroup $A$ of exponent greater than $2$ and $\iota_{f}=\bar{\iota}_A$, 
			\item $K\cong Q_8\times C_2^\ell$, for some $\ell\ge 0$, and $\iota_{f}\in\{\bar{\iota}_i,\bar{\iota}_j,\bar{\iota}_k\}$.
		\end{itemize}
		Since $R=KL$ and since $G_1$ fixes $L$ pointwise, the action of $g\in G_1$ on $R$ is uniquely determined once the action of $g$ on $K$ is determined. Since we have at most four choices for the action of $g\in G_1$ on $K$, we deduce that $|G_1|$ divides $4$. If $|G_1|=2$, then $|G:R|=2$ and hence $R\unlhd G$, which  contradicts $R=\norm G R$. Thus $4=|G_1|=|G:R|$ and $K\cong Q_8\times C_2^\ell$, for some $\ell\ge 0$. 
		
		Since $|G:R|=4$, the transitive action of $G$ on the right cosets of $R$ gives rise to a permutation group of degree $4$ and hence $G/K$ is isomorphic to a transitive subgroup of $\Sym(4)$. As $R/K=\norm{G/K}{R/K}$, we deduce that  $G/K$ is isomorphic to either $\Sym(4)$ or $\Alt(4)$. 
		
		If $R/K$ were a $2$-group, we reach a contradiction using the same argument as in Case ~\eqref{case1pablo}. So $R/K$ is a maximal subgroup of $G/K$ which is not a $2$-group, hence $R/K$ isomorphic to either $\Sym(3)$ or $\Alt(3)$. 
		
		Let $C$ be a Sylow $3$-subgroup of $R$. Thus $C=\langle c\rangle$ is a cyclic group of order 3. Since $K$ is a $2$-group and $R=KL$,  replacing $C$ by a suitable $R$-conjugate, from Sylow's theorem, we can assume that $C\le L$.
		 Let $k\in K$ with $k\notin L$. As $k$ is not fixed by each element of $G_1$, there exists  $x\in G_1$ such that $k^x=k^{-1}\ne k$. 
		Now, as $c^{x^{-1}}=c$, we obtain
		 \begin{align} \label{eq:repeat}
(ck)^x=c^{kx}=c^{x^{-1}kx}=c^{k^{\iota_x}}=c^{k^{-1}}=ck^{-1}.
		\end{align}
		On the other hand, $(ck)^x\in \{ck,(ck)^{-1}\}$. If $(ck)^x=ck$, then we deduce $k=k^{-1}$, contradicting the fact that $k^x\ne k$. If $(ck)^x=(ck)^{-1}$, we deduce $k^{-1}c^{-1}=ck^{-1}$ and hence  $k^{-1}=ck^{-1}c=c^2(k^{-1})^{\iota_c}$. Again we obtain a contradiction because $k$ and $k^{\iota_c}$ belong to $K$ but $c^2\notin K$.

		\smallskip
		
		\noindent\textsc{Case~\eqref{case3pablo}} 
		
		\smallskip
		
		\noindent Before proceeding with this case, we collect some information on $G/K$. Observe that in this case, $R/K$ is a generalized dihedral group over the abelian group $KL/K$. Consider the set $\Omega$ of the right cosets of $R/K$ in $G/K$. By~\eqref{selfnormalizing1} $R/K$ is a maximal subgroup of $G/K$. So $G/K$ is a primitive permutation with generalised dihedral point stabilisers.
		
		 These groups were classified in~\cite[Lemma~$2.2$]{DSV}. Using this and the fact that $G_1$ is $2$-elementary abelian group, the only possibility that can occur is that $G/K$ is a primitive group of affine type of degree $|R:K|=|G_1|$.
%
		Since $G=G_1R$ and $R\cap G_1=1$,  $G_1K/K$ acts regularly on $\Omega$.
		Moreover, as $KG_1\unlhd G$ by~\eqref{eq:day1}, $G_1K/K$ is the socle of $G/K$. Since every element of $G_1$ is an involution (it fixes or inverts each element of $R$), then $G_1K/K$ is an elementary abelian $2$-group.

		 Now, $R/K$ acts by conjugation irreducibly as a linear group over the elementary abelian $2$-group $G_1K/K.$ Let $\ell K\in LK/K\setminus \{K\}.$
		 Since $LK/K$ is abelian, then $\textbf{C}_{G_1K/K}(\ell K)=\{aK\in G_1K/K \mid \ell^{-1}a\ell K=aK \}$ is stable under the conjugation by $uK$, for every $uK\in LK/K.$ Further, since $R/K=\langle rK, LK/K\rangle$ , where  $rK=r^{-1}K$, and $r^{-1}\ell rK=\ell^{-1} K,$ for every $\ell K\in LK/K,$ then $\textbf{C}_{G_1K/K}(\ell K)$ is stable under the conjugation by $x K$. In other words, we proved that $\textbf{C}_{G_1K/K}(\ell K)$ is a proper $R$-submodule of the irruducible $R$-module $G_1K/K,$ and consequently $ \textbf{C}_{G_1K/K}(\ell K)$ is trivial. Summing up, $KL/K$ is abelian and $ \textbf{C}_{G_1K/K}(\ell K)$ is trivial for every $\ell K\in LK/K\setminus\{K\}.$
		Thus $KL/K$ is a cyclic group of odd order. Moreover, as the socle $G_1K/K$ has even order, $|KL/K|$ must be odd. We let $t:=|KL/K|$. At this point, the reader might find it useful to consider Figure~\ref{figure111}. Since $KL/K$ is cyclic, there exists $c\in L$ with 
		$\langle c\rangle K=KL$ and with $o(cK)=t$.

		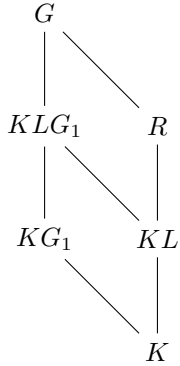
\begin{figure}
			\begin{tikzpicture}[node distance=1.5cm]
			\node(A0){$G$};
			\node(A1)[below of=A0]{$KLG_1$};
			\node(A2)[below of=A1]{$KG_1$};
			\node(A3)[right of=A1]{$R$};
			\node(A4)[below of=A3]{$KL$};
			\node(A5)[below of=A4]{$K$};
			\draw(A0)--(A1);
			\draw(A0)--(A3);
			\draw(A1)--(A2);
			\draw(A1)--(A4);
			\draw(A2)--(A5);
			\draw(A3)--(A4);
			\draw(A4)--(A5);
			\end{tikzpicture}
			\caption{Local structure of $\bar{G}$}\label{figure111}
		\end{figure}

		Suppose now that $K\nleq L$ and let $k\in K\setminus L$. As $k$ is not fixed by each element of $G_1$, there exists $x\in G_1$ with $k^x=k^{-1}\ne k$. 
Now, since  $x$ fixes $c$, we are in position to use the same argument as in  Case~\eqref{case2pablo}. That is (\ref{eq:repeat}) holds, and consequently either $k=k^{-1}$ or $c^2\in K.$ Since $k\ne k^{-1}$ and $o(cK)=t$ is odd, in both cases we get a contradiction.

		We conclude that $K\le L$. (For the proof here, it might be useful again considering Figure~\ref{figure111}.) In particular, $KL=L$. Fix $r\in R\setminus L$. As $|R:L|=2$, we have $R=L\cup rL$.  Now, $LG_1$ fixes $L$ and $rL$ setwise. The action induced by $LG_1$ on $L$ is the regular action of $L$ because $G_1$ fixes $L$ pointwise. As $LG_1\unlhd G$, we must also have that  the action of $LG_1$ on $rL$ is simply the regular action of $L$. In particular, for every $x\in G_1$, there exists $\ell_x\in L$ with the property that
		$$(r\ell)^x=r\ell \ell_x,\,\,\forall \ell\in L.$$
	The set $\{\ell_x\mid x\in G_1\}$ forms a subgroup of $L$, which we denote by $T$. As $G_1$ is elementary abelian, so is $T$. 

		Summing up, we have
		$$\ell^x=\ell,\quad (r\ell)^x=r\ell\ell_x,\,\,\forall x\in G_1,\forall \ell\in L.$$
	Using this and the fact that $T$ is a group we see that, if $x\in G_1$ fixes some point in $rL$, then $\ell_x=1$ and consequently  $x$ fixes all points in $rL$. Further, $x$ fixes all points in $L$, hence $x=1$. Therefore, each element in $G_1\setminus\{1\}$ acts fixed-point-freely on $rL$. Now, let $x\in G_1\setminus\{1\}$. Since $(r\ell)^x\in \{r\ell,(r\ell)^{-1}\}$ for each $\ell\in L$ 
	we deduce that $(r\ell)^{x}=(r\ell)^{-1}$ for every $\ell\in L$. Hence 
 $G_1\setminus\{1\}=\{x\}$. Therefore, $|G_1|=2$ and $|G:R|=2$ contradicting the fact that $\norm GR=R$.

		\smallskip
		We have shown that none of the three alternatives is possible. Therefore, we obtain a contradiction, and the contradiction has arisen from assuming $K\ne 1$.  Hence  $K=1$, which is our original claim (\ref{selfnormalizing2}).

		\medskip
		
		Now, as $R$ is maximal in $G$ and as $R$ is core-free in $G$, we may view $G$ as a primitive permutation group on the set $\Omega=G\backslash R$ of right cosets of $R$ in $G$. Observe that in this action $G_1$ acts as a regular subgroup and it is an elementary abelian $2$-group which itself is core-free in $G$.
		
		The primitive permutation groups containing an abelian regular subgroup have been classified by Caiheng Li in~\cite{Li}. Applying this classification~\cite[Theorem~$1.1$]{Li} to our group $G$ in its action on $\Omega$ and to its elementary abelian regular subgroup $G_1$, we deduce  that one of the following holds:
		\begin{enumerate}
			\item\label{caseA} $G$ is an affine primitive permutation group,
			\item\label{caseB} the set $\Omega$ admits a Cartesian decomposition $\Omega=\Delta^\ell$ (for some $\ell \ge 1$) and the primitive group $G$ preserves this cartesian decomposition; moreover, $\tilde{T}^\ell\le G\le \tilde{T}\mathrm{wr}\Sym(\ell)$, where the action of $\tilde{T}\wr\Sym(\ell)$ on $\Delta^\ell$ is the natural primitive product action. The group $\tilde{T}$ is either $\Alt(\Delta)$ or $\Sym(\Delta)$, 
$G_1=G_{1,1}\times G_{1,2}\times \cdots \times G_{1,\ell}$ with $G_{1,i}\le \tilde{T}$ and with $G_{1,i}$ acting regularly on $\Delta$, for each $i$.
		\end{enumerate}
	
	Now, we shall see that neither of these two alternatives is possible.

\smallskip
		
\noindent\textsc{Case~\eqref{caseA}}

\smallskip

\noindent Let $V$ be socle of $G$. Thus $V\unlhd G$ and $V$ is an elementary abelian $2$-group. Observe that $$G=VR=G_1R,$$
	where the first equality follows from the fact that $V$ acts transitively on $\Omega$ with point stabiliser $R$ and the second equality follows because $G$ acts also transitively on  $R$ with point stabilizer $G_1$. Moreover,
	$$V\cap R=1=G_1\cap R,$$
	where the first equality follows because $V$ acts regularly on $\Omega$ with point stabilizer $R$ and the second equality follows because $R$ acts regularly on itself with point stabilizer $G_1$.

	Since $G_1$ is a regular subgroup of the affine group $G$, from~\cite[Corollary~5~(1)]{caranti}, we deduce
	\begin{equation}\label{selfnormalizing77}
	V\cap G_1\ne 1.
	\end{equation}

	Let $$N:=\norm G{V\cap G_1}\quad \textrm{and let} \quad Q:=\norm R{V\cap G_1}.$$ Since $G_1$ is abelian, we have $G_1\le N$ and hence 
	\begin{equation*}N=N\cap G=N\cap RG_1=(N\cap R)G_1=QG_1.
	\end{equation*}
	Similarly, since $V$ is abelian, we have $V\le N$ and hence 
	\begin{equation*}N=N\cap G=N\cap RV=(N\cap R)V=QV.
	\end{equation*}
	Thus
	\begin{equation}\label{selfnormalizing6ter}N=QG_1=QV.
	\end{equation}

	Let $r\in R$ and let $v\in V\cap G_1$.
	We recall that $r^{v}\in \{r,r^{-1}\}$.
	
	If $r^v=r$, then 
	$1^r=r=r^v=1^{rv}$
	and hence $rvr^{-1}\in G_1.$ 
	If $r^v=r^{-1}$, then $1^{r^{-1}}=r^{-1}=r^v=1^{rv}$ 
	and hence $rvr=r^2(r^{-1}vr)\in G_1$. As $V\unlhd G$, we have $r^{-1}vr\in V$ and hence $r^2V\in G_1V/V$. Since all the elements of $G_1V/V$ have order at most $2$, it follows that $r^4V=V$, that is $r^4\in V\cap R=1$. This shows that, if $o(r)\ne 4$, then $r^{-1}vr\in V\cap G_1$. Therefore, all elements of $R$ of order different from $4$ normalise $V\cap G_1$ and hence they all lie in $Q$. 
	
	This shows that $R\setminus Q$ is either empty, or contains only elements of order $4$. 
	In the first case~\eqref{selfnormalizing6ter} yields $\norm G{V\cap G_1}=N=QV=RV=G,$ that is $V\cap G_1\unlhd G$. 
	Since $V$ is the unique minimal normal subgroup of $G$ and since $V\cap G_1\ne 1$ by~\eqref{selfnormalizing77}, we deduce that $V=V\cap G_1$, that is, $V\le G_1$. However, this contradicts the fact that $G_1$ is core-free in $G$. Thus 
	\begin{center}$Q<R$ and every element in $R\setminus Q$ has order $4$. 
	\end{center}
	
	For every $r\in R\setminus Q$, $r^2$ does not have order $4$, so $r^2\in Q$. This shows that $Q$ contains the square of each element of $R$, hence \begin{equation}\label{selfnormalizing3}
	Q\unlhd R
	\end{equation}
	and $R/Q$ is an elementary abelian $2$-group.
	
	Let $x\in G_1$ and let $r\in R$. If $r^x=r$, then $rxr^{-1}\in G_1\le G_1Q=N$. If $r^x=r^{-1}$, then $rxr\in G_1$ and hence $rxr=r^2(r^{-1}xr)\in G_1\le G_1Q=N$. Since $r^2\in Q$, we deduce that $r^{-2}\cdot r^2(r^{-1}xr)=r^{-1}xr\in N$. We have shown that,
	\begin{equation}\label{selfnormalizing4}
	\textrm{for every }r\in R,\,\,r^{-1} G_1r\le N.
	\end{equation}
	
	From~\eqref{selfnormalizing3} and~\eqref{selfnormalizing4}, we deduce that $R$ normalises $G_1Q=N$. Since $G_1$ also normalizes $N$, we have that $RG_1=G$ normalises $N$, that is, 
	\begin{equation}\label{selfnormalizing5}
	QV=QG_1=N\unlhd G.
	\end{equation}
	
	Since $Q\unlhd R$ and since $R$ is a maximal subgroup of $G$ by~\eqref{selfnormalizing1}, we deduce that either $\norm GQ=G$ or $\norm G Q=R$. If $\norm GQ=G$, then $Q$ is a normal subgroup of $G$ contained in the core-free subgroup $R$. Therefore $Q=1$.
	From~\eqref{selfnormalizing6ter}, we have $G_1=QG_1=N=QV=V$, contradicting the fact that $G_1$ is core-free in $G$. Thus 
	\begin{equation}\label{selfnormalizing6}\norm G Q=R.
	\end{equation}
	
	When $G$ is viewed as a permutation group on $R$, $QG_1$ is the setwise stabilizer in $G$ of $Q\subseteq R$, hence we can consider 
	the permutation group induced by $N=QG_1$ in its action on $Q$.
	From~\eqref{selfnormalizing6}, we have $\norm {N}Q=N\cap R=QG_1\cap R=Q(G_1\cap R)=Q$.
	Let $H$ be the kernel of the permutational representation of $N$ on $Q$. Note that $H\le G_1.$ 
	Now, $QH/H$ is a regular subgroup of $N/H\le \Sym(Q)$ and, for every $rH\in QH/H$ and for every $gH\in G_1/H$, we have $r^{g}H\in \{rH,r^{-1}H\}$.
	If $\norm {N/H}{QH/H}=QH/H$,  
	from the minimality of our counterexample, we deduce that either $N=G$ or $G_1$ acts trivially on $Q$.  In the first case, $G=N=\norm{G}{V\cap G_1}$, that is $G_1\cap V$ is a normal subgroup of $G$.
	Since $V$ is the unique minimal subgroup of $G$, and since $V\cap G_1\ne 1$ by~\eqref{selfnormalizing77}, we deduce that $V=V\cap G_1$,  and consequently, $V= G_1$. However, this contradicts the fact that $G_1$ is core-free in $G$. Therefore $G_1$ fixes $Q$ pointwise,  that is, $G_1$ is the kernel of the action of $N=QG_1$ on $Q$ and hence 
	\begin{equation}\label{selfnormalizing7}
	G_1\unlhd N=QG_1=VG_1.
	\end{equation} 
	
	Let 
	$$U:=\langle G_1^g\mid g\in G\rangle.$$
	Observe that $U\unlhd G$. From~\eqref{selfnormalizing5}, for every $g\in G$, we have $G_1^g\le N^g=N,$ that is $U\le N.$
	 Moreover, for every $g\in G$, from~\eqref{selfnormalizing7}, we have $G_1^g\unlhd N^g=N$. Since $G_1$ is an elementary abelian $2$-group, then each $G_1^g$ is a normal $2$-subgroup of $N$, for every $g\in G$. Consequently $U$ is a normal $2$-subgroup of $G$. In particular, $U\cap R$ is a normal $2$-subgroup of $R$. 
	
	Since $V$ is an  irreducible $\mathbb{F}_2R$-module and $U\cap R\unlhd R$, we deduce that $V$ is completely reducible $\mathbb{F}_2(U\cap R)$-module by Clifford's theorem. Since $V$ has characteristic $2$ and since $U\cap R$ is a $2$-group, this can happen only when $$U\cap R=1.$$
Since $V$ is the unique minimal normal subgroup of $G$ and since $U\unlhd G$, we have $V\le U.$ Further, $U=U\cap G=U\cap G_1R=(U\cap R)G_1=G_1 $ and hence $V=G_1$. This is a contradiction because $V$ is normal in $G$ but $G_1$ is core-free in $G$.
	
	Therefore we can assume that $\norm {N/H}{QH/H}>QH/H.$ That is, there exists a non-identity element $g\in G_1$ normalizing $QH/H$. Hence, for every $r\in Q$,  $g^{-1}rg=uh$,  for some $u\in Q$ and for some $h\in H$. Since $g\in G_1$, and $r^g\in \{r,r^{-1}\}$, we get $u=u^h=1^{uh}=1^{g^{-1}rg}=r^{g}.$ This means that $g^{-1}rgH\in \{rH,(rH)^{-1}\}$ for every $r\in Q$, and consequently $\iota_g$ is a non-identity automorphism of $QH/H$ with the property that $(rH)^{\iota_g}\in \{rH,(rH)^{-1}\}$, for every $ rH\in QH/H$. Thus from Theorem~\ref{l:aut},  $Q\cong QH/H$ is either an abelian group of exponent greater than $2$ or a generalized dicyclic group. 
	
		Since $V$ is an  irreducibly $\mathbb{F}_2R$-module and $ {\bf O}_2(Q)\unlhd R$, we deduce that $V$ is completely reducible $\mathbb{F}_2(Q)$-module by Clifford's theorem. Since $V$ has characteristic $2$ and since ${\bf O}_2(Q)$ is a $2$-group, this can happen only when \begin{align}\label{trivial2}
{\bf O}_2(Q)=1.
		\end{align}
 If $Q$ is a generalised dicyclic group, that is, $Q=Dic(A,y,x)$, with $A$ an abelian group of even order and of exponent greater than $2$, and $y$ an involution in $A$, then $\langle y\rangle$
	is a characteristic subgroup of order $2$, which contradicts (\ref{trivial2}). 
 Thus $Q$ is an abelian group, and $Q$ has odd order by (\ref{trivial2}).  Since $N=QV=QG_1$ by~\eqref{selfnormalizing5}, and since $V\unlhd N$, then $V$ is the unique Sylow $2$-subgroup of $N$. As $|G_1|=|V|$ and $G_1\le N$, we get $G_1=V$. This contradicts the fact that $G_1$ is core-free in $G$.

\smallskip

\noindent\textsc{Case~\eqref{caseB}}

\smallskip

\noindent	We identify  $\Omega$ with $\Delta^\ell$, and we recall that $\Alt(\Delta)^\ell \le G\le \Sym(\Delta)\mathrm{wr}\Sym(\ell)$.
	Let $\delta_1\in \Delta$ and let $\omega=(\delta_1,\ldots,\delta_1)\in \Omega$. Since $R$ is a maximal subgroup of $G$, replacing $R$ by a suitable conjugate we may suppose that $R=G_\omega$. Now, $\Alt(\Delta\setminus\{\delta_1\})^\ell\le R$. 
	Further, recall that
	$G_1=G_{1,1}\times G_{1,2}\times \cdots \times G_{1,\ell}$, where $G_{1,i}\le \Sym(\Delta)$ is an elementary abelian $2$-subgroup of acting regularly on $\Delta$, for each $i$. Let $\delta_2\in \Delta\setminus\{\delta_1\}$. As $G_{1,1}\le \Sym(\Delta)$ is transitive on $\Delta$, there exists $g\in G_{1,1}$ such that $\delta_1^g=\delta_2$ and, since $G_{1,1}$ is a $2$-group, rearranging the points from $\delta_3$ onwards if necessary, we can assume
	$$g=(\delta_1\,\delta_2)(\delta_3\,\delta_4)(\delta_5\,\delta_6)(\delta_7\,\delta_8)\cdots.$$
	(Observe that $|\Delta|\ge 8$ because $|\Delta|$ is a power of $2$ larger than $5$.) Let consider the $3$-cycle $r=(\delta_2\,\delta_3\,\delta_4)$ and observe that it lies in $R$ because it fixes the point $\delta_1$ and $R=G_{\omega}$.
	
	In this new setting, to look at the original action of $G$ on $R$, we have to identify the set $R$ with the set of right cosets of $G_1$ in $G$. In particular,
	$$G_1r=G_1(\delta_2\,\delta_3\,\delta_4)$$
	is such a point. We have
	$$G_1rg=G_1(\delta_2\,\delta_3\,\delta_4)(\delta_1\,\delta_2)(\delta_3\,\delta_4)(\delta_5\,\delta_6)(\delta_7\,\delta_8)\cdots 
	=G_1(\delta_1\,\delta_2\,\delta_4)(\delta_5\,\delta_6)(\delta_7\,\delta_8)\cdots .$$
Since neither $rgr^{-1}\in G_1$ nor $rgr\in G_1$, then $G_1rg\notin\{G_1r, G_1r^{-1}\}$. This contradicts our hypotheses.

	\smallskip
We have shown that neither of the alternatives is possible. Therefore, we have contradicted the existence of such $G$ and $R$.
\end{proof}


This is sufficient to show that $\mathcal U_N$ is empty.

\begin{corollary}\label{cor:new}
When $R$ is neither abelian of exponent greater than $2$ nor generalised dicyclic, $\mathcal U_N=\emptyset$.
\end{corollary}

\begin{proof}
Recall from Notation~\ref{Subdivision} that when $R$ is neither abelian of exponent greater than $2$ nor generalised dicyclic $$\mathcal S_N=\{S\subseteq R\mid S=S^{-1},\,  \exists f\in \norm{\Aut(\Cay(R,S))}{N}\textrm{ with }f\ne 1 \textrm{ and }1^f=1\},$$ while
$$\mathcal{S}_N^1=\{S\in \mathcal{S}_N\mid R<\norm{\Aut(\Cay(R,S))}{R}\},$$ and $$\mathcal U_N=\{S\in \mathcal{S}_N\setminus \mathcal S_N^1 \mid  \forall f\in \norm{\Aut(\Cay(R,S))}{N}\textrm{ with }1^f=1\textrm{ we have }x^f\in \{x,x^{-1}\} \forall x \in R\}.$$ Notice that the set of all elements of $\Aut(\Cay(R,S))$ that fix the vertex $1$ and fix or invert every other element of $R$ is a subgroup of $\Aut(\Cay(R,S))$. By Lemma~\ref{yetanother} with $G$ being generated by $R$ and the set of all such elements, we have $\mathcal U_N=\emptyset$. This is because every set that could lie in $\mathcal U_N$ must appear in $\mathcal S_N^1$. 
\end{proof}

\section{Groups with a ``large" normal subgroup}\label{sec:Nlarge}

We begin this section with a lovely little general result showing that in a non-abelian group, there cannot be a group automorphism $\alpha$ such that the result of computing $nn^\alpha$ is constant for more than $3/4$ of the group elements (and in fact in an abelian group, this can only happen if $\alpha$ is the automorphism that inverts every group element). For the special case where $\alpha$ is trivial and the constant is $1$, our proof relies on (so does not replace) classical work by Liebeck and MacHale \cite{LMac}. 

\begin{lemma}\label{lemma:icecream}
Let $N$ be a group, let $\alpha$ be an automorphism of $N$ and let $t\in N$. Then 
one of the following holds:
\begin{enumerate}
\item\label{lemma:icecream1}$|\{n\in N\mid nn^\alpha =t\}|\le 3|N|/4$,
\item\label{lemma:icecream2}$N$ is abelian, $t=1$ and $n^\alpha=n^{-1}$ $\forall n\in N$.
\end{enumerate}
\end{lemma}
\begin{proof}
We let $\mathcal{S}:=\{n\in N\mid nn^\alpha =t\}$.
Suppose $|\mathcal{S}|>3|N|/4$. Observe that, for every $n\in \mathcal{S}$, we have $n^\alpha=n^{-1}t$.

As $|\mathcal{S}|>3|N|/4$, we have $\mathcal{S}^{\alpha^{-1}}\cap\mathcal{S}\ne \emptyset$. Let $n\in \mathcal{S}^{\alpha^{-1}}\cap \mathcal{S}$, so that $n, n^\alpha \in \mathcal S$. Then $nn^\alpha =t$ because $n\in\mathcal{S}$, and $n^\alpha(n^\alpha)^{\alpha}=t$ because $n^\alpha\in \mathcal{S}$. Therefore, $t=n^\alpha (n^\alpha)^\alpha=(nn^\alpha)^\alpha=t^\alpha$, that is, $t=t^\alpha$.

As $|\mathcal{S}|>3|N|/4$, we have $|\mathcal{S}\cdot t\cap\mathcal{S}|=|\mathcal{S}\cdot t|+|\mathcal{S}|-|\mathcal{S}\cdot t\cup\mathcal{S}|>3|N|/4+3|N|/4-|N|=|N|/2$. Let $n\in \mathcal{S}\cdot t\cap \mathcal{S}$. Then $n=mt$, for some $m\in \mathcal{S}$. Therefore 
$$t^{-1}m^{-1}\cdot t=n^{-1}t=n^\alpha=(mt)^\alpha=m^\alpha t^\alpha = m^{-1}t\cdot t.$$
From this we obtain $mt=t^{-1}m$, that is, $t^m=t^{-1}$. As $n=mt$, we also have $t^n=t^{-1}$. We have shown that, for every 
$n\in \mathcal{S}\cdot t\cap\mathcal{S}$, we have $t^n=t^{-1}$. For every two elements $n_1,n_2\in N$ with  $t^{n_1}=t^{-1}=t^{n_2}$, we have $n_1n_2^{-1}\in \cent Nt$. Therefore, we deduce that $|N|/2<|\mathcal{S}\cdot t\cap \mathcal{S}|\le |\cent N t|$. 
Thus $N=\cent N t$ and $t\in \Z N$. Moreover, for every $n\in \mathcal{S}t\cap \mathcal{S}$, we have $t^n=t^{-1}$ and, as $t\in \Z N$, we have $t^n=t$. Thus $t^2=1$. Summing up, $t$ is a central element of $N$ of order at most $2$.

Suppose that $t=1$. Then $\mathcal{S}=\{n\in N\mid n^\alpha=n^{-1}\}$. In particular, $\alpha$ is an automorphism 
inverting more than 
$3|N|/4$ of the elements of $N$. From a classical result of Liebeck and MacHale~\cite{LMac}, we deduce that $N$ is abelian and 
$\alpha$ is the automorphism inverting each element of $N$, that is, $n^\alpha=n^{-1}$ $\forall n\in N$.

Suppose that $t\ne 1$. Since  $t\in \Z N$ and since $t^\alpha=t$, we may consider the group $\bar{N}:=N/\langle t\rangle$ and the induced automorphism $\bar{\alpha}:\bar{N}\to\bar{N}$. In particular, in $\bar{N}$, the set $\mathcal{S}$ projects to
the set $\bar{\mathcal{S}}=\{\bar{n}\in \bar{N}\mid \bar{n}^{\bar{\alpha}}=\bar{n}^{-1}\}$. Since this set has cardinality larger than $3|\bar{N}|/4$, applying again the theorem of Liebeck and MacHale, we deduce that $\bar{N}$ is abelian and $\bar{n}^{\bar{\alpha}}=\bar{n}^{-1}$ $\forall \bar{n}\in \bar{N}$. It follows that, for every $n\in N$, $n^\alpha\in \langle t\rangle n^{-1}=\{n^{-1},tn^{-1}\}$. 

Set $\mathcal{S}':=\{n\in N\mid n^\alpha=n^{-1}\}$. In particular, $\{\mathcal{S},\mathcal{S}'\}$ is a partition of $N$ and $|\mathcal{S}'|=|N\setminus\mathcal{S}|<|N|/4$.

Suppose that $N$ is not abelian. As $|N\setminus \Z N|\ge |N|/2$ and $|\mathcal{S}|>3|N|/4$, there exists $n\in (N\setminus \Z N)\cap \mathcal{S}$. Since $\bar{N}$ is abelian, we have $[N,N]=\langle t\rangle$, from which it follows that $|N:\cent N n|=2$. For every
 $m\in \cent N n\cap \mathcal{S}$, we have $(nm)^\alpha=n^\alpha m^\alpha=n^{-1}t\cdot m^{-1}t=n^{-1}m^{-1}t^2=m^{-1}n^{-1}=(nm)^{-1}$ and hence $nm\in \mathcal{S}'$. This shows that $n(\cent Nn \cap \mathcal{S})\subseteq \mathcal{S}'$. Now, $$
|\mathcal{S}'|\ge |n(\cent N n\cap \mathcal{S})|=|\cent N n \cap \mathcal{S}|=|\cent N n|+|\mathcal{S}|-|\cent N n\cup \mathcal{S}|\ge |\cent N n|+|\mathcal{S}|-|N|=|\mathcal{S}|-\frac{|N|}{2}>\frac{|N|}{4},$$
contradicting the fact that $|\mathcal{S}'|<|N|/4$. This contradiction has arisen assuming that $N$ is not abelian and hence $N$ is abelian.

Now, for every $n,m\in \mathcal{S}$, we have $(nm)^\alpha=n^{-1}t\cdot m^{-1}t=n^{-1}m^{-1}t^2=(nm)^{-1}$ and hence $nm\in \mathcal{S}'$. Therefore, $\mathcal{S}\cdot \mathcal{S}\subseteq \mathcal{S}'$, but this is impossible because $|\mathcal{S}'|<|\mathcal{S}|$. This contradiction has arisen from assuming $t\ne 1$ and hence $t=1$ and the proof is now complete.
\end{proof}

We will also require a similar result that considers when inversion is applied after the automorphism.

\begin{lemma}\label{lemma:gelato}
Let $N$ be a group, let $\alpha$ be an automorphism of $N$ and let $t\in N$. Then 
one of the following holds:
\begin{enumerate}
\item\label{lemma:gelato1}$|\{n\in N\mid n(n^\alpha)^{-1} =t\}|\le 3|N|/4$,
\item\label{lemma:gelato2}$t=1$ and $n^\alpha=n$ $\forall n\in N$.
\end{enumerate}
\end{lemma}
\begin{proof}
The proof of this is very similar to the proof of Lemma~\ref{lemma:icecream}, so we omit some of the repeated details.

We let $\mathcal{S}:=\{n\in N\mid n(n^\alpha)^{-1} =t\}$.
Suppose $|\mathcal{S}|>3|N|/4$. Observe that, for every $n\in \mathcal{S}$, we have $n^\alpha=t^{-1}n$.

As before, by taking some $n \in \mathcal{S}^{\alpha^{-1}}\cap\mathcal{S}$, we can conclude that $t=t^\alpha$.

As $|\mathcal{S}|>3|N|/4$, we can argue as before that $|\mathcal{S}^{-1}t\cap\mathcal{S}|>|N|/2$. Let $n\in \mathcal{S}^{-1}t\cap \mathcal{S}$. Then $n=mt$, for some $m\in \mathcal{S}^{-1}$; that is, $m^{-1} \in \mathcal S$. Notice that this means $(m^{-1})^\alpha=t^{-1}m^{-1}$, so $m^\alpha=mt$. Therefore 
$$t^{-1}(mt)=t^{-1}n=n^\alpha=(mt)^\alpha=m^\alpha t^\alpha = (mt)t.$$
From this we obtain $mt=t^{-1}m$, that is, $t^m=t^{-1}$. As $n=mt$, we also have $t^n=t^{-1}$. We have shown that, for every 
$n\in \mathcal{S}^{-1} t\cap\mathcal{S}$, we have $t^n=t^{-1}$. As before, this implies that $|N|/2<|\mathcal{S}^{-1} t\cap \mathcal{S}|\le |\cent N t|$. 
Thus $N=\cent N t$ and $t\in \Z N$. As before, this implies that $t^2=1$. Summing up, $t$ is a central element of $N$ of order at most $2$.

Suppose that $t=1$. Then $\mathcal{S}=\{n\in N\mid n^\alpha=n\}$. In particular, $\alpha$ is an automorphism 
fixing more than 
half of the elements of $N$. Since the set of fixed points of an automorphism is a subgroup of $N$, we deduce that  
$\alpha=1$; that is, $n^\alpha=n$ $\forall n\in N$.

Suppose that $t\ne 1$. Since  $t\in \Z N$ and since $t^\alpha=t$, we may consider the group $\bar{N}:=N/\langle t\rangle$ and the induced automorphism $\bar{\alpha}:\bar{N}\to\bar{N}$. In particular, in $\bar{N}$, the set $\mathcal{S}$ projects to
the set $\bar{\mathcal{S}}=\{\bar{n}\in \bar{N}\mid \bar{n}^{\bar{\alpha}}=\bar{n}\}$. Since this set has cardinality larger than $|\bar{N}|/2$, again we see that $\bar{n}^{\bar{\alpha}}=\bar{n}$ $\forall \bar{n}\in \bar{N}$. It follows that, for every $n\in N$, $n^\alpha\in \langle t\rangle n=\{n,tn\}$. 

Set $\mathcal{S}':=\{n\in N\mid n^\alpha=n\}$. In particular, $\{\mathcal{S},\mathcal{S}'\}$ is a partition of $N$ and $|\mathcal{S}'|=|N\setminus\mathcal{S}|<|N|/4$.

Now, for every $n,m\in \mathcal{S}$, we have $(nm)^\alpha=(tn)(tm)=(nm)t^2=nm$ since $t$ is central of order $2$, and hence $nm\in \mathcal{S}'$. Therefore, $\mathcal{S}\cdot \mathcal{S}\subseteq \mathcal{S}'$, but this is impossible because $|\mathcal{S}'|<|\mathcal{S}|$. Again this contradiction completes our proof.
\end{proof}

Our next few results show that except in some very special cases, if we have a group $T$ with an index-$2$ subgroup $N$ and a permutation of $T$ that has a very specific sort of action on every element of the nontrivial coset of $N$ in $T$, then the number of subsets of $T$ that are closed under both inversion and this permutation is vanishingly small relative to the number of Cayley graphs on $T$.
\begin{lemma}\label{lemma:aux1}
	Let $T$ be a finite group, let $N$ be a subgroup of $T$ having index $2$, let $\gamma\in T\setminus N$, let $t\in N$ and  let $\alpha_t:T\to T$ be any permutation  defined by 
	\begin{eqnarray*}
		n^{\alpha_t} \in N\quad\textrm{and}\quad (\gamma n)^{\alpha_t}=\gamma tn,\,\forall n\in N.
	\end{eqnarray*}
	Then one of the following holds:
	\begin{enumerate}
		\item\label{lemma:aux11}$|\{X\subseteq T\mid X=X^{-1}, X^{\alpha_t}=X\}|\le 2^{\mathbf{c}(T)-\frac{|N|}{16}}$,
		\item\label{lemma:aux12}$T\cong C_4\times C_2^\ell$ for some $\ell\in\mathbb{N}$, $t$ is the only non-identity square in $T$ and $N$ is an elementary abelian $2$-group,
		\item\label{lemma:aux13}$o(t)=2$, $t=\gamma^2$ and $T=\Dic(N,\gamma^2,\gamma)$,
		\item\label{lemma:aux14}$t=1$.
	\end{enumerate}
	In parts~\eqref{lemma:aux12},~\eqref{lemma:aux13} and~\eqref{lemma:aux14}, if $n^{\alpha_t} \in \{n, n^{-1}\}$ for every $n \in N$, then we have $x^{\alpha_t}\in \{x,x^{-1}\}$ $\forall x\in T$. 
\end{lemma}
\begin{proof}
	If $t=1$, then we obtain part~\eqref{lemma:aux14}. Thus, for the rest of the argument, we assume $t\ne 1$.
	
	Observe that $\alpha_t$ fixes $N$ setwise and induces on $T\setminus N$ a permutation which is the product of disjoint cycles each of whose lengths is $o(t)$. For simplicity, we let $\mathcal{S}:=\{X\subseteq T\mid X=X^{-1}, X^{\alpha_t}=X\}$.

	If $o(t)\ge 3$, then 
	$$|\mathcal{S}|\le 2^{\mathbf{c}(N)+\frac{|T\setminus N|}{3}}=
	2^{\mathbf{c}(N)+\frac{|N|}{3}}=
	2^{\frac{|N|+|I(N)|}{2}+\frac{|N|}{3}}\le
	2^{\frac{|N|+|I(T)|}{2}+\frac{|N|}{3}}
	\le 2^{\mathbf{c}(T)-\frac{|N|}{6}}$$
	and hence part~\eqref{lemma:aux11} follows.
	
	The only remaining possibility is $o(t)=2$. Consider $H:=\langle \alpha_t,\iota\rangle$, where $\iota:T\to T$ is the mapping defined by $x^\iota=x^{-1}$ $\forall x\in T$. Clearly, $S\in \mathcal{S}$ if and only if $S$ is $H$-invariant. The orbits of $H$ on $T\setminus N$ have even cardinality because $o(\alpha_t)=o(t)=2$ and $\alpha_t$ has no fixed points on $T\setminus N$. There are only two possibilities for $H$ having an orbit of cardinality $2$ on $T\setminus N$:
	\begin{itemize}
		\item this orbit is $\{\gamma n,\gamma tn\}$ where both $\gamma n$ and $\gamma tn$ are involutions (in this case $\iota$ fixes both 
		$\gamma n$ and $\gamma tn$), 
		\item this orbit is $\{\gamma n ,\gamma tn\}$ and $(\gamma n)^{-1}=\gamma tn$ (in this case $(\gamma n)^{\alpha_t}=(\gamma n)^\iota$). 
	\end{itemize}
	Let $n_0$ be an element in $N$ with $o(\gamma n_0)=o(\gamma tn_0)=2$. As $o(\gamma n_0)=2$, we have $n_0\gamma=\gamma^{-1}n_0^{-1}$ and hence $$1=(\gamma tn_0)^2=\gamma tn_0\gamma t n_0=\gamma t \gamma^{-1}n_0^{-1}tn_0.$$ Therefore $t(\gamma^{-1}n_0^{-1})t=\gamma^{-1}n_0^{-1}$. Since $o(t)=2$, we deduce $(n_0\gamma)^t=n_0\gamma$, that is, $n_0\gamma\in \cent Tt$.
	As $\gamma n_0=(n_0\gamma)^{\gamma^ {-1}}\in {\cent T{t}}^{\gamma^{-1}}=\cent T{t^{\gamma^ {-1}}}$, the elements of the first type are in the set $$\mathcal{A}:=I([T\setminus N] \cap \cent T {t^{\gamma^ {-1}}})=I(\cent{T\setminus N}{t^{\gamma^{-1}}}).$$ Let $n_1$ 
	be an element in $N$ with $(\gamma n_1)^{-1}=\gamma t n_1$. Let $n\in N$ and suppose that $\gamma n_1n\in T\setminus N$ also satisfies $(\gamma n_1n)^{-1}=\gamma tn_1n$. This means $n^{-1}\gamma tn_1=\gamma tn_1 n$, that is, $n^{{(\gamma tn_1)}^{-1}}=n^{-1}$. Therefore, the elements of the second type are in the set $$\mathcal{B}:=\gamma n_1\{n\in N\mid n^{\gamma tn_1}=n^{-1}\}.$$

	Observe that $\mathcal{A}$ or $\mathcal{B}$ might be the empty set: $\mathcal{A}=\emptyset$ when there is no involution in $\cent {T\setminus N}{t^{\gamma^{\color{blue} {-1}}}}$, $\mathcal{B}=\emptyset$ when there is no element $n_1\in N$ with $(\gamma n_1)^{-1}=\gamma tn_1$. Observe also that $\mathcal{A}\cap \mathcal{B}=\emptyset$: indeed, if $\gamma n\in \mathcal{A}\cap\mathcal{B}$, then $(\gamma n)^2=1$ and $(\gamma n)^{-1}=\gamma tn$, that is $t=1$, which is a contradiction.
	
	Since $X\in\mathcal{S}$ if and only $X$ is a union of orbits of $H$, we get
	\begin{eqnarray*}
		|\mathcal{S}|&\le &
		2^{\mathbf{c}(N)+\frac{|\mathcal{A}\cup \mathcal{B}|}{2}+\frac{|T\setminus N|-|\mathcal{A}\cup\mathcal{B}|}{4}}
		=
		2^{\mathbf{c}(N)+\frac{|\mathcal{A}\cup \mathcal{B}|}{4}+\frac{|T\setminus N|}{4}}
		=
		2^{\frac{|N|+|I(N)|}{2}+\frac{|\mathcal{A}\cup \mathcal{B}|}{4}+\frac{|N|}{4}}\\
		&=&
		2^{\frac{|T|+|I(N)|}{2}+\frac{|\mathcal{A}\cup \mathcal{B}|}{4}-\frac{|N|}{4}}=2^{\frac{|T|+|I(N)|}{2}+\frac{|\mathcal{A}|}{4}+\frac{|\mathcal{B}|}{4}-\frac{|N|}{4}}
		=
		2^{\frac{|T|+|I(N)\cup \mathcal{A}|}{2}-\frac{|\mathcal{A}|}{4}+\frac{|\mathcal{B}|}{4}-\frac{|N|}{4}}
		\le 
		2^{\mathbf{c}(T)-\frac{|\mathcal{A}|}{4}+\frac{|\mathcal{B}|}{4}-\frac{|N|}{4}}.
	\end{eqnarray*}
	If $|\mathcal{B}|\le 3|N|/4$, then
	$$|\mathcal{S}|\le 2^{\mathbf{c}(T)+\frac{3|N|}{16}-\frac{|N|}{4}}=
	2^{\mathbf{c}(T)-\frac{|N|}{16}}
	$$
	and part~\eqref{lemma:aux11} follows.
	Suppose now that $|\mathcal{B}|>3|N|/4$, that is, $|\{n\in N\mid n^{\gamma tn_1}=n^{-1}\}|> 3|N|/4$. This means that the action of $\gamma tn_1$ by conjugation on $N$ inverts more than $3/4$ of the elements of $N$. From~\cite{LMac},  $N$ is abelian and the action of $\gamma tn_1$ by conjugation on $N$ inverts each element of $N$. Therefore $\mathcal{B}\supset\gamma N$ and hence $\gamma\in \mathcal{B}$. Therefore $\gamma^{-1}=\gamma t$, that is, $t=\gamma^{2}$ (since $o(t)=2$). When $N$ is an elementary abelian $2$-group, we deduce $T\cong C_4\times C_2^\ell$ for some $\ell\in\mathbb{N}$ and hence part~\eqref{lemma:aux12} holds. When $N$ has exponent greater than $2$, we deduce $T=\Dic(N,\gamma^2,\gamma)$ and hence part~\eqref{lemma:aux13} holds.
\end{proof}

The hypotheses of the next lemma look much like the previous one, with the additional assumption that $N$ is abelian (of exponent greater than $2$), and a different action on the nontrivial coset of $N$. The exceptional cases and the proof are quite different, though.

\begin{lemma}\label{lemma:aux2}
	Let $T$ be a finite group, let $N$ be an abelian subgroup of $T$ having index $2$ and exponent greater than $2$, let $t\in N$, let $\gamma\in T\setminus N$, let $\alpha_t:T\to T$ be any permutation  defined by 
	\begin{eqnarray*}
		n^{\alpha_t} \in N \quad\textrm{and}\quad (\gamma n)^{\alpha_t}=\gamma tn^{-1},\,\forall n\in N.
	\end{eqnarray*}
 Further suppose that either $o(\gamma)=2$, or $(\gamma n)^{\alpha_t}=\gamma n$ whenever $o(\gamma n)=2$.
	Then one of the following holds:
	\begin{enumerate}
		\item\label{lemma:aux21} $|\{X\subseteq T\mid X=X^{-1}, X^{\alpha_t}=X\}|\le 2^{\mathbf{c}(T)-\frac{|N|}{24}}$;
		\item\label{lemma:aux22} $T$ is abelian and $t=\gamma^{-2}$; 
		\item\label{lemma:aux24}$T\cong Q_8\times C_2^\ell$ and $N\cong C_4\times C_2^\ell$ for some $\ell\in \mathbb{N}$;
		\item\label{lemma:aux25}$t=\gamma^2$, $T\cong \langle x,y\mid x^4=y^4=(xy)^4,x^2=y^2\rangle\times C_2^\ell$ and $N\cong C_4\times C_2^{\ell+1}$ for some $\ell\in \mathbb{N}$. (The group with presentation $\langle x,y\mid x^4=y^4=(xy)^4,x^2=y^2\rangle$ has order $16$.)
	\end{enumerate}
	In parts~\eqref{lemma:aux22},~\eqref{lemma:aux24} and~\eqref{lemma:aux25}, if $n^{\alpha_t} \in \{n, n^{-1}\}$ for every $n \in N$, then we have $x^{\alpha_t}\in \{x,x^{-1}\}$ $\forall x\in T$.
\end{lemma}
\begin{proof}
	We let $\iota:T\to T$ the permutation defined by $x^\iota=x^{-1}$ $\forall x\in T$.
	Since $N$ is abelian, for every $n\in N$, we have 
	\begin{align*}
	(\gamma n)^{\alpha_t^2}&=((\gamma n)^{\alpha_t})^{\alpha_t}=(\gamma tn^{-1})^{\alpha_t}=\gamma t(tn^{-1})^{-1}=\gamma tnt^{-1}=\gamma n.
	\end{align*} Thus $\alpha_t$ is a permutation having order $2$. Clearly, $\iota$ has also order $2$. For simplicity, we let $\mathcal{S}:=\{X\subseteq T\mid X=X^{-1}, X^{\alpha_t}=X\}$. In particular, $X\in\mathcal{S}$ if and only if $X$ is $\langle\alpha_t,\iota\rangle$-invariant, that is, $X$ is a union of $\langle \alpha_t,\iota\rangle$-orbits.
	
	Observe that $n^{-1}\gamma^{-1}=\gamma \cdot (\gamma^{-1}n^{-1}\gamma^{-1})$ and $\gamma^{-1}n^{-1}\gamma^{-1}\in N$ because $|T:N|=2$. Therefore 
	\begin{equation}\label{covid19}(n^{-1}\gamma^{-1})^{\alpha_t}=(\gamma\cdot \gamma^{-1}n^{-1}\gamma^{-1})^{\alpha_t}=\gamma t \gamma n\gamma.
	\end{equation}
	
	We divide the proof in two cases.
	
	\smallskip
	
	\noindent\textsc{Case $(\gamma n)^{\alpha_t}=\gamma n$ whenever $o(\gamma n)=2$.}
	
	
	\smallskip
	
	\noindent 
Note that $$c(T)=\frac{|T|}{2}+\frac{|I(T)|}{2}=\frac{|T|}{2}+\frac{|I(N)|}{2}+\frac{|I(T\setminus N)|}{2}=c(N)+\frac{|N|}{2}+\frac{|I(T\setminus N)|}{2}.$$ So $c(N) = c(T)-|N|/2-|I(T\setminus N)|/2$.
	
	Given $n\in N$, the $\langle\iota\rangle$-orbit containing $\gamma n$ is $\{\gamma n,n^{-1}\gamma^{-1}\}$. Now there are only two possibilities for $\alpha_t$ not fusing this $\langle\iota\rangle$-orbit with another $\langle\iota\rangle$-orbit. The first possibility is when $\alpha_t$ fixes both $\gamma n$ and $n^{-1}\gamma^{-1}$; the second possibility is when $(\gamma n)^{\alpha_t}=(\gamma n)^{\iota}$, that is, $\gamma tn^{-1}=n^{-1}\gamma^{-1}$.  Let 
	\begin{align*}
	\mathcal{A}&:=\{n\in N\mid (\gamma n)^{\alpha_t}=\gamma n, (n^{-1}\gamma^{-1})^{\alpha_t}=n^{-1}\gamma^{-1}\},\\
	\mathcal{B}&:=\{n\in N\mid \gamma tn^{-1}=n^{-1}\gamma^{-1}\}.
	\end{align*}
	
	Given $n\in \mathcal{A}$, we have $\gamma tn^{-1}=(\gamma n)^{\alpha_t}=\gamma n$ and, from~\eqref{covid19}, $\gamma t\gamma n\gamma=(n^{-1}\gamma^{-1})^{\alpha_t}=n^{-1}\gamma^{-1}$. The first equality yields $n^2=t$. The second equality yields 
	$$t=\gamma^{-1}n^{-1}\gamma^{-2}n^{-1}\gamma^{-1}=\gamma^{-1}n^{-2}\gamma^{-3}=\gamma^{-1}t^{-1}\gamma^{-3},$$
	where in the second equality we have used  that $\gamma^2\in N$ and that $N$ is abelian. Therefore, if $n\in \mathcal{A}$, then 
	$n^2=t$ and $t=\gamma^{-1}t^{-1}\gamma^{-3}$. Observe that the second condition does not depend on $n$ any longer. This means that we have two possibilities for $\mathcal{A}$; either $\mathcal{A}=\emptyset$, or $\mathcal{A}=n_0\Omega_2(N)$ where $\Omega_2(N):=\{n\in N\mid o(n)\le 2\}$ and where $n_0\in N$ satisfies $n_0^2=t$. Summing up
	\[
	\mathcal{A}=
	\begin{cases}
	\emptyset&\textrm{if there is no }n\in N\textrm{ with }n^2=t, \textrm{or if }t\ne \gamma^{-1}t^{-1}\gamma^{-3},\\
	n_0\Omega_2(N)&\textrm{where }n_0\in N \textrm{ satisfies }n_0^2=t \textrm{ and }t= \gamma^{-1}t^{-1}\gamma^{-3}.
	\end{cases}
	\] 
	
	Given $n\in \mathcal{B}$, we have $t=\gamma^{-1}n^{-1}\gamma^{-1}n=\gamma^{-1}n^{-1}\gamma n\gamma^{-2}=[\gamma,n]\gamma^{-2}$ (using $\gamma^2 \in N$ in the second equality).  This means that we have two possibilities for $\mathcal{B}$; either $\mathcal{B}=\emptyset$, or $\mathcal{B}=n_1\cent N\gamma$ where $n_1\in N$ satisfies $t=[\gamma,n_1]\gamma^{-2}$. Summing up
	\[
	\mathcal{B}=
	\begin{cases}
	\emptyset&\textrm{if there is no }n\in N\textrm{ with }t=[\gamma,n]\gamma^{-2},\\
	n_1\cent N\gamma&\textrm{where }n_1\in N \textrm{ satisfies }t=[\gamma,n_1]\gamma^{-2}.
	\end{cases}
	\] 
	
	
We claim that $\mathcal{A} \cap \mathcal B=\{n \in N: o(\gamma n)=2\}$. Certainly if $o(\gamma n)=2$ then by the case we are in, $(\gamma n)^{\alpha_t}=\gamma n=(\gamma n)^{-1}$ and therefore $n \in \mathcal A \cap \mathcal B$. Conversely, if $n \in \mathcal A \cap \mathcal B$ then $(\gamma n)^{\alpha_t}=\gamma n$ and $(\gamma n)^{\alpha_t}=(\gamma n)^{-1}$, so $o(\gamma n)=2$.  Therefore $|\mathcal A \cap \mathcal B|=|I(T \setminus N)|$.

Using the sets $\mathcal{A}$ and $\mathcal{B}$ we are ready to estimate $|\mathcal{S}|$. Indeed, we have
\begin{eqnarray}\label{covid20}|\mathcal{S}|&\le& 
2^{\mathbf{c}(N)+\frac{|\gamma N\setminus (\gamma\mathcal{A}\cup\gamma\mathcal{B})|}{4}+\frac{|\gamma\mathcal{A}\setminus\gamma(\mathcal A \cap \mathcal B)|}{2}+\frac{|\gamma\mathcal{B}\setminus\gamma(\mathcal A \cap \mathcal B)|}{2}+|\gamma(\mathcal A \cap \mathcal B)|}\\&=&\nonumber
2^{\mathbf{c}(N)+\frac{|\gamma N|}{4}+\frac{|\mathcal{A}|}{4}+\frac{|\mathcal{B}|}{4}}=
2^{\mathbf{c}(T)-\frac{|N|}{2}+\frac{|\gamma N|}{4}+\frac{|\mathcal{A}|}{4}+\frac{|\mathcal{B}|}{4}-\frac{|I(T\setminus N)|}{2}}=
2^{\mathbf{c}(T)-\frac{|N|}{4}+\frac{|\mathcal{A}|}{4}+\frac{|\mathcal{B}|}{4}-\frac{|\mathcal A \cap \mathcal B|}{2}}.
\end{eqnarray}

If  $\mathcal{A}=\mathcal{B}=\emptyset$, then part~\eqref{lemma:aux21} follows immediately. Suppose then $\mathcal{A}$ and $\mathcal{B}$ are not both empty. If $\mathcal{A}=\emptyset$, then part~\eqref{lemma:aux21} follows as long as $N\ne \cent N\gamma$. If $N=\cent N\gamma$, then $[\gamma ,n_1]=1$ and hence $t=\gamma^{-2}$. Thus, we obtain part~\eqref{lemma:aux22}.  If $\mathcal{B}=\emptyset$, then part~\eqref{lemma:aux21} follows as long as $N\ne \Omega_2(N)$.  However, since we are assuming that $N$ has exponent greater than $2$, we cannot have $N=\Omega_2(N)$. Thus we have finished discussing the case $\mathcal{A}=\emptyset$ or $\mathcal{B}=\emptyset$. We now assume $\mathcal{A}\ne\emptyset\ne \mathcal{B}$. 
In particular, $|N:\cent N\gamma|\ge 2$ and $|N:\Omega_2(N)|\ge 2$. If $|N:\cent N\gamma|\ge 3$ or if  $|N:\Omega_2(N)|\ge 3$, then from~\eqref{covid20} we have
$$|\mathcal{S}|\le 2^{\mathbf{c}(T)-\frac{|N|}{4}+\frac{|\mathcal{A}|}{4}+\frac{|\mathcal{B}|}{4}}\le 2^{\mathbf{c}(T)-\frac{|N|}{4}+\frac{|N|}{12}+\frac{|N|}{8}}=2^{\mathbf{c}(T)-\frac{|N|}{24}}$$
and part~\eqref{lemma:aux21} follows.

It remains to deal with the case that $|N:\Omega_2(N)|=2=|N:\cent N\gamma|$, so $\mathcal{A}$ and $\mathcal{B}$ are both cosets of an index $2$ subgroup of $N$. If $\mathcal A \cap \mathcal B \neq \emptyset$ then since both are cosets of index-2 subgroups of $N$, it is straightforward to see that their intersection has cardinality at least $|N|/4$, and part~\eqref{lemma:aux21} follows. If $\mathcal{A}\cap\mathcal{B}=\emptyset$, we obtain that $\mathcal{A}$ and $\mathcal{B}$ are both cosets of the same index $2$ subgroup of $N$. Therefore, $\cent N \gamma=\Omega_2(N)$ and $N \cong C_4 \times C_2^\ell$ for some $\ell \in \mathbb N$. Let us call this index-2 subgroup of $N$, $M$. Therefore, we have either $\mathcal{A}=M$ and $\mathcal{B}=N\setminus M$, or $\mathcal{A}=N\setminus M$ and $\mathcal{B}=M$. In the first possibility, we have $n_0^2=1$, $\mathcal{A}=\Omega_2(N)$, $\gamma^4=1$ and $\gamma^2=[\gamma,n_1]=\gamma^{-1}n_1^{-1}\gamma n_1$. From this it follows $\gamma^{-1}=n_1^{-1}\gamma n_1$. Since $n_1^2=\gamma^2$ is the unique involution that is a square in $N$, we get part~\eqref{lemma:aux24}. In the second possibility, $\gamma^{-2}=t=n_0^2$. If we also have $(\gamma n_0)^2=t$,  then $T=\Dic(N,\gamma^2,\gamma)$ and we obtain again part~\eqref{lemma:aux24}. If $(\gamma n_0)^2\ne t$, then $\langle \gamma,n_0\rangle$ has order $16$ and is isomorphic to the group with presentation $\langle x,y\mid x^4=y^4=(xy)^4=1, x^2=y^2\rangle$ and we obtain part~\eqref{lemma:aux25}.

\smallskip

\noindent\textsc{Case $o(\gamma)=2$.}
For every $n\in N$, from~\eqref{covid19} (and using $o(\gamma)=2$), we have
\begin{align*}
(\gamma n)^{\alpha_t\iota \alpha_t\iota}&=(\gamma tn^{-1})^{\iota\alpha_t\iota}=((tn^{-1})^{-1}(\gamma)^{-1})^{\alpha_t\iota}=(\gamma t\gamma (tn^{-1})\gamma)^{\iota}=(\gamma tt^{\gamma}(n^{-1})^{\gamma})^\iota\\
&=n^{\gamma} (tt^{\gamma})^{-1}\gamma=(tt^{\gamma})^{-1}n^{\gamma} \gamma=(t^{\gamma})^{-1}t^{-1}\gamma n=\gamma (t^{\gamma} t)^{-1}n=\gamma (tt^{\gamma})^{-1}n.
\end{align*}
Moreover, $n^{\alpha_t\iota\alpha_t\iota}\in N$ $\forall n\in N$.
Define $z:=(tt^{\gamma'})^{-1}$ and $\delta:T\to T$ by $$n^\delta=n^{\alpha_t\iota\alpha_t\iota} \textrm{ and }(\gamma n)^\delta=\gamma zn,\,\, \forall n\in N.$$ In particular, $\delta=\alpha_t\iota\alpha_t\iota$.

Recall that $X\in\mathcal{S}$ if and only if $X$ is $\langle \alpha_t,\iota\rangle$-invariant. Since $\delta\in \langle\alpha_t,\iota\rangle$, we deduce that $X$ is also $\langle \iota,\delta\rangle$-invariant.  

\smallskip

\noindent\textsc{Subcase $o(z)\ge 3$.}

\smallskip

\noindent Since the orbits of $\delta$ on $T\setminus N$ have all length $o(z)\ge 3$, we have
$$|\mathcal{S}|\le 2^{\mathbf{c}(N)+\frac{|N|}{3}}=2^{\frac{|N|+|I(N)|}{2}+\frac{|N|}{2}-\frac{|N|}{6}}= 2^{\frac{|T|+|I(N)|}{2}-\frac{|N|}{6}}\le 2^{\mathbf{c}(T)-\frac{|N|}{6}}$$
and part~\eqref{lemma:aux21} follows. 

\smallskip

\noindent\textsc{Subcase $o(z)=2$.}

\smallskip

\noindent For every $n\in N$, we have
$$
(\gamma n)^{\iota\delta\iota\delta}=
(n^{-1}\gamma)^{\delta\iota\delta}=
(\gamma (n^{-1})^{\gamma})^{\delta\iota\delta}
=(\gamma z(n^{-1})^{\gamma})^{\iota\delta}
=(n^{\gamma} z\gamma)^{\delta}
=(\gamma nz^{\gamma})^{\delta}
=(\gamma z^{\gamma} n)^{\delta}
=\gamma zz^{\gamma} n.$$
Define $\delta':T\to T$ by $$n^{\delta'}=n^{\delta} \textrm{ and }(\gamma n)^{\delta'}=\gamma zz^{\gamma} n,\,\, \forall n\in N.$$
If $X\in\mathcal{S}$, then $X$ is $\langle \delta,\iota\rangle$-invariant and hence $X$ is also $\langle \delta,\delta'\rangle$-invariant. Suppose $z^{\gamma}\ne z$. Since the orbits of $\langle \delta,\delta'\rangle$ on $T\setminus N$ have all length $|\langle z,z^{\gamma'}\rangle|\ge 4$, we have
$$|\mathcal{S}|\le 2^{\mathbf{c}(N)+\frac{|N|}{4}}=2^{\frac{|N|+|I(N)|}{2}+\frac{|N|}{2}-\frac{|N|}{4}}=2^{\frac{|T|+|I(N)|}{2}-\frac{|N|}{4}}\le 2^{\mathbf{c}(T)-\frac{|N|}{4}}$$
and part~\eqref{lemma:aux21} follows. 

Suppose $o(z)=2$ and $z^{\gamma}=z$. For every $n\in N$, we have
$$(\gamma n)^{\iota\delta}
=(n^{-1}\gamma)^{\delta}
=(\gamma (n^{-1})^{\gamma})^{\delta}
=\gamma z(n^{-1})^{\gamma}
=z\gamma (n^{-1})^{\gamma}
=zn^{-1}\gamma
=(\gamma zn)^{\iota}
=(\gamma n)^{\delta\iota}.$$
This shows that $\iota\delta=\delta \iota$ in its action on $T\setminus N$ and hence $\langle \iota_{|T\setminus N},\delta_{|T\setminus N}\rangle$ is an elementary abelian $2$-group of order $1$, $2$ or $4$. (Here we are denoting by $\iota_{|T\setminus N}$ and by $\delta_{|T\setminus N}$ the restrictions of $\iota$ and of $\delta$ to $T\setminus N$.) This group cannot have order $1$ because $o(z)= 2$ and hence $\delta_{|T\setminus N}$ is not the identity permutation.

If this group has order $2$, then $\iota_{|T\setminus N}$ must be either $\delta_{|T\setminus N}$ or the identity permutation. Suppose that $\iota_{|T\setminus N}=\delta_{|T\setminus N}$. Then for every $n \in N$ we have $n^{-1}\gamma=\gamma zn$, so $n^{\gamma}=zn^{-1}$ and hence $nn^{\gamma}=z$. But since $z \neq 1$, Lemma~\ref{lemma:icecream} implies that we cannot have $z=nn^{\gamma}$ for every $n \in N$.

So we must have $\iota_{|T\setminus N}$ being the identity permutation, that is, $n^{-1}\gamma=(\gamma n)^\iota=\gamma n$, so $n^{\gamma}=n^{-1}$ $\forall n\in N$. In particular, $\mathbf{c}(\gamma N)=|N|$ and $\mathbf{c}(T)=\mathbf{c}(N)+|N|$. Since the orbits of $\langle \delta\rangle$ on $T\setminus N$ have all length $o(z)=2$, we have $|\mathcal{S}|\le 2^{\mathbf{c}(N)+|N|/2}=2^{\mathbf{c}(T)-|N|/2}$ and part~\eqref{lemma:aux21} follows. 

It remains to consider the case that $\langle \iota_{|T\setminus N},\delta_{|T\setminus N}\rangle$ has order $4$. By the orbit counting lemma, the number of orbits of $\langle\iota\rangle$ on $T\setminus N$ is 
\begin{equation}\label{covid21}\frac{1}{2}(|T\setminus N|+|\mathrm{Fix}_{T\setminus N}(\iota)|)=\frac{1}{2}(|T\setminus N|+|I(T\setminus N)|)=\mathbf{c}(T\setminus N).
\end{equation} Also, by the orbit counting lemma, the number of orbits of $\langle\iota_{|T\setminus N},\delta_{|T\setminus N}\rangle$ on $T\setminus N$ is 
\begin{eqnarray*}
	\frac{1}{4}\left(|N|+|\mathrm{Fix}_{T\setminus N}(\iota)|+
	|\mathrm{Fix}_{T\setminus N}(\delta)|+
	|\mathrm{Fix}_{T\setminus N}(\iota\delta)|\right)&
	=&
	\mathbf{c}(T\setminus N)-\frac{|N|}{4}-\frac{|\mathrm{Fix}_{T\setminus N}(\iota)|}{4}+
	\frac{|\mathrm{Fix}_{T\setminus N}(\delta)|}{4}+\frac{|\mathrm{Fix}_{T\setminus N}(\iota\delta)|}{4}\\
	&=&
	\mathbf{c}(T\setminus N)-\frac{|N|}{4}-\frac{|\mathrm{Fix}_{T\setminus N}(\iota)|}{4}+\frac{|\mathrm{Fix}_{T\setminus N}(\iota\delta)|}{4}\\
	&\le& 
	\mathbf{c}(T\setminus N)-\frac{|N|}{4}+\frac{|\mathrm{Fix}_{T\setminus N}(\iota\delta)|}{4},
\end{eqnarray*}
where in the first equality we have used~\eqref{covid21} and in the second equality we have used the fact that $\delta$ has no fixed points on $T\setminus N$.
Now, $\gamma n\in\mathrm{Fix}_{T\setminus N}(\iota\delta)$ if and only if $\gamma n=(\gamma n)^{\iota\delta}=\gamma z(n^{-1})^{\gamma}$, that is, $z=nn^{\gamma}$. From Lemma~\ref{lemma:icecream}, we deduce $|\mathrm{Fix}_{T\setminus N}(\iota\delta)|\le 3|N|/4$ because $z\ne 1$. Thus
$$|\mathcal{S}|\le 
2^{\mathbf{c}(N)+\mathbf{c}(T\setminus N)-\frac{|N|}{4}+\frac{3|N|}{16}}=
2^{\mathbf{c}(T)-\frac{|N|}{16}}$$
and part~\eqref{lemma:aux21} follows. 

\smallskip

\noindent\textsc{Subcase $o(z)=1$.}

\smallskip

\noindent In this case, $tt^{\gamma}=z=1$ and $t^{\gamma}=t^{-1}$. In this case, for every $n\in N$, we have $$
(\gamma n)^{\iota\alpha_t}
=(\gamma (n^{-1})^{\gamma})^{\alpha_t}
=\gamma tn^{\gamma}
=t^{-1}\gamma n^{\gamma}
=t^{-1} n\gamma
=(\gamma tn^{-1})^{\iota}
=(\gamma n)^{\alpha_t\iota}.$$
This shows that $\iota\alpha_t=\alpha_t\iota$ on $T \setminus N$, and hence (in particular) $\langle \iota_{|T\setminus N},(\alpha_t)_{|T\setminus N}\rangle$ is an elementary abelian $2$-group of order $1$, $2$ or $4$. If $(\alpha_t)_{|T\setminus N}$ is the identity mapping,  then
$\gamma n=(\gamma n)^{\alpha_t}=\gamma t n^{-1},$ for every $n\in N$. In particular, $\gamma t=\gamma t t^{-1}$  which implies $t=1$. This means that for every $n \in N$, $\gamma n=(\gamma n)^{\alpha_t}=\gamma n^{-1}$,
so that $N$ is an elementary abelian $2$-group, contradicting our hypothesis that $N$ has exponent greater than $2$.

If $\iota_{T\setminus N}$ is the identity mapping, then $\mathbf{c}(\gamma N)=|N|$ and hence $\mathbf{c}(T)=\mathbf{c}(N)+|N|$. Observe that $$\mathrm{Fix}_{T\setminus N}(\alpha_t):=\{\gamma n\mid t=n^2\}.$$ Let $n_0^2=t$, an easy computation shows that
$$\mathrm{Fix}_{T\setminus N}(\alpha_t)=\gamma n_0\Omega_2(N), $$
hence $|\mathrm{Fix}_{T\setminus N}(\alpha_t)|=| \Omega_2(N)| \le |N|/2$. This shows that $\langle (\alpha_t)_{|T\setminus N}\rangle$ has at most $|N|/2+(|N|/2)/2=3|N|/4$ orbits on $T\setminus N$. Therefore
$$|\mathcal{S}|\le 2^{\mathbf{c}(N)+\frac{3|N|}{4}}=2^{\mathbf{c}(T)-|N|+\frac{3|N|}{4}}=2^{\mathbf{c}(T)-\frac{|N|}{4}}$$ and part~\eqref{lemma:aux21} follows. So we can assume that $\iota_{T\setminus N}$ is not the identity.

 
Since $\gamma^2=1$, when $\iota_{|T\setminus N}=(\alpha_t)_{|T\setminus N}$, then
	$t^{-1}\gamma=(\gamma t)^{\iota_{T\setminus N}}=(\gamma t)^{\alpha_t}=\gamma,$ so $t=1$. Further, $n^{-1}\gamma=(\gamma n)^{\iota_{T\setminus N}}=(\gamma n)^{\alpha_t}=\gamma n^{-1},$ for every $n\in N$, that is $T$ is abelian, and part~\eqref{lemma:aux22} holds.

It only remains to consider the case that $\langle\iota_{|T\setminus N}, (\alpha_t)_{|T\setminus N}\rangle$ has order $4$. 

By the orbit counting lemma, the number of orbits of $\langle\iota,\alpha_t\rangle$ on $T\setminus N$ is 
\begin{eqnarray}\label{covid19_1}
&&\frac{1}{4}\left(|N|+|\mathrm{Fix}_{T\setminus N}(\iota)|+|\mathrm{Fix}_{T\setminus N}(\alpha_t)|+
|\mathrm{Fix}_{T\setminus N}(\iota\alpha_t)|\right)\\\nonumber
&&=
\mathbf{c}(T\setminus N)-\frac{|N|}{4}-\frac{|\mathrm{Fix}_{T\setminus N}(\iota)|}{4}+\frac{|\mathrm{Fix}_{T\setminus N}(\alpha_t)|}{4}+\frac{|\mathrm{Fix}_{T\setminus N}(\iota\alpha_t)|}{4},
\end{eqnarray}
where the equality between the two members follows by~\eqref{covid21}.
If $|\mathrm{Fix}_{T\setminus N}(\alpha_t)|\le |N|/3$ and $|\mathrm{Fix}_{T\setminus N}(\iota\alpha_t)|\le |N|/2$, or $|\mathrm{Fix}_{T\setminus N}(\alpha_t)|\le |N|/2$ and $|\mathrm{Fix}_{T\setminus N}(\iota\alpha_t)|\le |N|/3$, then we immediately obtain part~\eqref{lemma:aux21}. Therefore we suppose that this does not hold. An easy computation reveals that
\begin{align*}
\mathrm{Fix}_{T\setminus N}(\iota\alpha_t)&:=\{\gamma n\mid t^{-1}=[n,\gamma]\}.
\end{align*}

As $(\alpha_t)_{|T\setminus N}$ and $(\iota\alpha_t)_{|T\setminus N}$ are not the identity mappings, we deduce 
\begin{itemize}
	\item $\mathrm{Fix}_{T\setminus N}(\alpha_t)=\gamma n_0\Omega_2(N)$, $n_0^2=t$ and $|N:\Omega_2(N)|=2$, 
	\item $\mathrm{Fix}_{T\setminus N}(\iota\alpha_t)=\gamma n_1\cent N {\gamma}$, $t^{-1}=[n_1, \gamma]$ and $|N:\cent N{\gamma}|=2$,
		\item $|\mathrm{Fix}_{T\setminus N}(\alpha_t)|=|N|/2=|\mathrm{Fix}_{T\setminus N}(\iota\alpha_t)|$.
\end{itemize} 

If $\Omega_2(N)\ne \cent N{\gamma}$ or if $\mathrm{Fix}_{T\setminus N}(\alpha_t)=\mathrm{Fix}_{T\setminus N}(\iota\alpha_t)$, we have $|\mathrm{Fix}_{T\setminus N}(\iota)|\ge |N|/4$, because $\mathrm{Fix}_{T\setminus N}(\iota)$ contains both $\gamma (\Omega_2(N)\cap \cent N{\gamma})$ and $ \mathrm{Fix}_{T\setminus N}(\alpha_t)\cap \mathrm{Fix}_{T\setminus N}(\iota\alpha_t)$. Hence, from~\eqref{covid19_1}, 
the number of orbits of $\langle\iota,\alpha_t\rangle$ on $T\setminus N$ is at most
\begin{eqnarray*}
	&&
	\mathbf{c}(T\setminus N)-\frac{|N|}{4}-\frac{|N|}{16}+\frac{|N|}{8}+\frac{|N|}{8}=\mathbf{c}(\gamma N)-\frac{|N|}{16}
\end{eqnarray*}
and part~\eqref{lemma:aux21} follows again. Assume, at last, $\Omega_2(N)=\cent N{\gamma}$ and $\mathrm{Fix}_{T\setminus N}(\alpha_t)\ne\mathrm{Fix}_{T\setminus N}(\iota\alpha_t)$. Set $M:=\Omega_2(N)=\cent N{\gamma}$. Then $\mathrm{Fix}_{T\setminus N}(\alpha_t)=\gamma M$ and $\mathrm{Fix}_{T\setminus N}(\iota\alpha_t)=\gamma(N\setminus M)$, or 
$\mathrm{Fix}_{T\setminus N}(\alpha_t)=\gamma(N\setminus M)$ and $\mathrm{Fix}_{T\setminus N}(\iota\alpha_t)=\gamma M$. If $\mathrm{Fix}_{T\setminus N}(\alpha_t)=\gamma M$, then $t=1$ and $1=t^{-1}=[\gamma,n_1]$. Thus $n_1\in \cent N{\gamma}=M$ and hence $\mathrm{Fix}_{T\setminus N}(\iota\alpha_t)=\gamma M$, contradicting $\mathrm{Fix}_{T\setminus N}(\iota\alpha_t)=\gamma(N\setminus M)$. Thus $\mathrm{Fix}_{T\setminus N}(\alpha_t)=\gamma(N\setminus M)$ and $\mathrm{Fix}_{T\setminus N}(\iota\alpha_t)=\gamma M$. As $\mathrm{Fix}_{T\setminus N}(\iota\alpha_t)=\gamma M=\gamma \cent N{\gamma}$, we have $n_1\in \cent N{\gamma}$ and hence $t^{-1}=[\gamma,n_1]=1$. Then $n_0^2=t=1$ and hence $\mathrm{Fix}_{T\setminus N}(\alpha_t)=\gamma \Omega_2(N)=\gamma M$, contradicting $\mathrm{Fix}_{T\setminus N}(\alpha_t)=\gamma(N\setminus M)$.
\end{proof}

The next lemma again has a similar flavour. This time we are assuming that the index-$2$ subgroup $N$ of $T$ is generalised dicyclic, and we need to assume that our permutation fixes each of the cosets of the abelian subgroup $A$ of $N$ setwise. 

\begin{lemma}\label{lemma:aux3}
Let $T$ be a finite group, let $N=\Dic(A,y,x)$ be a generalised dicyclic subgroup of $T$ having index $2$, let $t\in N$, let $\gamma\in T\setminus N$, 
let $\alpha_t:T\to T$ be any permutation  defined by 
\begin{eqnarray*}
a^{\alpha_t}\in A, (xa)^{\alpha_t} \in xA, \forall a \in A,\quad\textrm{and}\quad (\gamma n)^{\alpha_t}=\gamma tn^{\bar{\iota}_A},\,\forall n\in N.
\end{eqnarray*}
Recall that $\bar{\iota}_A$ is given in Definition~$\ref{defeq:2}$. Then one of the following holds:
\begin{enumerate}
\item\label{lemma:aux31}$|\{S\subseteq T\mid X=X^{-1}, X^{\alpha_t}=X\}|\le 2^{\mathbf{c}(\gamma N)-\frac{|N|}{24}}$,
\item\label{lemma:aux33}$\gamma^2=y=t$ and $a^\gamma=a^{-1}$ $\forall a\in A$,
\item\label{lemma:aux34}$t=1$, $\langle \gamma, A\rangle$ is abelian, and $T=\Dic(\langle \gamma, A \rangle, y, x)$.
\end{enumerate}
In parts~\eqref{lemma:aux33} and ~\eqref{lemma:aux34}, if $n^{\alpha_t} \in \{n, n^{-1}\}$ for every $n \in N$, then 
we have $z^{\alpha_t}\in \{z,z^{-1}\}$ $\forall x\in T$.
\end{lemma}
\begin{proof}
We let $\iota:T\to T$ the permutation defined by $z^\iota=z^{-1}$ $\forall z\in T$.
For simplicity, we let $\mathcal{S}:=\{X\subseteq T\mid X=X^{-1}, X^{\alpha_t}=X\}$.
Observe that, for every $a\in A$, we have $a^{\alpha_t}\in A$ and 
\begin{equation}\label{covid23}(\gamma a)^{\alpha_t}=\gamma t a^{\bar{\iota}_A}=\gamma t a.\end{equation}

Suppose $o(t)\ge 3$. Then the orbits of $\langle \alpha_t\rangle$ on $\gamma A$ all have length $o(t)\ge 3$ and hence
$$|\mathcal{S}|\le 2^{\mathbf{c}(T\setminus (\gamma A\cup \gamma^{-1}A)
)+\frac{|\gamma A|}{3}}\le 2^{\mathbf{c}(T)-\frac{|A|}{2}+\frac{|A|}{3}}=2^{\mathbf{c}(T)-\frac{|A|}{6}}=2^{\mathbf{c}(T)-\frac{|N|}{12}}$$
and part~\eqref{lemma:aux31} follows in this case. In particular, for the rest of the proof we may suppose that $o(t)\le 2$. Since $N$ is generalised dicyclic and $t\in N$, we obtain $t\in A$. Now, for every $a\in A$, we have $(\gamma a)^{\alpha_t}=\gamma t a\in \gamma A$ and hence $\gamma A$ is $\alpha_t$-invariant. Therefore $\alpha_t$ has $|A|/o(t)$ cycles on $\gamma A$. This also means that $\gamma x A$ is $\alpha_t$-invariant.

Suppose that $\gamma^2\notin A$, that is, $\gamma A\ne \gamma^{-1}A$. Then $T/A$ is a cyclic group and $N=\langle \gamma^2,A\rangle$.  If $o(t)\ne 1$, then  
$$|\mathcal{S}|\le 2^{\mathbf{c}(T\setminus (\gamma A\cup \gamma^{-1}A))+\frac{|A|}{2}}=2^{\mathbf{c}(T)-|A|+\frac{|A|}{2}}=2^{\mathbf{c}(T)-\frac{|A|}{2}}=2^{\mathbf{c}(T)-\frac{|N|}{4}}$$
and part~\eqref{lemma:aux31} follows in this case. Suppose then $t=1$. In this case $\alpha_t$ fixes $\gamma A$ pointwise. 
 For every $a\in A$, we have
\begin{equation}\label{covid22}(\gamma^{-1}a)^{\alpha_t}=(\gamma (\gamma^{-2}a))^{\alpha_t}=\gamma (\gamma^{-2}a)^{\bar{\iota}_A}=\gamma \gamma^{2}a=\gamma^{3}a.
\end{equation}
As $\langle \gamma^2,A\rangle=N=\Dic(A,y,x)$ and as all elements in $N\setminus A$ have order $4$, we deduce $o(\gamma^2)=4$ and $o(\gamma)=8$. In particular, $\gamma^3\ne \gamma^{-1}$ and from~\eqref{covid22}
 we deduce that $\alpha_t$ has no fixed points on $\gamma^{-1}A$. Hence $\alpha_t$ has at most $|A|/2$ cycles on $\gamma^{-1}A$. Therefore 
$$|\mathcal{S}|\le 2^{\mathbf{c}(T\setminus (\gamma A\cup \gamma^{-1}A))+\frac{|A|}{2}}=2^{\mathbf{c}(T)-|A|+\frac{|A|}{2}}=2^{\mathbf{c}(T)-\frac{|A|}{2}}=2^{\mathbf{c}(T)-\frac{|N|}{4}}$$
and part~\eqref{lemma:aux31} follows in this case. 

Henceforth we may assume that $\gamma^2\in A$. Then $\langle \gamma,A\rangle$ is a group having a subgroup $A$ of index $2$. Furthermore, since both $N=\langle x, A \rangle$ and $\langle \gamma,A\rangle$ are index-$2$ subgroups of $T$, we must have $(\gamma x)^2 \in N \cap \langle \gamma,A\rangle=A$. Also, since $\gamma$ and $x$ both normalise $A$, so does $\gamma x$. So $\langle \gamma x,A \rangle$ is a group having a subgroup of index $2$ and $\alpha_t$ restricts to a permutation of $\langle \gamma x,A \rangle$. Since $t \in A$ and $o(t) \le 2$ we see that $x$ and $t$ commute, so for every $a \in A$ we have 
\begin{equation}\label{covid24}(\gamma x a)^{\alpha_t}=\gamma t(xa)^{\bar{\iota}_A}=\gamma tx^{-1}a=
\gamma x^{-1}t a=\gamma x(x^{2}t)a.
\end{equation}
So we can apply Lemma~\ref{lemma:aux1} to the group $\langle \gamma x, A\rangle$ and the permutation $(\alpha_t)_{|\langle \gamma x, A\rangle}$ with $\gamma x$ taking the role of the ``$\gamma$" in that lemma, and $x^2t$ taking the role of ``$t$."

If part~\eqref{lemma:aux11} in Lemma~\ref{lemma:aux1} holds, then
$$|\mathcal{S}|\le 2^{\mathbf{c}(T\setminus\langle\gamma x,A\rangle)+\mathbf{c}(\langle\gamma x,A\rangle)-\frac{|A|}{16}}=2^{\mathbf{c}(T)-\frac{|N|}{32}}$$
and conclusion~\eqref{lemma:aux31} holds.

If part~\eqref{lemma:aux12} in Lemma~\ref{lemma:aux1} holds, then $A$ is an elementary abelian $2$-group, but this contradicts our definition of a generalised dicyclic group together with our hypothesis that $N$ is such a group.

So either part~\eqref{lemma:aux13} in Lemma~\ref{lemma:aux1} holds, so that $o(x^2t)=2$, $x^2t=(\gamma x)^2$, and $\langle \gamma x, A\rangle=\Dic(A,(\gamma x)^2, \gamma x)$; or part~\eqref{lemma:aux14} holds, so that $x^2t=1$, meaning $x^2=t$. We postpone further consideration of these cases briefly.

We can also apply Lemma~\ref{lemma:aux1} to the group $\langle \gamma, A\rangle$ and the permutation $\alpha_t$. In this case $\gamma$ takes the role of ``$\gamma$" in the lemma, and $t$ takes the role of ``$t$".

If part~\eqref{lemma:aux11} in Lemma~\ref{lemma:aux1} holds, then
$$|\mathcal{S}|\le 2^{\mathbf{c}(T\setminus\langle\gamma,A\rangle)+\mathbf{c}(\langle\gamma,A\rangle)-\frac{|A|}{16}}=2^{\mathbf{c}(T)-\frac{|N|}{32}}$$
and conclusion~\eqref{lemma:aux31} holds. 

If part~\eqref{lemma:aux12} in Lemma~\ref{lemma:aux1} holds, then $A$ is an elementary abelian $2$-group, again a contradiction.

So either part~\eqref{lemma:aux13} in Lemma~\ref{lemma:aux1} holds, so that $o(t)=2$, $t=\gamma^2$, and $\langle \gamma, A \rangle =\Dic(A, t,\gamma)$; or part~\eqref{lemma:aux14} of Lemma~\ref{lemma:aux1} holds, so that $t=1$. 

We have now applied Lemma~\ref{lemma:aux1} to two different subgroups of $T$, and have completed the proof except in the cases where parts~\eqref{lemma:aux13} or~\eqref{lemma:aux14} arise from both applications. We now consider these final four possible outcomes individually.

It is not possible that part~\eqref{lemma:aux14} holds in both applications, since this would imply that $t=1$ and $x^2=t$, contradicting $o(x)=4$ from the definition of a generalised dicyclic group.

If part~\eqref{lemma:aux13} holds in both applications, then $\langle \gamma x, A\rangle=\Dic(A,(\gamma x)^2, \gamma x)$ implies that $a^{\gamma x}=a^x=a^{-1}$, so $a^\gamma=a$  for every $a \in A$. But $\langle \gamma, A \rangle =\Dic(A, t,\gamma)$ implies that $a^\gamma=a^{-1}$ for every $a \in A$. Taken together, these imply that $A$ is an elementary abelian $2$-group, again a contradiction.

If part~\eqref{lemma:aux13} holds in the first application and part~\eqref{lemma:aux14} holds in the second, then we have $t=1$, ($o(x^2t)=2$), $x^2t=(\gamma x)^2$, and $\langle \gamma x, A\rangle=\Dic(A,(\gamma x)^2, \gamma x)$. Since $\langle \gamma x, A\rangle=\Dic(A,(\gamma x)^2, \gamma x)$, we see that $a^{\gamma x}=a^x=a^{-1}$, so $a^\gamma=a$  for every $a \in A$, and $\langle \gamma, A \rangle$ is abelian. Since $x^2t=x^2=(\gamma x)^2$, we have $\gamma^x=\gamma^{-1}$, so $T=\Dic(\langle \gamma, A\rangle, y, x)$. This is conclusion~\eqref{lemma:aux34}.

Finally, if part~\eqref{lemma:aux14} holds in the first application and part~\eqref{lemma:aux13} holds in the second, then we have $y=x^2=t$, $o(t)=2$, $t=\gamma^2$, and $\langle \gamma, A \rangle =\Dic(A, t,\gamma)$. This is conclusion~\eqref{lemma:aux33}.

\end{proof}

With these preliminary results in hand, we are ready to prove bounds on the number of connection sets that admit various types of graph automorphisms. Recall Notation~\ref{Subdivision}. We already have bounds on $|\mathcal S_N^1|$ and on $|\mathcal U_N|$. Our goal in this section is to bound $|\mathcal T_N|$ when $|N|$ is relatively large. In order to do this, we need to further subdivide $\mathcal T_N$.

\begin{notation}\label{notationstep1}
{\rm For what follows, $R$ is a group that is neither generalised dicyclic, nor abelian of exponent greater than $2$. We let $N$ be normal subgroup of $R$ and we let 
\begin{align*}
\mathcal{T}_N^1&:=\{S\in \mathcal{S}_N\setminus\mathcal{S}_N^1\mid&&  \exists x\in R\textrm{ and }\exists f\in \norm{\Aut(\Cay(R,S))}{N}\textrm{ with }1^f=1\textrm{ and }(xN)^f\notin \{xN,x^{-1}N\}\},\\
\mathcal{T}_N^2&:=\{S\in \mathcal{S}_N\setminus \mathcal{S}_N^1\setminus\mathcal T_N^1\mid&&\exists f\in \norm{\Aut(\Cay(R,S))}N\setminus\cent{\Aut(\Cay(R,S))}N \textrm{ with } 1^f=1 \textrm{ and }\\
& && N \textrm{ is neither abelian of exponent greater than 2 nor generalised dicyclic, or}\\
& && N\textrm{ is abelian of exponent greater than 2 and }n^{f}\ne n^{-1} \textrm{ for some } n\in N, \textrm{ or}\\
& &&N=\Dic(A,y,x)\not\cong Q_8\times C_2^\ell\textrm{  and }n^f\ne n^{\bar{\iota}_A} \textrm{ for some }n\in N, \textrm{ or}\\
& &&N\cong Q_8\times C_2^\ell \textrm{ and } n^f\notin \{n^{\bar{\iota}_i},n^{\bar{\iota}_j},n^{\bar{\iota}_k}\} \textrm{ for some }n\in N\},\\
\mathcal{T}_N^3&:=\{S\in\mathcal{S}_N\setminus\mathcal S_N^1 \setminus\bigcup_{\ell=1}^2\mathcal{T}_N^\ell\mid&& 
\exists x\in R\textrm{ and }\exists f\in \norm{\Aut(\Cay(R,S))}{N}\textrm{ with }1^f=1, (xN)^f\ne xN \textrm{ and }\\
& &&\textrm{either }N \textrm{ is non-abelian or there exists }n\in N \textrm{ with }(xn)^f\ne (xn)^{-1}\},\\
\mathcal{T}_N^4&:=\{S\in \mathcal{S}_N\setminus \mathcal S_N^1\setminus\bigcup_{\ell=1}^3\mathcal{T}_N^\ell\mid&&  \exists x\in R\textrm{ and }\exists f\in \norm{\Aut(\Cay(R,S))}{N}\textrm{ with }1^f=1\textrm{ and }x^f\notin \{x,x^{-1}\}\}.
\end{align*}}
It should be clear from this definition that $$\mathcal T_N=\bigcup_{\ell=1}^4 \mathcal T_N^\ell.$$
\end{notation}

We will bound the cardinality of each of these sets. Most of the bounds we find will only be vanishingly small relative to $2^{\mathbf c(R)}$ if $|N|$ is relatively large compared to $|R|$. Specifically, they will all work if $|N| \ge 9\log_2|R|$. In order to create the best possible bound, however, we will want to balance $|N|$ against $|R/N|$, so we will use these bounds only when $|N| \ge \sqrt{|R|}$.

The first bound is only useful if $|N|/2$ dominates $2\log_2|R|$. In particular, it will be useful if $|N| \ge 5\log_2|R|$.

\begin{proposition}\label{step2}
We have $|\mathcal{T}_N^1|\le 2^{\mathbf{c}(R)-\frac{|N|}{2}+2\log_2|R|-\log_2|N|+(\log_2|N|)^2+2}$.
\end{proposition}
\begin{proof}
Let $S\in\mathcal{T}_N^1$ and set $G_S:=\norm{\Aut(\Cay(R,S))}{N}$. Say, $(xN)^f=yN$, for some $xN,yN\in R/N$ with $yN\notin\{xN,x^{-1}N\}$ and for some $f\in G_S$ with $1^f=1$. Now, $x^f=yt$, for some $t\in N$. Observe that
\begin{equation}\label{virusvirus}(xn)^f=x^{nf}=x^{f(f^{-1}nf)}=ytn^{\iota_f},
\end{equation}
where we are denoting by $\iota_f:N\to N$ the automorphism induced by the conjugation via $f$ on $N$. Observe that we have at most $|\Aut(N)|\le 2^{(\log_2|N|)^2}$ choices for the automorphism $\iota_f$. Therefore, as $t\in N$, given $xN$ and $yN$, we deduce from~\eqref{virusvirus} that we have at most $|N|2^{(\log_2|N|)^2}$ choices for the permutation $f_{|xN}:xN\to yN$ restricted to $xN$.

We consider various possibilities: 
\begin{itemize}
\item[(i)] $o(xN)=o(yN)=2$, or
\item[(ii)] $o(xN)>2$ and $o(yN)>2$, or 
\item[(iii)] $o(xN)=2$ and $o(yN)>2$, or 
\item[(iv)] $o(xN)>2$ and $o(yN)=2$.
\end{itemize} 
We consider these cases in turn: we let $\mathcal{B}_i,\mathcal{B}_{ii},\mathcal{B}_{iii},\mathcal{B}_{iv}$ be the subsets of $\mathcal{S}_N^2$ satisfying, respectively, (i),~(ii),~(iii) or~(iv). In the first case, the number of inverse-closed subsets of $R\setminus (xN\cup yN)$ is $2^{\mathbf{c}(R)-\mathbf{c}(xN)-\mathbf{c}(yN)}$ and the number of inverse-closed $f$-invariant subsets $T$ of $xN\cup yN$ is at most $2^{\mathbf{c}(xN)}$, because once $T\cap xN$ has been chosen the set $T\cap yN$ must equal $(T\cap xN)^f$. Therefore
\begin{eqnarray*}
|\mathcal{B}_i|&\le& 
|N|2^{(\log_2|N|)^2}|R/N|^2 2^{\mathbf{c}(R)-\mathbf{c}(xN)-\mathbf{c}(yN)}\cdot 2^{\mathbf{c}(xN)}\\
&=& 
2^{\mathbf{c}(R)-\mathbf{c}(yN)+2\log_2|R|-\log_2|N|+(\log_2|N|)^2}
\le 2^{\mathbf{c}(R)-\frac{|N|}{2}+2\log_2|R|-\log_2|N|+(\log_2|N|)^2}.
\end{eqnarray*}

 In the second case, the number of inverse-closed subsets of $R\setminus (xN\cup yN\cup x^{-1}N\cup y^{-1}N)$ is $2^{\mathbf{c}(R)-2|N|}$ and the number of inverse-closed $f$-invariant subsets $T$ of $xN\cup yN\cup x^{-1}N\cup y^{-1}N$ is at most $2^{|N|}$, because once  $T\cap xN$ has been chosen we must have $T\cap x^{-1}N=(T\cap xN)^{-1}$, $T\cap yN=(T\cap xN)^f$ and $T\cap y^{-1}=((T\cap xN)^f)^{-1}$. Therefore
\begin{eqnarray*}
|\mathcal{B}_{ii}|&\le& |N|2^{(\log_2|N|)^2}|R/N|^2 2^{\mathbf{c}(R)-2|N|}\cdot 2^{|N|}= 
2^{\mathbf{c}(R)-|N|+2\log_2|R|-\log_2|N|+(\log_2|N|)^2}.
\end{eqnarray*}

In the third case, the number of inverse-closed subsets of $R\setminus (xN\cup yN\cup y^{-1}N)$ is $2^{\mathbf{c}(R)-\mathbf{c}(xN)-|N|}$ and the number of inverse-closed $f$-invariant subsets of $xN\cup yN\cup y^{-1}N$ is at most $2^{|N|}$, because once we choose a subset of $xN$ all the others are uniquely determined. Therefore
\begin{eqnarray*}
|\mathcal{B}_{iii}|&\le& |N|2^{(\log_2|N|)^2}|R/N|^22^{\mathbf{c}(R)-\mathbf{c}(xN)-|N|}\cdot 2^{|N|}\le 2^{\mathbf{c}(R)-\frac{|N|}{2}+2\log_2|R|-\log_2|N|+(\log_2|N|)^2}.
\end{eqnarray*}
The fourth case is similar to the third case and we have
$|\mathcal{B}_{iv}|\le 2^{\mathbf{c}(R)-\frac{|N|}{2}+2\log_2|R|-\log_2|N|+(\log_2|N|)^2}.$

The proof now follows by adding the contribution of the four sets $\mathcal{B}_i$, $\mathcal{B}_{ii}$, $\mathcal{B}_{iii}$ and $\mathcal{B}_{iv}$. 
\end{proof}

Our second bound is useful whenever $|N|$ grows with $|R|$.

\begin{proposition}\label{step3}
We have $|\mathcal{T}_N^2|\le 2^{\mathbf{c}(R)-\frac{|N|}{96}+(\log_2|N|)^2}$.
\end{proposition}
\begin{proof}
Given $S\in\mathcal{T}_N^2$, we let $G_S:=\norm {\Aut(\Cay(R,S))}N$. Given $f\in (G_S)_1$, we let $\iota_f:N\to N$ denote the automorphism induced by the action of conjugation of $f$ on $N$.  Let $f\in (G_S)_1\setminus\cent{(G_S)_1}N$ witnessing that $S\in\mathcal{T}_N^2$,  that is,  
\begin{itemize}
\item $N$ is neither an abelian group of exponent greater than $2$ nor a generalised dicyclic group, or
\item $N$ is an abelian group of exponent greater than $2$ and $\iota_f\ne \iota$ (where $\iota:N\to N$ is defined by $x^\iota=x^{-1}$, for every $x\in N$), or
\item $N=\Dic(A,x,y)\not\cong Q_8\times C_2^\ell$ and $\iota_f\ne \bar{\iota}_{A}$ (where $\bar{\iota}_{A}$ is given in Definition~\ref{defeq:2}), or
\item $N\cong Q_8\times C_2^\ell$ and $\iota_f\notin\{\bar{\iota}_i,\bar{\iota}_j,\bar{\iota}_k\}$ (where $\bar{\iota}_{i},\bar{\iota}_j,\bar{\iota}_k$ are given in Definition~\ref{defeq:2}).
\end{itemize}

In each of these cases, by Theorem~\ref{l:aut} applied to $N$, we deduce that the number of $f$-invariant inverse-closed subsets of $N$ is at most $2^{\mathbf{c}(N)-|N|/96}$. In particular, 
$$|\mathcal{T}_N^2|\le 2^{\mathbf{c}(R\setminus N)}\cdot 2^{\mathbf{c}(N)-\frac{|N|}{96}}|\Aut(N)|\le 2^{\mathbf{c}(R)-|N|/96+(\log|N|)^2},$$ 
where the first factor accounts for the number of inverse-closed subsets of $R\setminus N$, the second factor accounts for the number of inverse-closed $f$-invariant subsets of $N$ and the third factor accounts for the number of choices of $\iota_f$.
\end{proof}

For our third bound to be useful, we need $|N|/8$ to dominate $\log_2|R|$. In particular, it will be useful if $|N| \ge 9\log_2|R|$.

\begin{proposition}\label{step4}
We have $|\mathcal{T}_N^3|\le 2^{\mathbf{c}(R)-\frac{|N|}{8}+\log_2|R|+(\log_2|N|)^2}$.
\end{proposition}
\begin{proof}
Given $S\in\mathcal{T}_N^3$, we let $G_S:=\norm {\Aut(\Cay(R,S))}N$. Given any element $\kappa\in G_S$, we let $\iota_\kappa:N\to N$ denote the automorphism induced by the action of conjugation of $\kappa$ on $N$. Let $x\in R$ and let $f\in (G_S)_1\setminus\cent{(G_S)_1}N$ with $o(xN)>2$ and assume either
\begin{itemize}
\item $N$ is non-abelian, or
\item $N$ is  abelian  and there exists $n\in N$ with $(xn)^f\ne (xn)^{-1}$. 
\end{itemize}
As $S\notin \mathcal{T}_N^1$, we have $(xN)^f\in \{xN,x^{-1}N\}$ and hence $(xN)^f=x^{-1}N$. Thus $x^f=x^{-1}t$, for some $t\in N$. Observe that
\begin{equation}\label{virus1virus1}(xn)^f=x^{nf}=x^{f(f^{-1}nf)}=x^{-1}tn^{\iota_f}.
\end{equation}
From~\eqref{virus1virus1}, we deduce that we have at most $|\Aut(N)||N|\le 2^{(\log_2|N|)^2+\log_2|N|}$ choices for the restriction $f_{|xN}:xN\to x^{-1}N$ of $f$ to $xN$. Let $\beta:xN\to xN$ be the permutation obtained by composing first $f_{|xN}$ and then $\iota:x^{-1}N\to xN$, where $\iota$ is defined by $(x^{-1}n)^\iota=(x^{-1}n)^{-1}=n^{-1}x$ $\forall n\in N$. Thus, from~\eqref{virus1virus1}, we have
$$(x n)^\beta=((xn)^f)^{\iota}=(x^{-1}tn^{\iota_f})^{-1}=(n^{-1})^{\iota_f}t^{-1}x=x(n^{-1})^{\iota_{fx}}(t^{-1})^{\iota_x}.$$
Since $S$ is inverse-closed and $f$-invariant, we deduce that $S\cap xN$ is $\beta$-invariant.

Let $\beta':N\to N$ the permutation defined by $n^{\beta'}=(n^{-1})^{\iota_{fx}}(t^{-1})^{\iota_x}$ $\forall n\in N$. An easy computation reveals that $n\in \mathrm{Fix}_N(\beta')$ if and only if $n^{-1}(n^{-1})^{\iota_{fx}}=t^{\iota_x}$. In particular, we are in the position to apply Lemma~\ref{lemma:icecream} (with $\alpha=\iota_{fx}$ and with the element $t$ there replaced by $t^{\iota_x}$ here ). From Lemma~\ref{lemma:icecream}, we have two possibilities:
\begin{itemize}
\item $|\mathrm{Fix}_{N}(\beta')|\le 3|N|/4$, or
\item $N$ is abelian, $t=1$ and $n^{\iota_{fx}}=n^{-1}$ $\forall n\in N$.
\end{itemize}
If the second possibility holds, then $N$ is abelian, $\iota_f=\iota_{x^{-1}}\iota$ and from~\eqref{virus1virus1} we get $(xn)^f=x^{-1}(n^{\iota_{x^{-1}}})^{-1}=x^{-1}xn^{-1}x^{-1}=(xn)^{-1}$ for every $n \in N$; however, this contradicts the fact that $S\in\mathcal{T}_N^3$. Therefore, $|\mathrm{Fix}_N(\beta')|\le 3|N|/4$.

The definition of $\beta'$ and the previous paragraph yield that $\beta$ has at most
$$\frac{3|N|}{4}+\frac{|N|-\frac{3|N|}{4}}{2}=\frac{7|N|}{8}$$
orbits. Since $S\cap xN$ is $\beta$-invariant, the number of choices for $S\cap xN$ is at most $2^{7|N|/8}$. By taking in account the contributions of $\iota_f$, $xN$ and $t$, we obtain
$$|\mathcal{T}_N^3|\le 2^{(\log_2|N|)^2}|N||R/N|2^{\mathbf{c}(R\setminus (xN\cup x^{-1}N))}2^{\frac{7|N|}{8}}=2^{\mathbf{c}(R)-\frac{|N|}{8}+\log_2|R|+(\log_2|N|)^2}.\qedhere$$
\end{proof}

Our fifth bound is again useful whenever $|N|$ grows with $|R|$.

\begin{proposition}\label{step5}
We have 
 $|\mathcal{T}_N^4|\le 2^{\mathbf{c}(R)-\frac{|N|}{24}+\log_2|R|+2}.$

\end{proposition}
\begin{proof}
Given $S\in\mathcal{T}_N^4$, we let $G_S:=\norm {\Aut(\Cay(R,S))}N$. Given any element $\kappa\in G_S$, we let $\iota_\kappa:N\to N$ denote the automorphism induced by the action of conjugation of $\kappa$ on $N$. Let $\gamma\in R$ and let $f\in (G_S)_1$ with  $\gamma^f\notin\{\gamma,\gamma^{-1}\}$.
Furthermore, if possible we will choose $\gamma$ so that $o(\gamma)=2$. Therefore we may assume that if $o(\gamma)\neq 2$, then $(\gamma')^f=\gamma'$ for every $\gamma' \in R$ with $o(\gamma')=2$. (This will be important when we apply Lemma~\ref{lemma:aux2}.)

 We now consider various possibilities depending on the behaviour of $\gamma N$, but first, we state the fact that the set $S$ does not lie in $\mathcal{T}_N^2$ in a manner tailored to our current needs:

\smallskip

\noindent\textsc{Case A} $(G_S)_1=\cent {(G_S)_1}N$, or

\noindent\textsc{Case B} $N$ is  abelian of exponent greater than $2$ and, for every $f\in (G_S)_1\setminus \cent {(G_S)_1}N$ we have $n^f=n^{-1}$ $\forall n\in N$, so $|(G_S)_1:\cent {(G_S)_1}N|=2$, or 

\noindent\textsc{Case C} $N=\Dic(A,y,x)\not\cong Q_8\times C_2^\ell$, for every $f\in (G_S)_1\setminus \cent {(G_S)_1}N$, $A=\cent N f$  and the automorphism $\iota_f$ induced by $f$ on $N$ is $\bar{\iota}_{A}$, or

\noindent\textsc{Case D} $N=Q_8\times C_2^\ell$, $|(G_S)_1:\cent {(G_S)_1}N|\in \{2,4\}$, for every $f\in (G_S)_1\setminus \cent {(G_S)_1}N$,  the automorphism $\iota_f$ induced by $f$ on $N$ is one of $\bar{\iota}_{i}$, $\bar{\iota}_{j}$, $\bar{\iota}_{k}$.

\smallskip

In particular, in cases B, C, and D, $n^{\iota_f}\in \{n,n^{-1}\}$ $\forall n\in N$.

\smallskip

Suppose that $\gamma\in N$. Since $1^f=1$ and since $f$  normalises $N$, we have $\gamma^f=\gamma^{\iota_f}\in \{\gamma,\gamma^{-1}\}$. For the rest of the proof, we may suppose that $\gamma\notin N$. Since $S\notin\mathcal{T}_N^1$, we have $(\gamma N)^f\in \{\gamma N,\gamma^{-1}N\}$. 

Suppose $(\gamma N)^f\ne \gamma N$. Since $S\notin\mathcal{T}_N^3$, we have $(\gamma n)^f=(\gamma n)^{-1}$ $\forall n\in N$ and hence, in particular, $\gamma^f=\gamma^{-1}$. Therefore, for the rest of the proof, we may suppose that $(\gamma N)^f=\gamma N$.

Since $\gamma^f\in \gamma N$, there exists $t\in N$ with $\gamma^f=\gamma t$. Now,
\begin{equation}\label{virus119}
(\gamma n)^f=\gamma^{nf}=\gamma^{f\cdot f^{-1}n f}=(\gamma t)^{n^{\iota_f}}=\gamma tn^{\iota_f},\quad \forall n\in N.
\end{equation}

Suppose now that $\gamma N\neq \gamma^{-1}N$. Then $(\gamma n)^{-1} \in \gamma^{-1}N \neq \gamma N$ for every $n \in N$. Since $(\gamma N)^f =\gamma N$, we cannot have $(\gamma n)^{-1}=(\gamma n)^f$. Thus the orbits of $f$ fuse orbits of the inverse map on $\gamma N \cup \gamma^{-1} N$ unless $f$ has any fixed points on $\gamma N$; that is, unless (using $(\gamma n)^f=\gamma n$ in~\eqref{virus119}) there exists some $n \in N$ with
\begin{equation}\label{fixed-points}
t=n (n^{\iota_f})^{-1}.
\end{equation}
Note that~\eqref{virus119} with $n=1$ together with $\gamma^f \neq \gamma$ implies that $t \neq 1$. So applying Lemma~\ref{lemma:gelato} to $N$ with $n^{\alpha}=n^{\iota_f}$ implies that the number of fixed points of $f$ in $\gamma N$ is at most $3|N|/4$. Therefore the action of $f$ on $\gamma N$ together with the action of the inverse map on $\gamma N \cup \gamma^{-1}N$ results in at least $|N|/4$ orbits of  length at least $4$ and all other orbits having length at least $2$. So when $f_{|\gamma N}$ is given, the number of choices for $S \cap (\gamma N \cup \gamma^{-1} N)$ is at most $2^{(3|N|/4)/2+(|N|/4)/4}=2^{7|N|/16}$. Therefore 
 $$|\mathcal{T}_N^4|\le 3|N||R/N|2^{\mathbf c(R)-\mathbf c(\gamma N \cup \gamma^{-1} N)}2^{7|N|/16} \le 2^{2+\log_2|R|} 2^{\mathbf c(R)-|N|+7|N|/16}=2^{\mathbf c(R)-9|N|/16+\log_2 |R|+2}$$
(where $3|N|$ is the number of choices for the restriction $f_{\gamma N}:\gamma N\to \gamma N$ of $f$ to $\gamma N$, and $|R/N|$ is the number of choices for $\gamma N\in R/N$).

For the remainder of the proof we may assume that $\gamma N=\gamma^{-1} N$, meaning that $N$ is an index-$2$ subgroup of $\langle \gamma, N\rangle$.

Suppose that $f\in \cent {G_S}N$. Then,~\eqref{virus119} becomes $n^f=n$ and $(\gamma n)^f=\gamma tn$, $\forall n\in N$. When $f_{|\gamma N}$ is given, from Lemma~\ref{lemma:aux1}, we deduce that the number of choices for $S\cap \langle \gamma, N\rangle$ is at most $2^{\mathbf{c}(\langle \gamma,N\rangle)-\frac{|N|}{16}}$ (recall that the other cases cannot arise since $\gamma^f \notin \{\gamma, \gamma^{-1}\}$). Therefore
 $$|\mathcal{T}_N^4|\le|N||R/N|2^{\mathbf{c}(R)-\mathbf{c}(\langle\gamma,N\rangle)}2^{\mathbf{c}(\langle\gamma,N\rangle)-\frac{|N|}{16}}\le 2^{\mathbf{c}(R)-\frac{|N|}{16}+\log_2|R|}.$$
	(where $|N|$ is the number of choices for the restriction $f_{\gamma N}:\gamma N\to \gamma N$ of $f$ to $\gamma N$, and $|R/N|$ is the number of choices for $\gamma N\in R/N$).
Therefore, for the rest of the proof we may suppose that $f\notin\cent {G_S}N$. In particular, only Case~B,~C or~D may arise.

Suppose that Case~B holds. Then,~\eqref{virus119} becomes $n^f=n^{-1}$ and $(\gamma n)^f=\gamma tn^{-1}$, $\forall n\in N$, so $n^{\iota_f}=n^{-1}$ for every $n \in N$.  As already observed at the beginning, if $\gamma$ cannot be chosen with $o(\gamma)=2,$ then for every $\gamma n\in \gamma N$ with $o(\gamma n)=2,$ we have $(\gamma n)^f=\gamma n.$ So we may apply Lemma~\ref{lemma:aux2} with $f_{|\langle\gamma,N\rangle}$ taking the role of $\alpha_t$.

When $f_{|\gamma N}$ is given, from Lemma~\ref{lemma:aux2}, we deduce that the number of choices for $S\cap \langle \gamma, N\rangle$ is at most $2^{\mathbf{c}(\langle \gamma,N\rangle)-\frac{|N|}{24}}$ (again, the other cases cannot arise since $\gamma^f \notin \{\gamma, \gamma^{-1}\}$). Therefore

 $$|\mathcal{T}_N^4|\le |N||R/N|2^{\mathbf{c}(R)-\mathbf{c}(\langle\gamma,N\rangle)}2^{\mathbf{c}(\langle\gamma,N\rangle)-\frac{|N|}{24}}\le 2^{\mathbf{c}(R)-\frac{|N|}{24}+\log_2|R|}$$
 (again, $|N|$ is the number of choices for the restriction $f_{\gamma N}:\gamma N\to \gamma N$ of $f$ to $\gamma N$, and $|R/N|$ is the number of choices for $\gamma N\in R/N$).

Cases~C and~D can be dealt with simultaneously. Here,~\eqref{virus119} becomes $n^f=n^{\bar{\iota}_A}$ and $(\gamma n)^f=\gamma tn^{\bar{\iota}_A}$, $\forall n\in N$. When $f_{|\gamma N}$ is given, from Lemma~\ref{lemma:aux3}, we deduce that the number of choices for $S\cap \langle\gamma, N\rangle$ is at most $2^{\mathbf{c}(\langle \gamma,N\rangle)-\frac{|N|}{24}}$ (again, the other cases cannot arise since $\gamma^f \notin \{\gamma, \gamma^{-1}\}$). Therefore

 $$|\mathcal{T}_N^4|\le 3|N||R/N|2^{\mathbf{c}(R)-\mathbf{c}(\langle\gamma,N\rangle)}2^{\mathbf{c}(\langle\gamma,N\rangle)-\frac{|N|}{24}}\le 2^{\mathbf{c}(R)-\frac{|N|}{24}+\log_2|R|+2}$$
(where $3|N|$ is the number of choices for the restriction $f_{\gamma N}:\gamma N\to \gamma N$ of $f$ to $\gamma N$, and $|R/N|$ is the number of choices for $\gamma N\in R/N$).

\end{proof}

Combining these results, we are able to bound $|\mathcal T_N|$.


\begin{proof}[Proof of Theorem~$\ref{main1}$]
Since the initial statement excludes $\mathcal S_N^1$, its proof follows by adding the bounds produced in Propositions~\ref{step2},~\ref{step3},~\ref{step4} and~\ref{step5} for $|\mathcal T_N^i|$, for each $1 \le i \le 4$. If we drop the condition $R=\norm  {\Aut(\Cay(R,S))}R$, then we must also add the bound produced in 
Proposition~\ref{propo:aut} for $\mathcal{S}_N^1$ (which has no effect on the bound we have given). Using Proposition~\ref{propo:aut} requires us to exclude groups that are either abelian of exponent greater than $2$, or generalised dicyclic. 
\end{proof}

\section{Groups with a ``small" normal subgroup}\label{sec:Nsmall}

We begin this section of our paper with a counting result that we will need. The flavour of this result is quite distinct from most of the rest of the paper, and we have placed it in advance of the introduction of the notation and situational information that we will be using for the rest of this section.

\begin{lemma}\label{lemmanew}
Let $X$ be a set and let $f$ and $g$ be permutations of $X$. Then either
\begin{enumerate}
\item\label{enumerate1} $|\{S\subseteq X\mid |S\cap S^f|=|S\cap S^g|\}|\le \frac{3}{4}\cdot 2^{|X|}$, or
\item\label{enumerate2} there exists a subset $I\subseteq X$  such that
\begin{itemize}
\item $I$ is $f$- and $g$-invariant (that is, $I^f=I$ and $I^g=I$),
\item $f_{|I}=g_{|I}$, 
\item $f_{|X\setminus I}=(g^{-1})_{|X\setminus I}$.
\end{itemize}
\end{enumerate}
\end{lemma}
\begin{proof}
We denote by $F$ and by $G$ the permutation matrices of $f$ and $g$, respectively. Therefore, $F$ and $G$ are $|X|\times |X|$-matrices with $\{0,1\}$ entries, with rows and columns indexed by the set $X$ and such that
\begin{align*}
F_{x,y}=
\begin{cases}
1&\textrm{if }x^f=y,\\
0&\textrm{otherwise},
\end{cases}
\qquad
G_{x,y}=
\begin{cases}
1&\textrm{if }x^g=y,\\
0&\textrm{otherwise}.
\end{cases}
\end{align*}
Let $A:=F-G$. For any $S\subseteq X$, let $\delta_S\in \mathbb{Z}^X$ be the ``indicator'' vector of the set $S$, that is,  
\begin{align*}
(\delta_S)_x:=
\begin{cases}
1&\textrm{if }x\in S,\\
0&\textrm{otherwise}.
\end{cases}
\end{align*}
Finally, let $\langle\cdot,\cdot\rangle:\mathbb{Q}^X\times \mathbb{Q}^X\to\mathbb{Q}$ be the standard scalar product and let $(e_x)_{x\in X}$ be the canonical basis of $\mathbb{Q}^X$.

With the  notation above, for every subset $S$ of $X$, we have $$|S\cap S^f|=\langle \delta_S,F\delta_S\rangle\,\,\textrm{ and }\,\,|S\cap S^g|=\langle \delta_S,G\delta_S\rangle.$$ 
Therefore,
$$\{S\subseteq X\mid |S\cap S^f|=|S\cap S^g|\}=
\{S\subseteq X\mid \langle \delta_S,F\delta_S\rangle=\langle \delta_S,G\delta_S\rangle\}
=\{S\subseteq X\mid \langle \delta_S,A\delta_S\rangle=0\}.$$
For simplicity, we write  $\Delta:\{0,1\}^X\to \mathbb{Q}$ for the mapping defined by $\delta\mapsto\Delta(\delta)=\langle\delta,A\delta\rangle$, for every $\delta\in \{0,1\}^X$.

Suppose first that, there exist $i,j\in X$ with $i\ne j$ and $A_{i,j}+A_{j,i}\ne 0$. Fix $\delta_x\in \{0,1\}$ arbitrarily for every $x\in X\setminus \{i,j\}$, and let $\eta:=\sum_{x\in X\setminus\{i,j\}}\delta_x e_x$. By restricting $\Delta$, we define the function $\Delta': \{0,1\}\times \{0,1\}\to \mathbb{Q}$  by setting
\begin{eqnarray*}
(\delta_i,\delta_j)&\mapsto &\Delta'(\delta_i,\delta_j):=\Delta(\eta+\delta_ie_i+\delta_j e_j)=\langle \eta+\delta_ie_i+\delta_j e_j,A(\eta+\delta_ie_i+\delta_j e_j)\rangle\\
&&=\langle \eta,A\eta\rangle+\delta_i\langle \eta,Ae_i\rangle+\delta_j\langle \eta,Ae_j\rangle+
\delta_i\langle e_i,A\eta\rangle+\delta_j\langle e_j,A\eta\rangle\\
&&+\delta_i^2\langle e_i,Ae_i\rangle
+\delta_j^2\langle e_j,Ae_j\rangle
+\delta_i\delta_j\langle e_i,Ae_j\rangle
+\delta_i\delta_j\langle e_j,Ae_i\rangle.
\end{eqnarray*}
A computation yields 
$$\Delta'(0,0)+\Delta'(1,1)-\Delta'(1,0)-\Delta'(0,1)=A_{i,j}+A_{j,i}\ne 0.$$
In particular, at least one out of the four choices $(\delta_i,\delta_j)\in \{(0,0), (0,1), (1,0), (1,1)\}$ gives rise to a non-zero value for $\Delta(\eta+\delta_ie_i+\delta_je_j)$. Therefore, 
for every choice of  $\delta_x\in \{0,1\}$ with $x\in X\setminus \{i,j\}$, we have at most three more choices for $\delta_i,\delta_j\in\{0,1\}$, for constructing a vector $\delta\in \{0,1\}^X$ with $\Delta(\delta)=0$. Therefore,
$$\{S\subseteq X\mid \langle \delta_S,A\delta_S\rangle=0\}\le 2^{|X|-2}\cdot 3=\frac{3}{4}\cdot 2^{|X|}$$
and~\eqref{enumerate1} holds.

Suppose that, for every $i,j\in X$ with $i\ne j$, we have $A_{i,j}+A_{j,i}=0$. In this case, $$\delta:=\sum_{x\in X}\delta_xe_x\mapsto \Delta(\delta)=\sum_{x\in X}A_{x,x}\delta_x.$$
If $A_{i,i}\ne 0$ for some $i\in X$, then we may use the same argument as in the previous paragraph by fixing $\delta_x\in \{0,1\}$ arbitrarily for every $x\in X\setminus \{i\}$, and by considering the restriction of
$\Delta$ as a function $\Delta'(\delta_i)$ of $\delta_i\in \{0,1\}$ only. In this case, we see that one of the two choices for $\delta_i$ gives rise to  a vector $\delta\in \{0,1\}^X$ with $\Delta(\delta)=0$. Therefore, 
$$\{S\subseteq X\mid \langle \delta_S,A\delta_S=0\rangle\}\le 2^{|X|-1}\cdot 1\le \frac{3}{4}2^{|X|}$$
and~\eqref{enumerate1} holds.

Suppose now that, for every $i,j\in X$ with $i\ne j$, we have $A_{i,j}+A_{j,i}=0$ and $A_{i,i}=0$, that is, $A$ is antisymmetric. Let $I$ be the set of rows of $A=F-G$ that are zero. From the fact that $A$ is antisymmetric and from the definition of $A$, we see that $I$ is $f$- and $g$-invariant, $f_{|I}=g_{|I}$ and $f_{|X\setminus J}=g^{-1}_{|X\setminus J}$. In particular,~\eqref{enumerate2} holds. 
\end{proof}

Incidentally, we observe that, if~\eqref{enumerate2} holds in Lemma~\ref{lemmanew}, then $|S\cap S^f|=|S\cap S^g|$, for every subset $S$ of $X$. We find this quite interesting on its own. For instance, $f:=(1\,2\,3\,4\,5)(6\,7\,8)(9\,10\,11\,12)$ and $g:=(1\,5\,4\,3\,2)(6\,7\,8)(9\,12\,11\,10)$ have the property that $|S\cap S^f|=|S\cap S^g|$, for every subset $S$ of $\{1,\ldots,12\}$.

\medskip

\subsection{Specific notation}\label{notation specific}
Henceforth, let $R$ be a finite group of order $r$ acting regularly on itself via the right regular representation: here, we identify the elements of $R$ as permutation in $\Sym(R)$. Let $N$ denote a non-identity proper normal subgroup of $R$. We let $b:=|R:N|$ and we let $\gamma_1,\ldots,\gamma_b$ be coset representatives of $N$ in $R$. Moreover, we choose $\gamma_1:=1$ to be the identity in $R$. Observe that  $R/N$ defines a group structure on $\{1,\ldots,b\}$ by setting $ij=k$ for every $i,j,k\in \{1,\ldots,b\}$  with $\gamma_i N\gamma_j N=\gamma_k N$. 

Write $v_0:=1$ where $v_0$ has to be understood as a point in the set $R$.
For each $i\in \{1,\ldots,b\}$, set $\mathcal{O}_i:={v_0}^{\gamma_iN}=\gamma_i N=N\gamma_i$. Observe that the $\mathcal{O}_i$s are the orbits of $N$ on $R$, the group $N$ acts regularly on $\mathcal{O}_i$ and $|\mathcal{O}_i|=|N|$.

For an inverse-closed subset $S$ of $R$, we let $\Cay(R,S)$ be the Cayley graph of $R$ with connection set $S$, and we denote by $F_S$ the largest subgroup of $\Aut(\Cay(R,S))$ under which each orbit of $N$ is invariant. In symbols we have
$$F_S:=\{g\in\Aut(\Cay(R,S))\mid \mathcal{O}_i^g=\mathcal{O}_i,\textrm{ for each }i\in \{1,\ldots,b\}\}.$$
(The subscript $S$ in $F_S$ will make some of the later notation cumbersome to use, but it constantly emphasizes that the definition of ``$F$'' depends on $S$.) Similarly, we define
$$B_S:=F_S\cap \norm {\Aut(\Cay(R,S))}N.$$


As above, let $S$ be an inverse-closed subset of $R$. For a vertex $u$ of $\Cay(R,S)$ in $\mathcal{O}_i$, 
\begin{center}
let $\sigma(S,u,j)$ denote  the neighbours of $v_0$ and  $u$ lying in $\mathcal{O}_j$.
\end{center} 
See Figure~\ref{fig:23}. It is clear that
$$\sigma(S,u,j)=S\cap S^{g_u}\cap \mathcal{O}_j=(S\cap \mathcal{O}_j)\cap S^{g_u}=S_j\cap S^{g_u},$$
where $g_u\in R$ with $v_0^{g_u}=u$. Since $u\in\mathcal{O}_i$, we have $u=v_0^{\gamma_ik_u}$, for some $k_u\in N$. In particular, $g_u=\gamma_i k_u$.
Let $s\in S$ with $s^{g_u}\in S_j$. Then $s^{g_u}\in \mathcal{O}_j=v_0^{\gamma_jN}=v_0^{N\gamma_j}$ and $s^{g_u\gamma_j^{-1}}\in v_0^{N}=\mathcal{O}_1$. Since $g_u$ maps the element $v_0$ of $\mathcal{O}_1$ to the element $u$ of $\mathcal{O}_i$, we see that $g_u\in \gamma_iN$ and $s\in \mathcal{O}_1^{\gamma_jg_u^{-1}}=v_0^{N\gamma_j\gamma_i^{-1}}=v_0^{\gamma_j\gamma_i^{-1}N}=\mathcal{O}_{ji^{-1}}$. This shows 
\begin{equation}\label{eq6-----}
\sigma(S,u,j)=S_j\cap S_{ji^{-1}}^{g_u}=S_j\cap S_{ji^{-1}}^{\gamma_i k_u}.
\end{equation}

For two distinct vertices $u,v\in \mathcal{O}_i$ and $j\in \{1,\ldots,b\}$, let $$\Psi(\{u,v\},j):=\{S\subseteq R\mid S=S^{-1} \textrm{ and }|\sigma(S,u,j)|= |\sigma(S,v,j)|\}.$$

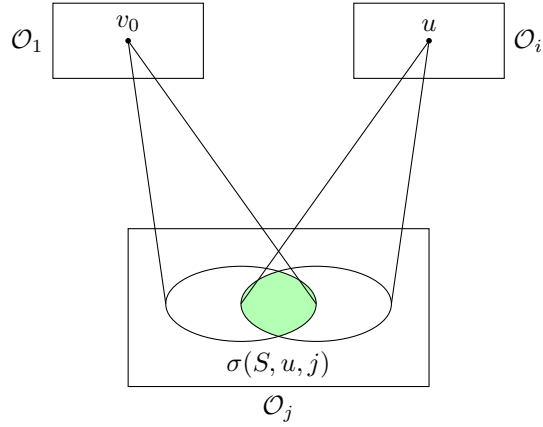
\begin{figure}
\begin{tikzpicture}
\draw (-3,0) rectangle (-1,1);
\draw (1,0) rectangle (3,1);
\draw [fill] (-2,.5) circle [radius=.03];
\draw [fill] (2,.5) circle [radius=.03];
\node [above] at (-2,.5) {$v_0$};
\node [above] at (2,.5) {$u$};
\node [left] at (-3,.5) {$\mathcal{O}_1$};
\node [right] at (3,.5) {$\mathcal{O}_i$};
\node [below] at (0,-4.1) {$\mathcal{O}_j$};
\node [below] at (0,-3.5){$\sigma(S,u,j)$};
\draw (-2,-2) rectangle (2,-4.1);
\begin{scope}
\clip {(-.5,-3) ellipse (1cm and .5cm)};
\fill[green, opacity=0.3] {(.5,-3) ellipse (1cm and .5cm)};
\end{scope}
\draw (-.5,-3) ellipse (1cm and .5cm);
\draw (.5,-3) ellipse (1cm and .5cm);
\draw (-2,.5) -- (-1.5,-3);
\draw (-2,.5) -- (.5,-3);
\draw (2,.5) -- (-.5,-3);
\draw (2,.5) -- (1.5,-3);
\end{tikzpicture}
\caption{The definition of $\sigma(S,u,j)$}\label{fig:23}
\end{figure}

In the results that follow, we use the notation that we have established here. Our aim with the next few results is to show that $|\Psi(\{u,v\},j)|$ is at most $\frac{3}{4}\cdot 2^{\mathbf c(R)}$. This will subsequently be used to bound the number of graphs admitting automorphisms that fix the vertex $1$ and also fix each $\mathcal O_i$ setwise while mapping $u$ to $v$. We generally end up with some other possibilities that we gradually eliminate by introducing additional assumptions.

\begin{proposition}\label{propo:1}Let $i\in \{2,\ldots,b\}$, let $u$ and $v$ be two distinct vertices in $\mathcal{O}_i$ and let $j\in \{1,\ldots,b\}\setminus\{1,i\}$. Then, one of the following holds:
\begin{enumerate}
\item\label{propo:eq1}$|\Psi(\{u,v\},j)|\le \frac{3}{4}\cdot 2^{\mathbf{c}(R)}$,
\item\label{propo:eq2}$j^2=i$, $\gamma_i=\gamma_j^2\bar{y}$ for some $\bar{y}\in N$, $k_u=\bar{y}^{-1}\gamma_j^{-1}\bar{y}k_v\gamma_j$, $k_v=\bar{y}^{-1}\gamma_j^{-1}\bar{y}k_u\gamma_j$ and $\gamma_ik_v,\gamma_ik_u$ centralize $N$, 
\item\label{propo:eq3}$o(ji^{-1})>2$, $o(j)=2$, $o(i)$ is even, $o(\gamma_j)=4$, $\gamma_j^2=k_v^{-1}k_u=k_u^{-1}k_v$, $N$ is abelian  and  $y^{\gamma_j}= y^{-1}$ for every $y\in N$,
\item\label{propo:eq3dash}$o(ji^{-1})=2$, $o(j)>2$, $o(i)$ is even, $o(\gamma_{ji^{-1}})=4$, $\gamma_{ji^{-1}}^2=k_v^{-1}k_u=k_u^{-1}k_v$, $N$ is abelian  and  $y^{\gamma_{ji^{-1}}}= y^{-1}$ for every $y\in N$,
\item\label{propo:eq4}$o(ji^{-1})=o(j)=2$ 
\end{enumerate}
\end{proposition}
\begin{proof}
We divide the proof in various cases.

\smallskip

\noindent\textsc{Case  } $j^2= i$.

\smallskip

\noindent Observe that, if $S\subseteq R$ is inverse-closed, then $S_{j^{-1}}=S_j^{-1}$. As $ji^{-1}=j^{-1}$, from~\eqref{eq6-----}, we obtain  
\begin{eqnarray}\label{eq6+++++xx}
|\sigma(S,u,j)|=
|S_{ji^{-1}}\cap S_j^{k_u^{-1}\gamma_i^{-1}}|=|S_{j^{-1}}\cap S_j^{k_u^{-1}\gamma_i^{-1}}|,\\\nonumber
|\sigma(S,v,j)|=|S_{ji^{-1}}\cap S_j^{k_v^{-1}\gamma_i^{-1}}|=|S_{j^{-1}}\cap S_j^{k_v^{-1}\gamma_i^{-1}}|.
\end{eqnarray}
Let $\iota:N\gamma_j^{-1}\to N\gamma_j$ be the mapping defined by $x\mapsto x^{\iota}=x^{-1}$ for every $x\in N\gamma_j^{-1}$ and set 
\begin{center}$f:=k_u^{-1}\gamma_i^{-1}\iota:N\gamma_j\to N\gamma_j$ and $g:=k_v^{-1}\gamma_{i}^{-1}\iota:N\gamma_j\to N\gamma_j$
\end{center} as permutations of $N\gamma_j$. Now,~\eqref{eq6+++++xx} yields
\begin{eqnarray}\label{eq6+++++x}
|\sigma(S,u,j)|=
|S_{j}^\iota\cap S_j^{k_u^{-1}\gamma_i^{-1}}|=|S_{j}\cap S_j^{k_u^{-1}\gamma_i^{-1}\iota}|=|S_{j}\cap S_j^{f}|,\\\nonumber
|\sigma(S,v,j)|=
|S_{j}^\iota\cap S_j^{k_v^{-1}\gamma_i^{-1}}|=|S_{j}\cap S_j^{k_v^{-1}\gamma_i^{-1}\iota}|=|S_{j}\cap S_j^{g}|.
\end{eqnarray}
From~\eqref{eq6+++++x}, we see that we are in the position to apply Lemma~\ref{lemmanew} with $X:=\mathcal{O}_j$. If Lemma~\ref{lemmanew}~\eqref{enumerate1} holds, then 
the number of  subsets $S_j\subseteq\mathcal{O}_j$ satisfying~\eqref{eq6+++++x} is at most $\frac{3}{4}\cdot 2^{|N|}$. Therefore
$$|\Psi(\{u,v\},j)|\le\frac{3}{4}2^{|N|}\cdot 2^{\mathbf{c}(R)-|N|},$$
observe that $2^{\mathbf{c}(R)-|N|}$ counts the number of inverse-closed subsets of $R\setminus (\gamma_jN\cup \gamma_j^{-1}N)$. Thus~\eqref{propo:eq1} is proved in this case.

Therefore, we may suppose that Lemma~\ref{lemmanew}~\eqref{enumerate2} holds. Therefore, there exists an $f$- and $g$-invariant subset $I$ of $N\gamma_j$ such that $f_{|I}=g_{|I}$ and $f_{|N\gamma_j\setminus I}=(g^{-1})_{|N\gamma_j\setminus I}$. If $I\ne\emptyset$, then there exists $x\in I$ and hence
$$x^{k_u^{-1}\gamma_i^{-1}\iota}=x^f=x^g=x^{k_v^{-1}\gamma_i^{-1}\iota}.$$
Simplifying $\iota$ and $\gamma_i^{-1}$, we obtain $xk_u^{-1}=xk_v^{-1}$.
This yields $k_u=k_v$, contradicting the fact that $u\ne v$. Therefore $I=\emptyset$ and hence $f=g^{-1}$.

This means that, for every $x\in N\gamma_j$, we have
\begin{align}\label{thisequation}
x&=x^{fg}=x^{k_u^{-1}\gamma_i^{-1}\iota k_v^{-1}\gamma_i^{-1}\iota}=(xk_u^{-1})^{\gamma_i^{-1}\iota k_v^{-1}\gamma_i^{-1}\iota}
=(xk_u^{-1}\gamma_i^{-1})^{\iota k_v^{-1}\gamma_i^{-1}\iota}=(\gamma_i k_u x^{-1})^{k_v^{-1}\gamma_i^{-1}\iota}\\\nonumber
&=(\gamma_i k_u x^{-1}k_v^{-1})^{\gamma_i^{-1}\iota}=(\gamma_i k_u x^{-1}k_v^{-1}\gamma_i^{-1})^\iota=
\gamma_i k_v x k_u^{-1}\gamma_i^{-1}.
\end{align}
As $j^2=i$, there exists $\bar{y}\in N$ with 
\begin{equation}\label{thisthisequation}\gamma_i=\gamma_j^2\bar{y}.
\end{equation}
When $x=\gamma_j$,~\eqref{thisequation} gives $$\gamma_i^{-1}\gamma_j\gamma_i=k_v \gamma_j k_u^{-1}.$$Using~\eqref{thisthisequation}, we obtain $\gamma_i^{-1}\gamma_j\gamma_i=\bar{y}^{-1}\gamma_j\bar{y}$. Therefore
\begin{equation}\label{thisthisthisequation}
k_u=\bar{y}^{-1}\gamma_j^{-1}\bar{y}k_v\gamma_j.
\end{equation} From~\eqref{thisequation},~\eqref{thisthisequation} and~\eqref{thisthisthisequation}, we obtain
$$x=\gamma_i k_v x\gamma_j^{-1}k_v^{-1}\bar{y}^{-1}\gamma_j^{-1},\qquad\forall x\in N\gamma_j.$$
By writing $x=y\gamma_j$ with $y\in N$, we deduce
$$y=(\gamma_ik_v)y(\gamma_i k_v)^{-1},\qquad \forall y\in N.$$
Since $y$ is an arbitrary element of $N$, we get that $\gamma_i k_v$ centralizes $N$. From this and from~\eqref{thisthisequation} and~\eqref{thisthisthisequation} we see that~\eqref{propo:eq2} holds.~$_\blacksquare$ 

\smallskip

For the rest of the proof, we suppose $j^2\ne i$. From~\eqref{eq6-----}, we obtain  
\begin{equation}\label{eq6+++++}
|\sigma(S,u,j)|=|S_{ji^{-1}}\cap S_j^{k_u^{-1}\gamma_i^{-1}}|\quad\textrm{and}\quad |\sigma(S,v,j)|=|S_{ji^{-1}}\cap S_j^{k_v^{-1}\gamma_i^{-1}}|.
\end{equation}
From~\eqref{eq6+++++}, we see that the condition ``$|\sigma(S,u,j)|=|\sigma(S,v,j)|$" imposes no constraint on $S_x$, for $x\notin \{j,ji^{-1},j^{-1},(ji^{-1})^{-1}\}$. Observe that
$$\{j,j^{-1}\}\ne \{ji^{-1},(ji^{-1})^{-1}\},$$
because we are assuming $j^2\ne i$. As usual, there is one implicit condition on the set $S$: it is inverse-closed. This suggests a natural decomposition of $S$.  Write $R_{j,i}:=\gamma_jN\cup \gamma_j^{-1}N\cup\gamma_{ji^{-1}}N\cup \gamma_{ji^{-1}}^{-1}N$ and $R_{j,i}^c:=R\setminus R_{j,i}$. We have
\begin{equation}\label{eqcases1}
\mathbf{c}(R_{j,i})=
\begin{cases}
2|N|&\textrm{if }o(j)>2 \textrm{ and }o(ji^{-1})>2,\\
|N|+\mathbf{c}(\gamma_jN)&\textrm{if }o(j)=2 \textrm{ and }o(ji^{-1})>2,\\
|N|+\mathbf{c}(\gamma_{ji^{-1}}N)&\textrm{if }o(j)>2 \textrm{ and }o(ji^{-1})=2,\\
\mathbf{c}(\gamma_jN)+\mathbf{c}(\gamma_{ji^{-1}}N)&\textrm{if }o(j)=o(ji^{-1})=2.\\
\end{cases}
\end{equation}
Observe that $R_{j,i}$ and $R_{j,i}^c$ are inverse-closed; moreover, we may write $S:=S_{j,i}\cup S_{j,i}^c$, where $S_{j,i}\subseteq R_{j,i}$ and $S_{j,i}^c\subseteq R_{j,i}^c$. 

Using this decomposition of the inverse-closed subsets, we get 
$$|\Psi(\{u,v\},j)|=A\cdot 2^{B},$$
where $2^B$ is the number of inverse-closed subsets $S_{j,i}^c\subseteq R_{j,i}^c$ and $A$ is the number of inverse-closed subsets $S_{j,i}\subseteq R_{j,i}$ such that 
$|S_{ji^{-1}}\cap S_j^{k_u^{-1}\gamma_i^{-1}}|=|S_{ji^{-1}}\cap S_j^{k_v^{-1}\gamma_i^{-1}}|$ with $S:=S_{j,i}\cup S_{j,i}^c$. 
We deduce
\begin{equation}\label{eqcases}
B=\mathbf{c}(R)-\mathbf{c}(R_{j,i}).
\end{equation}

\smallskip

\noindent\textsc{Case }$o(ji^{-1})>2$.

\smallskip

\noindent When $o(j)>2$, let $t_1$ be the number of subsets $S_j$ of $\mathcal{O}_j$ with $S_j^{k_u^{-1}}=S_j^{k_v^{-1}}$. When $o(j)=2$, let $t_1$ be the number of inverse-closed subsets $S_j$ of $\mathcal{O}_j$ with $S_j^{k_u^{-1}}=S_j^{k_v^{-1}}$. In both cases, let $$t_2=2^{\mathbf{c}(\gamma_jN\cup\gamma_j^{-1}N)}-t_1.$$

Observe that for every subset $S\subseteq R$ with $S_j^{k_u^{-1}}=S_j^{k_v^{-1}}$, we have $S\in \Psi(\{u,v\},j)$ because $S_j^{k_u^{-1}\gamma_i^{-1}}=S_j^{k_v^{-1}\gamma_i^{-1}}$ and hence $|S_{ji^{-1}}\cap S_j^{k_u^{-1}\gamma_i^{-1}}|=|S_{ji^{-1}}\cap S_j^{k_v^{-1}\gamma_i^{-1}}|$. (In other words, when $S_j^{k_u^{-1}}=S_j^{k_v^{-1}}$, we have no constraint on $S_{ji^{-1}}$.) If  $S_j^{k_u^{-1}}=S_j^{k_v^{-1}}$, then $S_j=S_j^{k_v^{-1}k_u}$ and hence $S_j$ is a union of $\langle k_v^{-1}k_u\rangle$-orbits. As $N$ acts regularly on $\mathcal{O}_j$, we have 
\begin{equation}\label{eqx}t_1\leq 2^{\frac{|N|}{o(k_v^{-1}k_u)}}.
\end{equation} 

Next let $S\in \Psi(\{u,v\},j)$  and suppose $S_j$ is a subset of $\mathcal{O}_j$ with $S_j^{k_u^{-1}}\neq S_j^{k_v^{-1}}$. Here to estimate the number of inverse-closed subsets $S$ of $R$ with $|S_{ji^{-1}}\cap S_j^{k_u^{-1}\gamma_i^{-1}}|=|S_{ji^{-1}}\cap S_j^{k_v^{-1}\gamma_i^{-1}}|$, we estimate the number of subsets satisfying the weaker (but easier to handle) condition $$|S_{ji^{-1}}\cap S_j^{k_u^{-1}\gamma_i^{-1}}|\equiv |S_{ji^{-1}}\cap S_j^{k_v^{-1}\gamma_i^{-1}}|\mod 2.$$ Now $S_j^{k_u^{-1}\gamma_i^{-1}}$ and $S_j^{k_v^{-1}\gamma_i^{-1}}$ are two distinct subsets of $\mathcal{O}_{ji^{-1}}$ of the same size $a$, say. Let $b$ be the size of $S_j^{k_u^{-1}\gamma_i^{-1}}\cap S_j^{k_v^{-1}\gamma_i^{-1}}$. Observe that $a-b>0$ because $S_j^{k_u^{-1}}\ne S_j^{k_v^{-1}}$. A subset $S_{ji^{-1}}$ of $\mathcal{O}_{ji^{-1}}$ with $|S_{ji^{-1}}\cap S_j^{k_u^{-1}\gamma_i^{-1}}|\equiv |S_{ji^{-1}}\cap S_j^{k_v^{-1}\gamma_i^{-1}}|\mod 2$ can be written as $X\cup Y$, where $X$ is an arbitrary subset of $\mathcal{O}_{ji^{-1}}\setminus (S_j^{k_v^{-1}\gamma_i^{-1}}\setminus S_j^{k_u^{-1}\gamma_i^{-1}})$ and $Y$ is a subset of $S_j^{k_v^{-1}\gamma_i^{-1}}\setminus S_j^{k_u^{-1}\gamma_i^{-1}}$ of size having parity uniquely determined by the parity of $|X|$. Therefore we have $2^{|N|-(a-b)}2^{(a-b)-1}=2^{|N|-1}$ choices for $S_{ji^{-1}}$. Altogether we have 
\begin{eqnarray*}
A&\le & t_1\cdot 2^{|N|}+t_2\cdot 2^{|N|-1}=t_12^{|N|}+(2^{\mathbf{c}(\gamma_jN\cup \gamma_j^{-1}N)}-t_1)2^{|N|-1}=2^{|N|+\mathbf{c}(\gamma_jN\cup\gamma_j^{-1}N)-1}+t_12^{|N|-1}
\end{eqnarray*}
As $o(ji^{-1})>2$, from~\eqref{eqcases1}, we have $|N|+\mathbf{c}(\gamma_jN\cup\gamma_j^{-1}N)=\mathbf{c}(R_{j,i})$ and hence, from~\eqref{eqx} (noting that if $o(j)>2$ then $\mathbf{c}(\gamma_jN\cup\gamma_j^{-1}N)=|N|$, and otherwise $\gamma_jN\cup\gamma_j^{-1}N=\gamma_jN$), we get
\begin{eqnarray}\label{hopefullylaststar}
A&\le &2^{\mathbf{c}(R_{j,i})-1}+t_12^{|N|-1}\leq 2^{\mathbf{c}(R_{j,i})-1}+2^{|N|+\frac{|N|}{o(k_v^{-1}k_u)}-1}\\\nonumber
&=&2^{\mathbf{c}(R_{j,i})}\left(\frac{1}{2}+\frac{1}{2^{1+\mathbf{c}(R_{j,i})-|N|-\frac{|N|}{o(k_v^{-1}k_u)}}}\right)=
2^{\mathbf{c}(R_{j,i})}\left(\frac{1}{2}+\frac{1}{2^{1+\mathbf{c}(\gamma_jN\cup\gamma_j^{-1}N)-\frac{|N|}{o(k_v^{-1}k_u)}}}\right)
.
\end{eqnarray}
When $\mathbf{c}(\gamma_jN\cup\gamma_j^{-1}N)> |N|/o(k_v^{-1}k_u)$,~\eqref{hopefullylaststar} yields $$A\le 2^{\mathbf{c}(R_{j,i})}\cdot\left(\frac{1}{2}+\frac{1}{2^2}\right)=\frac{3}{4}\cdot 2^{\mathbf{c}(R_{j,i})}$$
and hence~\eqref{propo:eq1} holds in this case. Assume $\mathbf{c}(\gamma_jN\cup\gamma_j^{-1}N)\le |N|/o(k_v^{-1}k_u)$, that is, $$
\frac{|N|}{o(k_v^{-1}k_u)}\ge\begin{cases}
\frac{|N\gamma_j|+|N\gamma_j\cap I(R)|}{2}&\textrm{when }o(j)=2,\\
|N|&\textrm{when }o(j)>2.
\end{cases}$$ As $k_v^{-1}k_u\ne 1$, we have $o(k_v^{-1}k_u)\ge 2$ and hence $o(j)=2$. Thus
$$\frac{|N|}{o(k_v^{-1}k_u)}\ge \frac{|N\gamma_j|+|N\gamma_j\cap I(R)|}{2}.$$
Since the left-hand side is at most $|N|/2$ and since the right-hand side is at least $|N|/2$, this implies $o(k_v^{-1}k_u)= 2$ and 
$$0\ge\frac{|N\gamma_j\cap I(R)|}{2}.$$
Therefore $N\gamma_j\cap I(R)=\emptyset$, $N\gamma_j$ contains no involutions and $\mathbf{c}(\gamma_jN)=|N|/2$.
Under these strong conditions we refine the upper bound in~\eqref{hopefullylaststar} by first improving our upper bound in~\eqref{eqx}. 

As $o(j)=2$, $N\gamma_j$ is inverse-closed. Recall that $t_1$ is the number of inverse-closed subsets $S_j\subseteq N\gamma_j$ with $S_j^{k_v^{-1}k_u}=S_j$. Consider the permutation $\iota:\gamma_j N\to \gamma_j N$ defined by mapping $$\gamma_jy\mapsto (\gamma_j y)^{-1}=y^{-1}\gamma_j^{-1},$$ for each $y\in N$, and consider the permutation $\delta:\gamma_j N\to \gamma_j N$ defined by mapping $$\gamma_jy\mapsto \gamma_j y k_v^{-1}k_u,$$ for each $y\in N$. Observe that $\iota$ and $\delta$ are involutions with no fixed points: 
$\iota$ has no fixed points because $\gamma_j N$ contains no involutions and $\delta$ is an involution because $o(k_v^{-1}k_u)=2$. In this new setting, $$t_1=2^o,$$ where $o$ is the number of orbits of $\langle \iota,\delta\rangle\le\mathrm{Sym}(\gamma_jN)$. Each orbit of $\langle \iota,\delta\rangle$ has even length, because $\iota$ has order $2$ and has no fixed points. Suppose $\langle\iota,\delta\rangle$ has at least one orbit of length greater then $2$. Then $o\le |N|/2-1$ (the upper bound is achieved when $\langle\iota,\delta\rangle$ has $|N|/2-2$ orbits of length $2$ and one of length $4$). Thus, in this case, $$t_1\le 2^{\frac{|N|}{2}-1}.$$ Using this slight improvement on $x$ and $\mathbf{c}(\gamma_j N)=|N|/2$, we obtain 
\begin{eqnarray*}
A&\le & t_1\cdot 2^{|N|}+t_2\cdot 2^{|N|-1}=t_12^{|N|}+(2^{\frac{|N|}{2}}-t_1)2^{|N|-1}=2^{\frac{3|N|}{2}-1}+t_12^{|N|-1}\\
&\leq&2^{\frac{3|N|}{2}-1}+2^{\frac{3|N|}{2}-2}=\frac{3}{4}\cdot 2^{\frac{3|N|}{2}}.
\end{eqnarray*}
As $\mathbf{c}(R_{j,i})=|N|+\mathbf{c}(\gamma_jN)=3|N|/2$ (see~\eqref{eqcases1}), we obtain
\begin{eqnarray}\label{hopefullylaststar1}
A&\leq& \frac{3}{4}\cdot 2^{\mathbf{c}(R_{j,i})}.
\end{eqnarray}
In particular, from~\eqref{eqcases} and~\eqref{hopefullylaststar1}, we see that~\eqref{propo:eq1} holds. 

It remains to suppose that each orbit of $\langle \iota,\delta\rangle$ has length $2$; this means  $\iota=\delta$, that is, $$(\gamma_j y)^\iota=(\gamma_j y)^\delta,\quad \forall y\in N.$$ In other words, $y^{-1}\gamma_j^{-1}=\gamma_j y k_v^{-1}k_u$, for every $y\in N$. Set $z:=k_v^{-1}k_u$. Applying this equality with $y=1$, we get $\gamma_j^{-1}=\gamma_j z$ and hence $\gamma_j^2=z$ because $z$ has order $2$. Thus we have $y^{-1}\gamma_j^{-1}=\gamma_j y\gamma_j^{-2}$ and hence $\gamma_j y\gamma_j^{-1}=y^{-1}$. This shows that the element $\gamma_j$ acts by conjugation on $N$ inverting each of its elements. Therefore, $N$ is abelian. 

To complete this case, we need to show that $o(i)$ is even. Observe that since $o(j)=2$ we have $j=(i^{-1})(ij)=((i^{-1})(ij))^{-1}=(ij)^{-1}i$. Therefore, $i^2j=(i)(ij)=(ij)^{-1}i^{-1}=ji^{-2}$ has order 2. Since $o(ij)=o(ji^{-1})>2$, we cannot have $i \in \langle i^2 \rangle$, so $o(i)$ must be even.  In particular,~\eqref{propo:eq3} holds.~$_\blacksquare$

\smallskip

\noindent\textsc{Case }$o(ji^{-1})=2$ and $o(j)>2$.

\smallskip

\noindent This case can be reduced to the case above. Set $u':=v_0^{g_u^{-1}}$ and observe that $g_u^{-1}=k_u^{-1}\gamma_i^{-1}$ and hence $u'\in \mathcal{O}_{i^{-1}}$. From~\eqref{eq6-----}, we have $$|\sigma(S,u,j)|=|S_j\cap S_{ji^{-1}}^{g_u}|=|S_j^{g_u^{-1}}\cap S_{ji^{-1}}|=|S_{ji^{-1}}\cap S_j^{g_u^{-1}}|=|\sigma(S,u',ji^{-1})|.$$
Similarly, $|\sigma(S,v,j)|=|\sigma(S,v',ji^{-1})|$, where $v':=v_0^{g_v^{-1}}$. In particular, $|\sigma(S,u,j)|=|\sigma(S,v,j)|$ if and only if $|\sigma(S,u',ji^{-1})|=|\sigma(S,v',ji^{-1})|$. Thus $|\Psi(\{u,v\},j)|=|\Psi(\{u',v'\},ji^{-1})|$. As $o(j)>2$ and $o(ji^{-1})=2$, this case follows by applying the previous case to $\Psi(\{u',v'\},ji^{-1})$. We obtain that either~\eqref{propo:eq1} or~\eqref{propo:eq3dash} holds. 

\smallskip

\noindent\textsc{Case }$o(ji^{-1})=o(j)=2$. This is the only remaining option.

\end{proof}



For three distinct vertices $u,v,w\in \mathcal{O}_i$ and $j\in \{1,\ldots,b\}$, let $$\Psi(\{u,v,w\},j):=\{S\subseteq R\mid S=S^{-1} \textrm{ and }|\sigma(S,u,j)|= |\sigma(S,v,j)|=|\sigma(S,w,j)|\}.$$

\begin{proposition}\label{ioddprop}
Let $i\in \{2,\ldots,b\}$, let $u, v$, and possibly $w$ be distinct vertices in $\mathcal{O}_i$ and let $j\in \{1,\ldots,b\}\setminus\{1,i\}$. Then unless $o(j)=o(ji^{-1})=2$, we can conclude that:
\begin{itemize}
\item if $o(i)$ is odd, then $|\Psi(\{u,v\},j)|\le \frac{3}{4}\cdot 2^{\mathbf{c}(R)}$ or $j^2=i$; and
\item if $w$ exists, then $|\Psi(\{u,v,w\},j)|\le \frac{3}{4}\cdot 2^{\mathbf{c}(R)}$.
\end{itemize}
\end{proposition}

\begin{proof}
Assume that we do not have $o(j)=o(ji^{-1})=2$.

We apply Proposition~\ref{propo:1} to $\{u,v\}$. If $o(i)$ is odd, we see immediately that Proposition~\ref{propo:1} parts~\eqref{propo:eq3}, \eqref{propo:eq3dash}, and~\eqref{propo:eq4} cannot arise. Parts~\eqref{propo:eq1} and~\eqref{propo:eq2} are the conclusions we desire.

We also apply Proposition~\ref{propo:1} for the pairs $\{v,w\}$ and $\{w,u\}$. If Proposition~\ref{propo:1} part~\eqref{propo:eq1} holds for one (or more) of the three pairs, then the result immediately follows. Therefore, we suppose that none of the pairs $\{v,w\}$, $\{v,u\}$ and $\{w,u\}$ satisfies Proposition~\ref{propo:1} part~\eqref{propo:eq1}.

Assume that there exists a pair satisfying Proposition~\ref{propo:1} part~\eqref{propo:eq2}. Then $j^2=i$. It follows that $o(j)>2$ and $o(ji^{-1})>2$. In particular, each pair satisfies Proposition~\ref{propo:1} part~\eqref{propo:eq2}. However, by applying Proposition~\ref{propo:1} part~\eqref{propo:eq2} to the pairs $\{u,v\}$ and $\{w,v\}$, we get
$$k_u=\bar{y}^{-1}\gamma_j^{-1}\bar{y}k_v\gamma_j=k_w,$$
contradicting the fact that $u\ne w$. Therefore, none of the pairs $\{v,w\}$, $\{v,u\}$ and $\{w,u\}$ satisfies Proposition~\ref{propo:1} part~\eqref{propo:eq2}.

Now, it is readily seen that, if one of the pairs satisfies Proposition~\ref{propo:1} part~\eqref{propo:eq3} (respectively, part~\eqref{propo:eq3dash}), then all pairs satisfy Proposition~\ref{propo:1} part~\eqref{propo:eq3} (respectively, part~\eqref{propo:eq3dash}). In particular, we deduce
$$k_v^{-1}k_w=\gamma_j^2=k_v^{-1}k_u,$$
contradicting the fact that $u\ne w$. (The argument when the pairs satisfy Proposition~\ref{propo:1} part~\eqref{propo:eq3dash} is similar.)
\end{proof}

For two distinct vertices $u,v\in \mathcal{O}_i$, let $$\Psi(\{u,v\}):=\bigcap_{j\in \{1,\ldots,b\}\setminus\{1,i\}}\Psi(\{u,v\},j).$$

Similarly, for three distinct vertices $u,v,w\in \mathcal{O}_i$ and $j\in \{1,\ldots,b\}\setminus\{1,i\}$, let $$\Psi(\{u,v,w\}):=\bigcap_{j\in \{1,\ldots,b\}\setminus\{1,i\}}\Psi(\{u,v,w\},j).$$

Our next result further refines these possibilities.

\begin{proposition}\label{ioddboundprop}
Let $i\in \{2,\ldots,b\}$, and let $u, v,$ and possibly $w$ be distinct vertices in $\mathcal{O}_i$. 
\begin{itemize}
\item If $o(i)$ is odd, then $|\Psi(\{u,v\})|\le 2^{\mathbf{c}(R)-0.02\cdot\frac{|R|}{|N|}}$.
\item If $w$ exists and $R/N$ is not an elementary abelian $2$-group, then $|\Psi(\{u,v,w\})|\le 2^{\mathbf{c}(R)-0.02\cdot\frac{|R|}{|N|}}$.
\end{itemize}
\end{proposition}
\begin{proof}
If $o(i)$ is odd, then $R/N$ is not an elementary abelian $2$-group, so we may assume this throughout the proof. 

We define an auxiliary  graph $X$: the vertex-set of $X$ is $\{\{j,j^{-1}\}\mid j\in R/N\}$ and the vertex $\{j,j^{-1}\}$ is declared to be adjacent to $$\{ji^{-1},ij^{-1}\},\,\{ij,j^{-1}i^{-1}\},\, \{j^{-1}i,i^{-1}j\}\, \textrm{ and }\,\{ji,i^{-1}j^{-1}\}.$$ In particular, $X$ is a graph with $\mathbf{c}(R/N)$ vertices and where each vertex has valency at most $4$. Observe that some vertex $\{j,j^{-1}\}$ might have valency less than four, because the elements $\{ji^{-1},ij^{-1}\}$, $\{ij,j^{-1}i^{-1}\}$, $\{j^{-1}i,i^{-1}j\}$ and $\{ji,i^{-1}j^{-1}\}$ are not necessarily distinct. Moreover, some vertex $\{j,j^{-1}\}$ might have a loop: indeed, it is easy to check that $\{j,j^{-1}\}$ has a loop if and only if $j^2\in \{i,i^{-1}\}$. 

Let $Y$ be the subgraph induced by $X$ on $R/N\setminus I(R/N)$. Since $R/N$ is not an elementary abelian $2$-group, by a result of Miller~\cite{Miller}, we get $|R\setminus I(R/N)|\ge |R/N|/4$. Now, a classical graph theoretic result of Caro-Tur\'{a}n-Wei~\cite{11, 40, 41} yields that $Y$ has an independent set, $\mathcal{I}$ say, of
cardinality at least 
$$
\sum_{
\substack{\{j,j^{-1}\}\\ o(j)>2}}\frac{1}{\mathrm{deg}_X(\{j,j^{-1}\})+1}\ge \frac{|R/N|/4}{5}=\frac{|R|}{20|N|}.$$

Thus $\mathcal{I}=\{\{j_1,j_1^{-1}\},\ldots,\{j_\ell,j_\ell^{-1}\}\}$, for some $\ell\ge |R|/20|N|$.
The independence of $\mathcal{I}$ yields that, for every two distinct vertices $\{j_u,j_u^{-1}\}$ and $\{j_v,j_v^{-1}\}$ in $\mathcal{I}$, the neighbourhood of $\{j_u,j_u^{-1}\}$ and $\{j_v,j_v^{-1}\}$ are disjoint. 
Therefore,~\eqref{eq6-----}  yields that the events $\Psi(\{u,v\},j)$ and $\Psi(\{u,v\},j')$ are independent, and likewise (if $w$ exists) that the events $\Psi(\{u,v,w\},j)$ and $\Psi(\{u,v,w\},j')$ are independent. 

Furthermore, if $o(i)$ is odd and one of these $\ell$ vertices corresponds to the unique $j$ with $j^2=i$ then the same vertex corresponds to $j^{-1}$, and $(j^{-1})^2 =i^{-1} \neq i$ since $o(i)$ is odd, so we may choose the event $\Psi(\{u,v\},j^{-1})$ instead of $\Psi(\{u,v\},j)$, avoiding the possibility that part~\eqref{propo:eq2} of Proposition~\ref{propo:1} arises. 

Thus, it follows  from Proposition~\ref{ioddprop} for either $\Psi=\Psi(\{u,v\})$ or $\Psi=\Psi(\{u,v,w\})$ as appropriate, that $$\Psi \le \left(\frac{3}{4}\right)^{\ell}\cdot 2^{\mathbf{c}(R)}\le \left(\frac{3}{4}\right)^{\frac{|R|}{20|N|}}\cdot 2^{\mathbf{c}(R)}=2^{\mathbf{c}(R)-\log_2(4/3)(\frac{|R|}{20|N|})}<2^{\mathbf{c}(R)-0.02\cdot \frac{|R|}{|N|}}.\qedhere$$
\end{proof}

We now use the bounds we have achieved, to show that the number of graphs admitting automorphisms that fix every orbit $\mathcal O_k$ setwise, but act nontrivially on some $\mathcal O_i$ is a vanishingly small fraction of the $2^{\mathbf c(R)}$ Cayley graphs on $R$, as long as either $o(i)$ is odd, or the orbit on $\mathcal O_i$ has length at least $3$. Actually, these formulas only produce results that are vanishingly small if $|N|$ is small enough relative to $|R|$ that $|R|/|N|$ grows with $|R|$, so this is the point at which it starts to become clear that we need to be assuming that $|N|$ is relatively small, in order to apply the results in this section. The result involving an orbit of length $3$ does not work in the case that $R/N$ is an elementary abelian $2$-group; this case will need to be handled separately.

\begin{lemma}\label{iodd-lemma}
Let 
\begin{eqnarray*}\mathcal{S}:=\{S\subseteq R&\mid& S=S^{-1}, \textrm{ there exists }i\in \{2,\ldots,b\} \textrm{ with $o(i)$ odd such that }\\
&&(F_S)_{v_0} \textrm{ has a nontrivial orbit on } \mathcal{O}_i\}.\end{eqnarray*}
Furthermore, if $R/N$ is not elementary abelian $2$-group, let
\begin{eqnarray*}
\mathcal{S}':=\{S\subseteq R&\mid& S=S^{-1}, \textrm{ there exists }i\in \{2,\ldots,b\} \textrm{ such that }\\
&&(F_S)_{v_0} \textrm{ has an orbit of cardinality at least }3\textrm{ on } \mathcal{O}_i\}.
\end{eqnarray*}
Then $|\mathcal{S}|\le 2^{\mathbf{c}(R)-0.02\frac{|R|}{|N|}+\log_2(|R||N|/2)}$ and $|\mathcal{S}'|\le 2^{\mathbf{c}(R)-0.02 \frac{|R|}{|N|}+\log_2(|R||N|^2/6)}$.
\end{lemma}

\begin{proof}
For each $i\in \{2,\ldots,b\}$ with $o(i)$ odd, let $\mathcal{S}_i$ be the subset of $\mathcal{S}$ defined by
$$\mathcal{S}_i:=\{S\subseteq R\mid S=S^{-1}, (F_S)_{v_0}\textrm{ has a nontrivial orbit on } \mathcal{O}_i\}.$$ If $o(i)$ is even then define $\mathcal{S}_i =\emptyset$.
Clearly, $\mathcal{S}=\bigcup_{i=2}^b\mathcal{S}_i$. 

Similarly, for each $i\in \{2,\ldots,b\}$, let $\mathcal{S}'_i$ be the subset of $\mathcal{S}'$ defined by
$$\mathcal{S}'_i:=\{S\subseteq R\mid S=S^{-1}, (F_S)_{v_0}\textrm{ has an orbit of cardinality at least 3 on } \mathcal{O}_i\}.$$
Clearly, $\mathcal{S}'=\bigcup_{i=2}^b\mathcal{S}'_i$.

Let $i\in \{2,\ldots,b\}$, let $S\in \mathcal{S}_{i}$ with $o(i)$ odd, or $S \in \mathcal{S}_i'$ (as appropriate) and let $u, v,$ and possibly $w$ be distinct vertices of $\mathcal{O}_i$ in the same $(F_S)_{v_0}$-orbit. In particular, there exists $f\in (F_S)_{v_0}$ with $u=v^{f}$, and if $w$ exists then there exists $f' \in (F_S)_{v_0}$ with $u^{f'}=w$. Since $f$ (and $f'$ if it exists) is an automorphism of $\Cay(R,S)$ fixing each $N$-orbit setwise, we deduce 
\begin{eqnarray*}
\sigma(S,v,j)^f&=&\sigma(S,v^f,j)=\sigma(S,u,j), \text{ and if $w$ exists then}\\
\sigma(S,v,j)^{f'}&=&\sigma(S,v^{f'},j)=\sigma(S,w,j),
\end{eqnarray*}
for every $j\in \{1,\ldots,b\}\setminus\{1,i\}$. Hence, $|\sigma(S,u,j)|=|\sigma(S,v,j)|(=|\sigma(S,w,j)|)$ and $S\in \Psi(\{u,v\},j)$ or $\Psi(\{u,v,w\},j)$. Since this holds for each $j\in \{1,\ldots,b\}\setminus\{1,i\}$, we get $S\in \Psi(\{u,v\})$ or $S \in \Psi(\{u,v,w\})$. 

The argument in the previous paragraph shows that
$$\mathcal{S}_i\subseteq \bigcup_{\substack{\{u,v\}\subseteq \mathcal{O}_i\\u \neq v}}\Psi(\{u,v\}) \text{ or } \mathcal{S}_i\subseteq \bigcup_{\substack{\{u,v,w\}\subseteq \mathcal{O}_i\\|\{u,v,w\}|=3}}\Psi(\{u,v,w\}).$$

From Proposition~\ref{ioddboundprop}, we deduce that
$$|\mathcal{S}|\le (b-1){|N|\choose 2}2^{\mathbf{c}(R)-0.02\cdot \frac{|R|}{|N|}}\le \frac{|R|}{|N|}\frac{|N|^2}{2}2^{\mathbf{c}(R)-0.02\cdot \frac{|R|}{|N|}}$$
and 
$$|\mathcal{S}'|\le (b-1){|N|\choose 3}2^{\mathbf{c}(R)-0.02\cdot \frac{|R|}{|N|}}\le \frac{|R|}{|N|}\frac{|N|^3}{6}2^{\mathbf{c}(R)-0.02\cdot \frac{|R|}{|N|}}.\qedhere$$
\end{proof}

Our next result deals specifically with the case that $R/N$ is an elementary abelian $2$-group. (We refer to Section~\ref{notation specific} for the definition of $B_S$.)

\begin{lemma}\label{case3} (Recall the notation in Section~$\ref{notation specific}$.) Suppose $R$ is not an abelian group of exponent greater than $2$, that
	$R$ is not a generalized dicyclic group and that $R/N$ is an elementary abelian $2$-group. Then
\begin{eqnarray*}|\{S\subseteq R\mid S=S^{-1}, (B_S)_{v_0}\ne 1\}|\le 2^{\mathbf{c}(R)-\frac{|R|}{192}+(\log_2|R|)^2+2}.
\end{eqnarray*}
\end{lemma}
\begin{proof}
Let $\mathcal{S}:=\{S\subseteq R\mid S=S^{-1}, (B_S)_{v_0}\ne 1\}$.	Observe that the definition of $B_S$ immediately yields $B_S\unlhd \Aut(\Cay(R,S))$. In particular, $RB_S$ is a group of automorphisms of $\Cay(R,S)$ acting transitively on the vertex set $R$ and normalizing $N$. Since $R$ is also transitive on the vertex set, the Frattini argument gives $RB_S=R(B_S)_{v_0}$.
	
	Let $$\mathcal{S}':=\{S\in \mathcal{S}\mid R<\norm{RB_S}R\}\quad \textrm{and}\quad \mathcal{S}'':=\mathcal{S}\setminus\mathcal{S}'.$$
	
	Since $R$ is not an abelian group of exponent greater than $2$ and since $R$ is not a generalized dicyclic group, Proposition~\ref{propo:aut} yields
	$$|\{S\subseteq R\mid S=S^{-1}, R<\norm {\Aut(\Cay(R,S))}R\}|\le 2^{\mathbf{c}(R)-\frac{|R|}{96}+(\log_2|R|)^2}.$$
	In particular, $|\mathcal{S}'|\le 2^{\mathbf{c}(R)-\frac{|R|}{96}+(\log_2|R|)^2}$.

	For each $S\in\mathcal{S}''$, choose $G_S$ a subgroup of $RB_S$ with $R<G_S$ and with $R$ maximal in $G_S$. Observe that $\norm {RB_S/N}{R/N}=R/N$, because $\norm {RB_S}R=R$.


	Let $K$ be the core of $R$ in $G_S$. Then
	$$K=\bigcap_{g\in G_S}R^g\ge \bigcap_{g\in G_S}N^g=N.$$
	Since $R$ is maximal in $G_S$, $G_S/K$ acts primitively and faithfully on the set of right cosets of $R$ in $G_S$. The stabilizer of a point in this action is $R/K$. As $N\le K$, we deduce that $R/K$ is an elementary abelian $2$-group. From~\cite[Lemma~$2.1$]{MSV}, we deduce $|G_S:R|=|(G_S)_{v_0}|$ is a prime odd number and $|R:K|=2$.
	
	We now partition the set $\mathcal{S}'$ further. We define
	\begin{align*}
	\mathcal{C}&:=\{S\in\mathcal{S}''\mid (G_S)_{v_0} \textrm{ does not act trivially by conjugation on }K\},\\
	\mathcal{C}'&:=\mathcal{S}''\setminus\mathcal{C}=\{S\in\mathcal{S''}\mid (G_S)_{v_0}\le \cent {G_S}K\}.
	\end{align*}
	In what follows, we obtain an upper bound on the cardinality of $\mathcal{C}$ and $\mathcal{C}'$.
	
	For each $S\in\mathcal{C}$, let $\pi_S:(G_S)_{v_0}\to \Aut(K)$ the natural homomorphism given by the conjugation action of $(G_S)_{v_0}$ on $K$. 
	For each $\varphi\in \Aut(K)\setminus\{id_K\}$, let $\mathcal{C}_\varphi:=\{S\in\mathcal{C}\mid \varphi\in \pi_S((G_S)_{v_0})\}$. In other words, $\mathcal{C}_\varphi$ consists of the connection sets $S$ such that $(G_S)_{v_0}$ contains an element acting by conjugation on $K$ as the automorphism $\varphi$. With this new setting, 
	$$\mathcal{C}\subseteq \bigcup_{\varphi\in \Aut(K)\setminus\{id_K\}}\mathcal{C}_\varphi.$$
	Since $|(G_S)_{v_0}|$ is odd, then $\varphi\in \pi_{S}((G_S)_{v_0})$ has odd order. Using this and applying Theorem~\ref{l:aut} to the group $K$, we deduce that  
	$$|\{S\cap K\mid S\in \mathcal{C}_\varphi\}|\le 2^{\mathbf{c}(K)-\frac{|K|}{96}},$$
	for every $\varphi\in\Aut(K)\setminus\{id_K\}$. In particular, as $|K|=|R|/2$, we have
	$$|\mathcal{C}_\varphi|\le 2^{\mathbf{c}(K)-\frac{|K|}{96}}\cdot 2^{\mathbf{c}(R\setminus K)}=2^{\frac{|K|+|I(K)|}{2}-\frac{|R|}{192}+\frac{|R\setminus K|+|I(R\setminus K)|}{2}}\le 2^{\frac{|R|+|I(R)|}{2}-\frac{|R|}{192}}=2^{\mathbf{c}(R)-\frac{|R|}{192}}.$$
	Since $|\Aut(K)|\le 2^{(\log_2|K|)^2}$, we deduce
	$$|\mathcal{C}|\le 2^{\mathbf{c}(R)-\frac{|R|}{192}+(\log_2|R|)^2}.$$
	
	Let $S\in\mathcal{C}'$ and let $\eta_{S}$ be a generator of $(G_S)_{v_0}$: recall that $(G_S)_{v_0}$ is a cyclic group of order $p_S$, where $p_S$ is an odd prime number. Suppose that $\eta_{S}$ fixes some vertex $x\in R\setminus K$. Then 
	$x^{\eta_{S}}=x$, that is, $v_0^{x\eta_{S}}=v_0^{x}$. This yields $x\eta_{S} x^{-1}\in (G_S)_{v_0}$ and $x\in \norm {G_S}{(G_S)_{v_0}}$. Since $(G_S)_{v_0}$ centralizes $K$, we get $\langle K,x,(G_S)_{v_0}\rangle\le \norm{G_S}{(G_S)_{v_0}}$. As $G_S=\langle K,x,(G_S)_{v_0}\rangle$, we deduce $(G_S)_{v_0}\unlhd G_S$, which is a contradiction because $(G_S)_{v_0}$ is core-free in $G_S$. Therefore, $\eta_{S}$ fixes no vertex in $R\setminus K$. Fix $x\in R\setminus K$. Then $x^{\eta_{S}}=xk$, for some $k\in K\setminus\{1\}$. Observe that, for each $k'\in K$, the image of $xk'$ under $\eta_{S}$ is uniquely determined because
	$$(xk')^{\eta_{S}}=x^{k'\eta_{S}}=x^{\eta_{S} k'}=(x^{\eta_{S}})^{k'}=(xk)^k=xkk'.$$
	Applying this equality with $k'=k$, we deduce $o(k)=p_S$ and hence $k\in N$, because $R/N$ is an elementary abelian $2$-group.
	This shows that the mapping $\eta_{S}$ is uniquely determined by the image of one fixed element $x\in R\setminus K$, which has to be of the form $xk$ for some $k\in N$. Thus we have at most $|N|$ choices for $\eta_{S}$. Once that $\eta_{S}$ is fixed, we have at most $2^{|R|/2p_S}\le 2^{|R|/6}$ choices for an $\eta_{S}$-invariant subset of $R\setminus K$. We deduce
	$$|\mathcal{C}'|\le 2^{\mathbf{c}(K)}\cdot |N|\cdot 2^{\frac{|R|}{6}}\le 2^{\mathbf{c}(R)-\frac{|R|}{12}+\log_2|N|}\le 2^{\mathbf{c}(R)-\frac{|R|}{192}+(\log_2|R|)^2+1}.\qedhere$$
\end{proof}

We end this section by pulling together the above results. We are able to show that for all but a small number of connection sets, every connection set $S$ for every group $R$ containing a nontrivial proper normal subgroup $N$ is covered in one of the previous two results. However, we may have to substitute a larger normal subgroup $K>N$ of $R$ for $N$, which may mean that the bound we achieve is not useful. These situations can be covered by the results from Section~\ref{sec:Nlarge}.

\begin{proof}[Proof of Theorem~$\ref{main2}$]
We use the notation established in Section~\ref{notation specific}. Let $$\mathcal{S}:=\{S\subseteq R\mid S=S^{-1},\,  \exists f\in \norm{\Aut(\Cay(R,S))}{N}\textrm{ with }f\ne 1 \textrm{ and }1^f=1, f \textrm{ fixes each }N\textrm{-orbit setwise}\}.$$
Observe that, for every $S\in\mathcal{S}$, we have $(B_S)_{v_0}\ne 1$. We divide the set $\mathcal{S}$ futher:
\begin{align*}
\mathcal{S}_1:=&\{S\in\mathcal{S}\mid&& R<\norm{\Aut(\Cay(R,S))}R\},\\
\mathcal{S}_2:=&\{S\in \mathcal{S}\setminus\mathcal{S}_1\mid&& \exists i\in \{2,\ldots,b\} \textrm{ with $o(i)$ odd such that }(F_S)_{v_0} \textrm{ has a nontrivial orbit on } \mathcal{O}_i\},\\
\mathcal{S}_3:=&\{S\in\mathcal{S}\setminus(\mathcal{S}_1\cup\mathcal{S}_2)\mid&& R/N \textrm{ not an elementary abelian 2-group},\\
&&& \exists i\in \{2,\ldots,b\} \textrm{ such that }(F_S)_{v_0} \textrm{ has an orbit of cardinality at least }3\textrm{ on } \mathcal{O}_i\},\\
\mathcal{S}_4:=&\{S\in\mathcal{S}\setminus(\mathcal{S}_1\cup\mathcal{S}_2\cup\mathcal{S}_3)\mid&& R/N \textrm{ is an elementary abelian 2-group}, (B_S)_{v_0}\ne 1\},\\
\mathcal{S}_5:=&\mathcal{S}\setminus(\mathcal{S}_1\cup\mathcal{S}_2\cup\mathcal{S}_3\cup\mathcal{S}_4).
\end{align*}
From Proposition~\ref{propo:aut}, Lemma~\ref{iodd-lemma} and Lemma~\ref{case3}, we have explicit bounds for $\mathcal{S}_1$, $\mathcal{S}_2$, $\mathcal{S}_3$ and $\mathcal{S}_4$, and hence we may consider only the set $\mathcal{S}_5$. 

Let $S\in \mathcal{S}_5$. Since $S\notin\mathcal{S}_4$, $R/N$ is not an elementary abelian $2$-group. Since $S\notin\mathcal{S}_3$, $(F_S)_{v_0}$ has orbits of cardinality at most $2$, and so does $(B_S)_{v_0}$. Therefore, $(F_S)_{v_0}$ and $(B_S)_{v_0}$ are elementary abelian $2$-groups. 

Now let $L_S=\{ \gamma_j:(F_S)_{v_0} \text{ is trivial on }\mathcal O_j\}$. Notice that $L_S$ is in fact a group. Since $(F_S)_{v_0} $ is nontrivial, then $L_S$ is a proper subgroup of $R$. Since $S\notin\mathcal{S}_2$,  $\gamma_i \in L_S$ for every $i$ with $o(i)$ odd. Therefore $NL_S$ contains all elements of $R$ of odd order. Let $$K:=\bigcap_{g\in RB_S}(NL_S)^g$$ be the core of $NL_S$ in $RB_S$. Since all conjugates of $NL_S$ in $R$ also contain all elements of $R$ of odd order, we deduce that $K$ also contains all elements of $R$ of odd order and hence $R/K$ is a $2$-group. As $(B_S)_{v_0}$ is also a $2$-group, we obtain that $RB_S/K$ is a $2$-group. 
Therefore $\norm {RB_S/K}{R/K}>R/K$. However, this implies that $\norm {RB_S}R>R$, but this contradicts the fact that $S\notin\mathcal{S}_1$. This shows that $\mathcal{S}_5=\emptyset$.
Now, adding the bounds produced for $\mathcal{S}_i$  for each $1 \le i \le 4$, we get the result. Indeed, using the first bound in Lemma~\ref{iodd-lemma} and the fact that $|R|\ge 2|N|\ge 4,$ we get 
$$|\mathcal{S}_2|\le 2^{\mathbf{c}(R)-\frac{|R|}{192|N|}+\log_2|R|+\log_2|N|-1}\le 2^{\mathbf{c}(R)-\frac{|R|}{192|N|}+(\log_2|R|)^2-2}.$$
Further, if $|R|<8,$ then $|R|\ne 7$ (because  $N$ is a nontrivial proper subgroup), that is $|R|\le 6.$ Consequently, $$\log_2(|R||N|^2/6)\le2\log_2|R|-2\le (\log_2|R|)^2-2.$$ If $|R|\ge 8,$ then $$\log_2(|R||N|^2/6)\le \log_2|R|+2\log_2|N|\le 3\log_2 |R|-2\le (\log_2 |R| )^2-2.$$ Using these, and the second bound in Lemma~\ref{iodd-lemma} we get 
$$|\mathcal{S}_3|\le 2^{\mathbf{c}(R)-\frac{|R|}{192|N|}+\log_2(|R||N|^2/6)}\le 2^{\mathbf{c}(R)-\frac{|R|}{192|N|}+(\log_2|R|)^2-2}.$$

This together with Proposition~\ref{propo:aut}, and Lemma~\ref{case3}, yields

\begin{align*}
|\mathcal{S}|\le&  2^{\mathbf{c}(R)- \frac{|R|}{192|N|}+( \log_2 |R|)^2}(1+2^{-2}+2^{-2}+2^{2})\le 2^{\mathbf{c}(R)- \frac{|R|}{192|N|}+( \log_2 |R|)^2+3},
\end{align*}
as required. 

As in the proof of Theorem~\ref{main1}, we do not need to include the bound from Proposition~\ref{propo:aut} if we include the condition $R=\norm{\Aut(\Cay(R,S))}R$. If we omit this condition, then we include this extra piece (which does not affect the overall bound as we have stated it) but must not allow groups that are either abelian of exponent greater than $2$, or generalised dicyclic.
\end{proof}

(2010), 413--425.

\bibitem{DSV}E. Dobson, P. Spiga, G. Verret, Cayley graphs on abelian groups, \textit{Combinatorica }\textbf{36} (2016), 371--393.






\bibitem{God}C.~D. Godsil, GRRs for nonsolvable groups, \textit{Algebraic Methods in Graph Theory,} (Szeged, 1978), 221--239, \textit{Colloq. Math. Soc. J\'{a}nos Bolyai} \textbf{25}, North-Holland, Amsterdam-New York, 1981.

\bibitem{Go2}C.~D.~Godsil, On the full automorphism group of a graph, \textit{Combinatorica} \textbf{1} (1981), 243--256.




\bibitem{Het} D. Hetzel, \"{U}ber regul\"{a}re graphische Darstellung von aufl\"{o}sbaren Gruppen. Technische Universit\"{a}t, Berlin, 1976.

\bibitem{Im1} W. Imrich, Graphen mit transitiver Automorphismengruppen, \textit{Monatsh. Math.} \textbf{73} (1969), 341--347.

\bibitem{Im2} W. Imrich, Graphs with transitive abelian automorphism group, \textit{Combinat. Theory (Proc. Colloq. Balatonf\"{u}red, 1969}, Budapest, 1970, 651--656.

\bibitem{Im3} W. Imrich, On graphs with regular groups, \textit{J. Combinatorial Theory Ser. B.} \textbf{19} (1975), 174--180.



\bibitem{Li}C.~H.~Li, The finite primitive permutation groups containing an abelian regular subgroup, \textit{Proc. London Math. Soc.~(3)} \textbf{87} (2003), 725--747.

\bibitem{LMac}H.~Liebeck, D.~MacHale, Groups with Automorphisms Inverting most Elements, \textit{Math. Z.} \textbf{124} (1972), 51--63.

(1981), 69--81.



\bibitem{Miller}G.~A.~Miller, Groups containing the largest possible number of operators of order two, \textit{Amer.
Math. Monthly }\textbf{12} (1905), 149--151.

\bibitem{MSMS}J.~Morris, P.~Spiga, Asymptotic enumeration of Cayley digraphs, \textit{Israel J. Math.}, to appear.

\bibitem{MSV}J. Morris, P. Spiga, G. Verret, Automorphisms of Cayley graphs on generalised dicyclic groups, \textit{European J. Combin. }\textbf{43} (2015), 68--81.

\bibitem{NW1}L. A. Nowitz, M. Watkins, Graphical regular representations of direct product of groups, \textit{Monatsh. Math. }\textbf{76} (1972), 168--171.

\bibitem{NW2}L. A. Nowitz, M. Watkins, Graphical regular represntations of non-abelian groups, II, \textit{Canad. J. Math. }\textbf{24} (1972), 1009--1018.

\bibitem{NW3}L. A. Nowitz, M. Watkins, Graphical regular representations of non-abelian groups, I, \textit{Canad. J. Math. }\textbf{24} (1972), 993--1008.








\bibitem{spiga11}P.~Spiga, On the equivalence between a conjecture of Babai-Godsil and  a conjecture of Xu concerning the  enumeration of Cayley graphs, submitted.


\bibitem{40}P.~Tur\'an, An extremal problem in graph theory (hungarian), \textit{Mat. Fiz. Lapok} \textbf{48} (1941), 436--452.

\bibitem{Wat} M. E. Watkins, On the action of non-abelian groups on graphs, \textit{J. Combin. Theory} \textbf{11} (1971), 95--104.

\bibitem{41}V.~K.~Wei, A lower bound on the stability number of a simple graph, \textit{Bell Laboratories Technical Memorandum}, 81--11217--9, Murray Hill, NJ, 1981.


\end{document}